\documentclass[12pt,letter]{article}

\usepackage[top=1in, bottom=1in, left=1in, right=1.25in]{geometry}
\usepackage{float}
\usepackage{amsfonts}
\usepackage{amsmath} 
\usepackage{amssymb} 
\usepackage{amsthm} 
\usepackage{mathabx}
\usepackage{mathtools}
\usepackage{bm}
\usepackage{caption}
\usepackage{subcaption}
\usepackage{color}
\usepackage{enumitem}
\usepackage{setspace}  
\usepackage{titlesec}
\usepackage{multirow}
\usepackage{graphicx}
\usepackage{siunitx}
\usepackage{booktabs}
\usepackage{textcomp}
\usepackage{enumitem}

\usepackage[utf8]{inputenc}
\usepackage[english]{babel}
\usepackage[round]{natbib}
\usepackage[pdftex,hypertexnames=false,linktocpage=true]{hyperref}

\hypersetup{colorlinks=true,linkcolor=blue,anchorcolor=blue,citecolor=blue,filecolor=blue,urlcolor=blue,bookmarksnumbered=true,pdfview=FitB}
\titleformat{\paragraph}
{\normalfont\normalsize\bfseries}{\theparagraph}{1em}{}
\titlespacing*{\paragraph}
{0pt}{3.25ex plus 1ex minus .2ex}{1.5ex plus .2ex}

\newtheorem{theorem}{Theorem}[section]
\newtheorem{lemma}{Lemma}[section]

\newtheorem{corollary}{Corollary}[section]
\newtheorem{fact}{Fact}[section]
\newtheorem{claim}{Claim}[section]
\newtheorem*{example}{Example}
\newtheorem*{remark}{Remark}

\newcommand{\wh}{\widehat}
\newcommand{\wb}{\widebar}
\newcommand{\wt}{\widetilde}

\newcommand{\rank}{{\rm rank}}
\newcommand{\sign}{{\rm sign}}
\newcommand{\E}{{\rm E}}
\newcommand{\var}{{\rm Var}}
\newcommand{\cov}{{\rm Cov}}

\newcommand{\argmin}{\mathop{\arg\min}}

\newcommand{\bp}{\bm{p}}
\newcommand{\br}{\bm{r}}
\newcommand{\bv}{\bm{v}}
\newcommand{\bw}{\bm{w}}
\newcommand{\bx}{\bm{x}}
\newcommand{\by}{\bm{y}}
\newcommand{\bz}{\bm{z}}
\newcommand{\bEta}{\bm{\eta}}
\newcommand{\bmu}{\bm{\mu}}

\newcommand{\btheta}{\bm{\theta}}
\newcommand{\bbeta}{\bm{\beta}}

\newcommand{\bB}{\bm{B}}
\newcommand{\bD}{\bm{D}}

\newcommand{\bH}{\bm{H}}
\newcommand{\bI}{\bm{I}}
\newcommand{\bP}{\bm{P}}

\newcommand{\bR}{\bm{R}}
\newcommand{\bT}{\bm{T}}
\newcommand{\bX}{\bm{X}}
\newcommand{\bY}{\bm{Y}}

\newcommand{\bSigma}{\bm{\Sigma}}

\newcommand{\bDelta}{\bm{\Delta}}
\newcommand{\bPsi}{\bm{\Psi}}

\newcommand{\bIn}{\bm{1}_n}

\newcommand{\wbK}{\wb{K}}

\newcommand{\nb}{{\rm ne}}
\newcommand{\cl}{{\rm cl}}

\numberwithin{equation}{section}

\title{Neighborhood selection with application to social networks}

\author{Nana Wang \thanks{nnawang@ucdavis.edu}
\\Department of Statistics 
\\University of California, Davis
\and
Wolfgang Polonik \thanks{wpolonik@ucdavis.edu}
\\Department of Statistics 
\\University of California, Davis
}

\date{}

\begin{document}
\begin{spacing}{1}
\maketitle

\begin{abstract}
	The topic of this paper is modeling and analyzing dependence in stochastic social networks. 
	We propose a latent variable block model that allows the analysis of dependence between
	blocks via the analysis of a latent graphical model. Our approach is based on the 
	idea underlying the neighborhood selection scheme put forward by \cite{meinshausen2006high}. However, because of the latent nature of our model, estimates have to be 
	used in lieu of the unobserved variables. This leads to a novel analysis of graphical 
	models under uncertainty, in the spirit of~\cite{rosenbaum2010sparse}, or Belloni, 
	Rosenbaum and Tsybakov (2016). Lasso-based selectors, and a class of Dantzig-type
	selectors are studied. 
\end{abstract}

\section{Introduction} \label{sec: intro}
The study of random networks has been a topic of great interest in recent years, e.g. see~\cite{kolaczyk2009statistical} and~\cite{newman2010networks}. A network is defined as a structure composed of nodes and edges connecting nodes in various relationships.
The observed network can be represented by an $N\times N$ adjacency matrix $\bY = (Y_{ij})_{i,j = 1,\ldots,N}$, where $N$ is the total number of nodes within the network. For a binary relation network, as considered here, $Y_{ij}=1$ if there is an edge from node $i$ to node $j,$ and $0$ otherwise. In the following we identify an adjacency matrix $\bY$ with the network itself.

Most relational phenomena are dependent phenomena, and dependence is often of substantive interest. \cite{frank1986markov} and  \cite{wasserman1996logit} introduced exponential random graph models which allow the modelling of a wide range of dependences of substantive interest, including transitive closure. For such models, $Y_{ij}\in\{0,1\}$ and the distribution of $\bY$ is assumed to follow the exponential family form 
%
$P_{\btheta}(\bY=\by)=\exp\left(\btheta\cdot \bT(\by)-\phi(\btheta)\right), \by\in\mathcal{Y},$ 
%
where
$\phi(\btheta)=-\log\left(\sum_{\by\in\mathcal{Y}}\exp(\btheta\cdot \bT(\by))\right)$ and $\bT(\by):\mathcal{Y}\rightarrow \mathbb{R}^q,$ are the sufficient statistics, e. g. the total number of edges.
However, as mentioned in \cite{schweinberger2014local}, exponential random graph models are lacking neighborhood structure, and that makes modelling dependencies challenging for such networks. Neighborhoods (or communities, blocks) are in general defined as a group of individuals (nodes), such that individuals within a group interact with each other more frequently than with those outside the group. Recently, \cite{schweinberger2014local} proposed the concept of local dependence in stochastic networks. This concept allows for dependence within neighborhoods, while different neighborhoods are independent.

In contrast to that, our work is considering dependence between blocks, while the connections within blocks are assumed independent. We also assume the blocks to be known. We then propose to analyze dependencies between blocks by means of graphical models. To this end, we assume an undirected network so that 
\begin{align} \label{model: level 2}
	Y_{ij}|(\bP,\bz)\sim {\rm Bernoulli}\left(p_{\bz[i],\bz[j]}\right),
\end{align}
where $\bz[i] \in \wbK := \{1, \cdots, K\}, i = 1,\ldots, N$ indicate block memberships in one of $K$ blocks; $p_{k,\ell}, k,\ell \in \wbK$ govern the intensities of the connectivities within and between blocks, $0 < p_{k,l} < 1$; and $\bP = (p_{k, l})_{k, l\in\wbK}$ is a $K\times K$ symmetric matrix. We then put a Gaussian logistic model on the $p_{k,\ell}$. More precisely, for the diagonal elements $(p_{k,k})_{1\leq k\leq K}$, assume that 
\begin{align} \label{model: level 1}
	\log\left(\frac{p_{k,k}}{1-p_{k,k}}\right)=\bx_{k}^T\bbeta+\epsilon_{k},\; 1\leq k \leq K, 
\end{align}
where $\bx_{k}$ is a $(L\times1)$ vector of given co-variables corresponding to block $k$, and $\bbeta$ is the $(L\times1)$ parameter vector. Furthermore, $\bm{\epsilon} = (\epsilon_1,\cdots,\epsilon_K)^T$ with
\begin{align} \label{model: level1_error}
	\bm{\epsilon}\sim N(\bm{0},\bSigma), 
\end{align}
where $\bSigma=(\sigma_{kl})_{1\leq k,l\leq K}$ is an nonsingular covariance matrix. 
Each off-diagonal element $p_{k,l}$ ($k\neq l$) is assumed to be independent with all the other elements of $\bP$. The latter assumption is made to simplify the exposition. A similar model can be found in \cite{xu2014dynamic}.

The dependence between the $p_{k,k}$ induces dependence between blocks. We can thus analyze this induced dependence in our network model, by using methods from Gaussian graphical models, via selecting the zeros in the precision matrix $\bSigma^{-1}$. Adding dependencies between the $p_{k,\ell}$ with $k \ne \ell$ would increase the dimension of $\bSigma$, and induce `second order dependencies' to the network structure, namely, dependencies of block connections between different pairs of blocks.

It is crucial to observe that this Gaussian graphical model is defined in terms of the $p_{k,k}$ (or, more precisely, in terms of their log-adds ratios), and that these quantities obviously are not observed. Thus, they need to be estimated from our network data, and, to this end, we here assume the availability of iid observations of the network. This estimation, in turn, induces {\em additional randomness} to our analysis of the graphical model. We are therefore facing similar challenges as in the analysis of Gaussian graphical models {\em under uncertainty}. However, our situation is more complex, as will become clear below.

The methods for neighborhood selection considered here, are based on the column-wise methodology of \cite{meinshausen2006high}. We apply this methodology (under uncertainty) to some known selection methods from the literature, thereby, adjusting these methods for the additional uncertainty. The selection methods considered here are (i) the graphical Lasso of \cite{meinshausen2006high}, (ii) a class of Dantzig-type selectors, that includes the Dantzig selector of \cite{candes2007dantzig}, and (iii) the matrix uncertainty selector of \cite{rosenbaum2010sparse}. This will lead to `graphical' versions of the respective procedures. The graphical Dantzig selector already has been studied in \cite{yuan2010high}, but without the additional uncertainty we are facing here. This leads to novel selection methodologies for which we derive statistical guarantees. We also present numerical studies to illustrate their finite sample performance.

More details on our latent variable block model is discussed in Section~\ref{sec: lbm}. Thereby we also introduce some basic notation. Section~\ref{sec: est} introduces our neighborhood selection method-ologies, and presents results on their large sample performance. Tuning parameter selection is also discussed there. Numerical studies are presented in Section~\ref{sec: num}, and the proofs of our main results are in Appendix~\ref{sec: pfs}.. 

\section{Some important preliminary facts} \label{sec: lbm}
Let $\bEta=(\eta_{1},\cdots,\eta_{K})^T$ with $\eta_{k}=\log(p_{kk}/(1-p_{kk}))$ be the vector of log odds of the within-block  connection probabilities, and let $\bX_{K\times L}=(\bx_{1},\cdots,\bx_{K})^T$ be the design matrix. Our latent variable block model \eqref{model: level 2} - \eqref{model: level1_error} says that $\bEta\sim N(\bX\bbeta,\bSigma)$. The dependence among the $\eta_k,$ encoded in $\bSigma,$ is propagated to the $p_{kk}$. Let $\bSigma^{-1} = \bD = (d_{kl})_{1 \le k,l \le K}$, then the following fact holds.
\begin{fact} 
	Under \eqref{model: level 2} - \eqref{model: level1_error}, we have
	$d_{kl}=0,$ if and only if, $p_{k,k}$ is independent of $p_{l,l}$ given the other variables $\bp_{-(k,l)}=\{p_{i,j}: (i, j)\in\wbK\times\wbK\backslash\{(k, k), (l, l)\}, i \leq j\}$, or just given $\{p_{i,i}: i\in\wbK\backslash\{k, l\}\}$. 
\end{fact}
\noindent
In other words, if 
\begin{align*}
	E=\{(k,l): d_{kl}\neq 0, k\neq l\}
\end{align*}
denotes the edge set of the graph corresponding to $\bEta$, then, under our latent variable block model, $(k,l)\notin E$ if and only if $p_{k,k}$  is conditionally independent with $p_{l,l}$ given the other variables $\{p_{ii}:1\leq i \leq K, i\notin\{k,l\}\}$. Identifying nonzero elements in $\bD$ thus will reveal the conditional dependence structure of the blocks in our underlying network. 

We will use the relative number of edges within each block, as estimates for the unobserved values $p_{kk},k=1,\ldots,K$. Let $S_k=\sum_{\bz[i]=\bz[j]=k}Y_{ij},\,k=1,\cdots, K,$ denote the total number of edges in the $K$ blocks. 
\begin{fact} Under \eqref{model: level 2} - \eqref{model: level1_error}, we have
	$$\sign(\sigma_{kl})=\sign\left(\cov(S_k,S_l)\right).$$
\end{fact}
\noindent
For proofs of the two facts see \cite{paulo2012asymptotics} (page 13, Theorem 1.35) and \cite{liu2009nonparanormal} (Section 3, Lemma 2), respectively. 

\section{Neighborhood selection} \label{sec: est}
Here we discuss the identification of the nonzero elements in $\bD$. We first assume that \eqref{model: level 2} - \eqref{model: level1_error} holds with a known $\bbeta$, and we write $\bmu = (\mu_1,\cdots,\mu_K)^T =\bX\bbeta$. We also assume that $0 < p_{i,j} < 1$ for all $i, j \in\wbK$. Let $\bY^{(t)}, t = 1,\ldots,n$ denote $n$ iid observed networks with corresponding independent unobserved random vectors $\bp^{(t)}, t = 1,\ldots,n$ following our model. Let $\mathcal{A}_1,\cdots,\mathcal{A}_K$ denote the $K$ blocks of the networks $\bY^{(t)}$ and $\mathcal{V}=\{1,\cdots,N\}$ be the node set.
Assume $\mathcal{A}_k$ and $\mathcal{A}_l$ are mutually exclusive for $k\neq l$ so that $\bigcup_{k=1}^K\mathcal{A}_k=\mathcal{V}$. The number of possible edges within each block is $m_k=|\mathcal{A}_k|(|\mathcal{A}_k|-1)/2$ for $k=1,\cdots,K$, and the number of possible edges between block $k$ and block $l$ is then $|\mathcal{A}_k||\mathcal{A}_l|$ for $1\leq k\neq l\leq K$. We would like to point out again that the block membership variable $\bz$ is assumed to be known. 

\subsection{Controlling the estimation error}

Given a network $\bY^{(t)}$, let $S^{(t)}_k=\sum_{\bz[i]=\bz[j]=k}Y_{ij}^{(t)}, k = 1,\ldots,K, t = 1,\ldots,n$ denote the number of edges within block $k$ in network $t$. Natural estimates of $p_{kk}^{(t)}$ and $\eta_k^{(t)}$ are 
\begin{align} \label{def: hatPeta}
	\wt{p}_{k,k}^{(t)}=\frac{S_k^{(t)}}{m_k}
	\quad\text{ and}\quad
	\wt{\eta}_k^{(t)}=\log\Bigg(\frac{\wt{p}_{k,k}^{(t)}}{1-\wt{p}_{k,k}^{(t)}}\Bigg)
\end{align}
respectively. 

Let $\wt{\bEta}^{(t)}=(\wt{\eta}_1^{(t)},\cdots,\wt{\eta}_K^{(t)})^T$, 
and let $m_{\min}=\min_{1\leq k\leq K}m_k$ be the minimum number of possible edges within a block, which of course measures the minimum blocksize.
\begin{fact} 
	Assume that $K$ is fixed. Then, under \eqref{model: level 2} - \eqref{model: level1_error}, we have for each $t = 1,\cdots, n$, 
	\begin{align*}
	\wt{\bEta}^{(t)}\rightarrow N(\bX\bbeta, \bSigma)\quad \text{in distribution as }\,m_{\min}\rightarrow\infty.
	\end{align*}
\end{fact}
This result tells us that, if we base our edge selection on $\wt{\bEta}^{(t)}$, then, for $m_{\min}$ large, we are {\em close} to a Gaussian model, and thus we can hope that our analysis is similar to that of a Gaussian graphical model. However, the approximation error has to be examined carefully. In order to do that, we first truncate the $\wt{p}_{kk}^{(t)}$'s, or, equivalently, the $\wt{\eta}_k^{(t)}$. For $T > 0,$ let 
\begin{equation*}
	\wh{\eta}_k^{(t)} = 
	\begin{cases}
		-T & \mbox{if  $\wt{\eta}_k^{(t)} < -T$}\\
		\wt{\eta}_k^{(t)} &\mbox{if  $|\wt{\eta}_k^{(t)}| \leq T$} \\
		T & \mbox{if  $\wt{\eta}_k^{(t)} > T$}. 
	\end{cases}
\end{equation*}
This truncation corresponds to
\begin{equation*}
	\wh{p}_{k,k}^{(t)} = 
	\begin{cases}
		(1+e^T)^{-1} & \mbox{if  $\wt{p}_{k,k}^{(t)} < (1+e^T)^{-1}$ }\\
		\wt{p}_k^{(t)} &\mbox{if  $(1+e^T)^{-1} \leq \wt{p}_{k,k}^{(t)} \leq (1+e^{-T})^{-1}$ } \\
		(1+e^{-T})^{-1} & \mbox{if  $\wt{p}_{k,k}^{(t)} > (1+e^{-T})^{-1}$}. 
	\end{cases}
\end{equation*}
In what follows, we work with these truncated versions. Note that the dependence on $T$ is not indicated explicitly in this notation. 

The magnitude of $m_{\min}$  is important, as it reflects the accuracy of our estimates. This estimation error will crucially enter the performance of the graphical model based inverse covariance estimator. Under the latent variable block model, we have the following concentration result: 
\begin{lemma} \label{lemma: net_etaDiff}
	Let $\sigma^2 =\max_{k\in\wbK}\sigma_{kk}$ and $\mu_B =\max_{k\in\wbK}|\mu_k|$. Then, under \eqref{model: level 2} - \eqref{model: level1_error}, we have, for $\min(L,T) > \mu_B$, and $m_{\min}\ge 16M^2\log (nK)e^{2L}$, that
	\begin{multline} \label{eq: net_etaDiff}
		P\Big(\max_{1 \leq k \leq K \atop 1 \leq t \leq n} |\wh \eta_k^{(t)} - \eta_k^{(t)}| <8Me^{L}\sqrt{\frac{\log (nK) }{m_{\min}}}\Big)  \\
		\geq
		1 -\sqrt{\frac{2}{\pi}}\frac{nK\sigma}{\min\{L, T\}-\mu_B}\exp\Big(-\frac{(\min\{L,T\}-\mu_B)^2}{2\sigma^2}\Big) - \Big(\frac{1}{nK}\Big)^{2M^2 - 1}.
	\end{multline}
\end{lemma}
\begin{remark}
	Note that the larger $\mu_B$, the larger we need to choose both $T$ and $L$. A large $T$ will cause problems, because the $\wh{p}_k^{(t)}$ then might be too close to zero or one, causing challenges by definition of $\wh \eta_k^{(t)}$. A large $L$ makes our approximation less tight. Therefore we will have to control the size of $\mu_B$ (even if $\mu_B$ is known); see assumption A1.6 and B1.5.
	
	To better understand the bound in \eqref{eq: net_etaDiff}, suppose that the number of blocks, $K$, grows with $n$ such that $K(n)=O(n^\gamma),$ for some $\gamma>0.$ While $K$ is allowed to grow with $n$, we assume that $\sigma^2$ is bounded. If we further choose $0 < \min\{L, T\}-\mu_B = \gamma\log n,$ for some $\gamma>0$, then, there exists $c>0$, such that as $n\rightarrow\infty$, 
	\begin{align*}
		\sqrt{\frac{2}{\pi}}\frac{nK\sigma}{\min\{L, T\}-\mu_B}\exp\left(-\frac{(\min\{L,T\}-\mu_B)^2}{2\sigma^2}\right) = O\left(\exp\left(-c(\log n)^2\right)\right). 
	\end{align*}
	The last term on the right-hand side of \eqref{eq: net_etaDiff} can be controlled similarly, by choosing $M = \sqrt{\left(1+(c\log n)/(\gamma+1)\right)/2}$. 
	With these choices, we obtain an approximation error of  $\max_{1 \leq k \leq K, 1 \leq t \leq n} |\wh \eta_k^{(t)} - \eta_k^{(t)}| = O(n^{-p})$ by choosing the minimum blocksize large enough
	\begin{align*}
		m_{\min}^{-1} = O\left(n^{-2p}(\log n)^{-2}e^{-2L}\right). 
	\end{align*}
\end{remark}

\subsection{Edge selection under uncertainty}  \label{subsec: uncertainty}
In order to identify the nonzero elements in $\bD$, we consider the graphical model in terms of the distribution of $\bEta$.  Recall that $\bEta \in {\mathbb R}^K$, where each component of $\bEta$ belongs to one of the $K$ blocks, thus $\wbK=\{1,\cdots,K\}$ are not only the block labels, but also the node set in the underlying graph  corresponding to the joint distribution of the $\bEta$. Using Gaussianity of $\bEta$, the set $\nb_a=\{b\in\wbK: d_{ab}\neq 0\}$  is the neighborhood of node $a\in\wbK$ of the associated graph. We follow the idea of \cite{meinshausen2006high} to convert the problem into a series of linear regression problems: for each $a\in\wbK$, 
\begin{align*} 
	\eta_a - \mu_a = \sum_{b\in\wbK\backslash\{a\}}\theta_b^a(\eta_b - \mu_b) + v_a
\end{align*}
with the residual $v_a$ independent of $\{\eta_b: 1\leq b\neq a\leq K\}$. Let $\btheta^a = (\theta^a_1, \cdots, \theta^a_K)^T \in\mathbb{R}^K$ with $\theta^a_a = 0$, then the neighborhood can also be written as $\nb_a=\{b\in\wbK: \theta_b^a\neq 0\}$.

\cite{meinshausen2006high} consider the case of $n$ i.i.d. observations of $\bEta$. 
However, under the assumption of our model, we only have observations of $\wh{\bEta}=(\wh{\eta}_1,\cdots, \wh{\eta}_K)^T$. Under our assumptions, we have available $n$ independent realizations $\wh{\bEta}^{(1)}, \cdots, \wh{\bEta}^{(n)}$. 
Let $\wh{\bH}=(\wh{\bEta}^{(1)},\cdots,\wh{\bEta}^{(n)})^T$ be the $n\times K$-matrix with columns $\wh{\bEta}_a = (\wh \eta^{(1)}_a,\ldots, \wh\eta_a^{(n)})^T$, $a \in \wbK$.
Similarly denote by $\bH = (\bEta^{(1)},\ldots,\bEta^{(n)})^T$ the $n\times K$-matrix whose rows are $n$ independent copies of $\bEta$. Its column $\bEta_a, a \in \wbK$ are vectors of $n$ independent observations of $\eta_a$. That is, we can also write $\wh{\bH} = (\wh{\bEta}_1,\cdots,\wh{\bEta}_K)$ and $\bH=(\bEta_1,\cdots,\bEta_K)$.  With this notation, for all $a\in\wbK$, 
\begin{align} \label{eq: etaV}
	\bm{\eta_a} - \mu_a\bIn=\sum_{b\in\wbK}\theta_b^a(\bm{\eta_b} - \mu_b\bIn)+\bv_a.
\end{align}
Let $\bR=\wh{\bH}-\bH$. The new matrix model can be written as
\begin{align} \label{model: matrix}
	&(\wh{\bH} - \bIn{\bmu}^T) = (\bH - \bIn{\bmu}^T)+\bR \\
	&(\bEta^{(t)} - \bmu)\sim N(\bm{0},\bSigma) \quad\text{i.i.d. for }t=1,\cdots,n.
\end{align}
Moreover, for each $a\in\wbK$, let $\bH_{-a}=\{\bEta_b: b\in\wbK\backslash\{a\}\}$, $\wh{\bH}_{-a} = \{\wh{\bEta}_b: b\in\wbK\backslash\{a\} \}$ and  $\btheta^a_{-a}=(\theta^a_1,\cdots,\theta^a_{a-1},\theta^a_{a+1},\cdots,\theta^a_K)^T$, $\bmu_{-a} = (\mu_1,\cdots,\mu_{a-1}, \mu_{a+1}, \cdots,\mu_K)^T$. We can write the above model as
\begin{equation} \label{model: mus}
	\begin{aligned}
		\wh{\bEta}_a - \mu_a\bIn = (\bH_{-a} - \bIn\bmu_{-a}^T)\btheta^a_{-a}+\bm{\xi}_a\hspace{7.5mm}\qquad \\
		(\wh{\bH}_{-a} - \bIn\bmu_{-a}^T) = (\bH_{-a} - \bIn\bmu_{-a}^T)+\bR_{-a},\qquad
	\end{aligned}
\end{equation}
where $\bm{\xi}_a=\bv_a+(\wh{\bEta}_a-\bEta_a)$ and $\bR_{-a}=\wh{\bH}_{-a}-\bH_{-a}$.  
Note that \eqref{model: mus} has a similar structure as the model considered by \cite{rosenbaum2010sparse}. The important difference is that in our situation, we do not have independence of $\bm{\xi}_a$ and $\bR_{-a}$. 

\subsection{Edge selection under uncertainty using the Lasso} 
\label{subsec: lasso}
As in \cite{meinshausen2006high}, we define our Lasso estimate $\wh{\btheta}^{a,\lambda,lasso}$ of $\btheta^a$ as
\begin{align} \label{def: thetaLasso}
	\wh{\btheta}^{a,\lambda,lasso}=\argmin\limits_{\btheta: \theta_a=0}\big(n^{-1}\|(\wh{\bEta}_a - \mu_a\bIn) - (\wh{\bH} - \bIn\bm{\mu}^T)\btheta\|_2^2+\lambda\|\btheta\|_1\big).
\end{align}
The corresponding neighborhood estimate is
\begin{align*}
	\wh{\nb}_a^{\lambda, lasso} = \big\{b\in\wbK: \wh{\theta}^{a, \lambda, lasso}_b\neq0\big\}; 
\end{align*}
and the full edge set can be estimated by 
\begin{align*}
	\wh{E}^{\lambda,\wedge,lasso} = \big\{(a,b): a\in\wh{\nb}_b^{\lambda,lasso} \text{ and } b\in\wh{\nb}_a^{\lambda,lasso}\big\}
\end{align*}
or 
\begin{align*}
	\wh{E}^{\lambda,\vee, lasso} = \big\{(a,b): a\in\wh{\nb}_b^{\lambda,lasso}\text{ or }b\in\wh{\nb}_a^{\lambda,lasso}\big\}.
\end{align*}

In order to formulate statistical guarantees for the behavior of these estimates, we need the following assumptions. On top of the assumptions from \cite{meinshausen2006high}, which are assumptions A1.1 - A1.5, we need further assumption on the underlying network. 
\begin{itemize}		
	\item[A1] \emph{Assumptions on the underlying Gaussian graph}
	\begin{enumerate}
		\item \emph{High-dimensionality}: There exists some $\gamma>0$ so that
		$K(n)=O(n^\gamma)$ for $n\rightarrow\infty$.
		
		\item \emph{Nonsingularity}: For all $ a\in\wbK$ and $n\in\mathbb{N}$, $\var(\eta_a)=1$ and there exists $\upsilon^2>0$ so that
		\begin{align*}
			\var(\eta_a|\bEta_{\wbK\backslash\{a\}})
			\geq \upsilon^2.
		\end{align*}
		
		\item \emph{Sparsity}
		\begin{enumerate}
			\item There exists some $0\leq\kappa<1$ so that $
			\max_{a\in\wbK}\left|{\nb}_a\right|=O(n^\kappa)$ for $n\rightarrow\infty$.
			\item There exists some $\vartheta<\infty$ so that for all neighboring nodes $a,b\in\wbK$ and all $n\in\mathbb{N}$, 
			\begin{align*}
				\|\btheta^{a,{\nb}_b\backslash\{a\}}\|_1\leq\vartheta.
			\end{align*}
		\end{enumerate}
		
		\item \emph{Magnitude of partial correlations}: There exist a constant $c>0$ and some $1\geq\xi>\kappa$, so that for all $(a,b)\in E$,
		\begin{align*}
			|\pi_{ab}|\geq cn^{-(1-\xi)/2},
		\end{align*}
		where $\pi_{ab}$ is the partial correlation between $\eta_a$ and $\eta_b$.
		
		\item \emph{Neighborhood stability}: There exists some $\varrho<1$ so that for all  $ a,b\in\wbK$ with $b\notin{\nb}_a$,
		\begin{align*}
			|S_a(b)|<\varrho
		\end{align*}
		where 
		\begin{align*}
			S_a(b)=\sum_{k\in{\nb}_a}\sign\left(\theta_k^{a,{\nb}_a}\right)\theta_k^{b,{\nb}_a}.
		\end{align*}
		
		\item \emph{Asymptotic upper bound on the mean:} $\mu_B(n)=o(\log n)$ for $n\rightarrow\infty$.
	\end{enumerate}
	
	\item[A2] \emph{Block size of networks}: 
	There exists constants $c>0$ and $n_0,$ such that 
	\begin{align*}
		m_{\min}(n) \ge c\cdot n^\nu
		\quad
		\text{for }n\geq n_0,
	\end{align*}
	where $\nu>\max\{4-4\xi,2-2\xi+2\kappa\}.$
\end{itemize}
The following theorem shows that, for proper choice of $\lambda=\lambda_n$, our selection procedure finds the correct neighborhoods with high probability, provided $n$ is large enough. 
\begin{theorem} \label{corollary: net_mb0_asy}
	Let assumptions A1 and  A2 hold, and assume $\bbeta$ to be known. Let $\epsilon$ be such that 
	\begin{align*}
		0 < \max\big\{\kappa, \tfrac{4-\xi-\nu}{3},\tfrac{2+2\kappa-\nu}{2} \big\} < \epsilon <\xi. 
	\end{align*}
	If, for some $d_T, d_\lambda >0,$ we have $T_n \sim d_T\log n$ and $\lambda_n\sim d_\lambda n^{-(1-\epsilon)/2}$ \footnote{For two sequence $\{a_n\}$, $\{b_n\}$ of real numbers, we write $a_n \sim b_n$ for $\frac{a_n}{b_n}\rightarrow c$ for some $0 < c <\infty$. }, 
	respectively, then there exists a constant $c>0,$ such that 
	\begin{align*}
		P(\wh{E}^{\lambda,lasso} = E)
		= 1 - O\left(\exp\left(-c(\log n)^2\right)\right)
		\quad
		\text{as }n\rightarrow\infty.
	\end{align*}
\end{theorem}
\begin{remark}
	Assumption A2 says that the rate of increase of the minimum block size, which behaves like $\sqrt{m_{\min}},$ depends on the neighborhood size in  our graphical model, and on the magnitude of the partial correlations in the graphical model. Roughly speaking, large neighborhoods (large $\kappa$), and small partial correlations (small $\xi$), both require a large minimum block size (large $\nu$), which appears reasonable. The choice of a proper penalty parameter $\lambda_n$ also depends on these two parameters.
\end{remark}

\subsection{Edge selection with a class of Dantzig-type selectors under uncertainty} \label{subsec: ds}
In this section, we propose a novel class of Dantzig-type selectors that are iterated over all $a\in\wbK$. For a linear model as in \eqref{eq: etaV}, i.e. for fixed $a$, \cite{candes2007dantzig} introduced the Dantzig selector as a solution to the convex problem 
\begin{multline*}
	\quad
	\min\Big\{\|\btheta\|_1:\btheta\in\mathbb{R}^{K(n)}, \theta_a=0\text{ and } \\
	\Big|\frac{1}{n}\big(\bH_{-a} - \bIn\bmu_{-a}^T\big)^T\big((\bEta_a - \mu_a\bIn) - (\bH - \bIn{\bmu}^T)\btheta\big)\Big|_\infty \leq \lambda\Big\}, 
	\quad
\end{multline*}
where $\lambda\geq0$ is a  tuning parameter, and for a matrix $A =(a_{ij})$, $|\cdot|_\infty = \max_{ij}|a_{ij}|.$  Under our model, we define the Dantzig selector as a solution of the minimization problem 
\begin{multline} \label{def: thetaDs}
	\quad
	\min\Big\{\|\btheta\|_1:\btheta\in\mathbb{R}^{K(n)}, \theta_a=0\text{ and } \\
	\Big|\frac{1}{n}\big(\wh{\bH}_{-a} - \bIn\bmu_{-a}^T\big)^T\big((\wh{\bEta}_a - \mu_a\bIn) - (\wh{\bH} - \bIn\bmu^T)\btheta\big)\Big|_\infty \leq \lambda\Big\}
	\quad
\end{multline}
with $\lambda\geq0$. Moreover, when considering \eqref{model: mus}, the idea of matrix uncertainty selector (MU-selector) comes into our mind. In our setting, we define an MU-selector, a generalization of the Dantzig selector under matrix uncertainty, as a solution of the minimization problem
\begin{multline} \label{def: thetaMus}
	\quad \min\Big\{\|\btheta\|_1:\btheta\in\mathbb{R}^{K(n)}, \theta_a=0\text{ and} \\
	\Big|\frac{1}{n}\big(\wh{\bH}_{-a} - \bIn\bmu_{-a}^T\big)^T\big((\wh{\bEta}_a - \mu_a\bIn)-(\wh{\bH} - \bIn\bmu^T)\btheta\big)\Big|_\infty \leq \mu\|\btheta\|_1+\lambda\Big\}
\end{multline}
with tuning parameters $\mu\geq0$ and $\lambda\geq0$. Note that our MU-selector deals with matrix uncertainty directly, rather than replacing $\bH$ by $\wh{\bH}$ in the optimization equations like the Lasso or the Dantzig selector. What we mean by this is that our MU-selector is based on the structural equation \eqref{model: mus}, while both Lasso-based estimator and Dantzig selector are based on the linear model \eqref{eq: etaV} with the unknown $\bm{\eta}$'s simply replaced by their estimators.

Now we consider a class of Dantzig-type selectors, which can be considered as generalizations of the Dantzig selector and the MU-selector.  For each $a\in\wbK$, let the Dantzig-type selector $\wt{\btheta}^{a,\lambda,ds}$ be a solution of the optimization problem 
\begin{multline} \label{def: thetaG}
	\quad \min\Big\{\|\btheta\|_1: \btheta\in\mathbb{R}^{K(n)}: \theta_a=0\text{ and } \\
	\Big|\frac{1}{n}(\wh{\bH}_{-a} - \bIn\bmu_{-a}^T)^T\big((\wh{\bEta}_a - \mu_a\bIn)-(\wh{\bH} - \bIn\bmu^T)\btheta\big)\Big|_\infty
	\leq \lambda_{a,n}(\|\btheta\|_1)\Big\}, 
\end{multline}
where for each $n\in\mathbb{N}$, $\{\lambda_{a,n}(\cdot): a\in\wbK\}$ is a set of functions such that 
\begin{itemize}
	\item For each $n\in\mathbb{N}$ and $a\in\wbK$, $\lambda_{a,n}(\cdot)$ is an increasing function. 
	
	\item For all $n\in\mathbb{N}$, $\min_{a\in\wbK}\lambda_{a,n}(\cdot)$ is lower bounded by some constant $\lambda_n \geq 0$, i.e, for all $n\in\mathbb{N}$, there exists some $\lambda_n >0$ so that 
	\begin{align*}
		\min_{ a\in\wbK}\min_{\btheta\in\mathbb{R}^{K}: \theta_a = 0}\lambda_{a,n}(\|\btheta\|_1) \geq \lambda_n. 
	\end{align*}
	
	\item $\max_{a\in\wbK}\lambda_{a,n}(\|\btheta^a\|_1) = o(n^{-\frac{1-\xi}{2}}(\log n)^{-1})$, i.e, there exist $u_n = o(1)$ and $n_0\in\mathbb{N}$, so that, for all $n\geq n_0$, 
	\begin{align*}
		\lambda_{a,n}(\|\btheta^a\|_1) \leq u_nn^{-\frac{1-\xi}{2}}(\log n)^{-1}, \text{ for all }a\in\wbK. 
	\end{align*}
\end{itemize}
The Dantzig-type selector $\wt{\btheta}^{a,\lambda,ds}$ always exists, because the LSE $\wh{\btheta}^a$ defined as $\wh{\btheta}^a_{-a} = (\wh{\bH}_{-a} -\bIn\bmu_{-a}^T )^{+}(\hat{\bEta}_a - \mu_a\bIn)$ and $\wh{\theta}^a_{a} = 0$ belongs to the feasible set $\Theta_a$, where  
\begin{multline*}
	\Theta_a=\Big\{\btheta\in\mathbb{R}^{K(n)}: \theta_a=0\text{ and } \\
	\Big|\frac{1}{n}\big(\wh{\bH}_{-a} - \bIn\bmu_{-a}^T\big)^T\big((\wh{\bEta}_a - \mu_a\bIn)-(\wh{\bH} - \bIn\bmu^T)\btheta\big)\Big|_\infty \leq \lambda_{a,n}(\|\btheta\|_1)\Big\} 
\end{multline*}
for any $\lambda_{a,n}(\|\btheta\|_1)\geq0$. It may not be unique, however.  We will show that, similar to \cite{candes2007dantzig} and \cite{rosenbaum2010sparse}, under certain conditions, for large $n$, there exists a constant $t > 0$ such that the $l_\infty$-norm of the difference between the Dantzig-type selector $\wt{\btheta}^{a,\lambda,ds}$ and the population quantity $\btheta^a,$ can be bounded by $t\lambda_{a,n}(\|\wt{\btheta}^{a,\lambda,ds}\|_1)$ for all $a\in\wbK$ with large probability, where $t_n$ can be a constant large enough or of order $\log n$.  However, in general, sparseness cannot be guaranteed. This already has been observed in \cite{rosenbaum2010sparse}. Therefore, we consider a thresholded version of the Dantzig-type selector, which can also significantly improve the accuracy of the estimation of the sign.
Let $\wh{\btheta}^{a,\lambda,ds}\in\mathbb{R}^{K(n)}$ be defined as 
\begin{align} \label{def: thetaPd}
	\wh{\theta}^{a,\lambda,ds}_b=
	\begin{cases}
		0, & \quad b=a \\
		\wt{\theta}^{a,\lambda,ds}_bI\big(|\wt{\theta}^{a,\lambda,ds}_b| > t_n\lambda_{a,n}(\|\wt{\btheta}^{a,\lambda,ds}\|_1)\big), & \quad
		b\in\wbK\backslash\{a\}, \\
	\end{cases}
\end{align}
where $I(\cdot)$ is the indicator function, and $t_n$ is a sequence that satisfies $t_n^{-1} = o(1)$ and $t_n = O(\log n)$.  The corresponding neighborhood selector is, for all $a\in\wbK,$  defined as $\wh{\nb}_a^{\lambda,ds}=\{b\in\wbK: \wh{\theta}_b^{a,\lambda,ds}\neq 0\},$
and the corresponding full edge selector is 
\begin{align*}
	\wh{E}^{\lambda,\wedge,ds}=\big\{(a,b): a\in\wh{\nb}_b^{\lambda,ds} \text{ and } b\in\wh{\nb}_a^{\lambda,ds}\big\}
\end{align*}
or 
\begin{align*}
	\wh{E}^{\lambda,\vee,ds}=\big\{(a,b): a\in\wh{\nb}_b^{\lambda,ds} \text{ or } b\in\wh{\nb}_a^{\lambda,ds}\big\}.
\end{align*}

Similar to the Section~\ref{subsec: lasso}, in order to derive some consistency properties, we need assumption about the underlying Gaussian graph (B1), and the minimum block size in the underlying network (B2). 
\begin{itemize}
	\item[B1] \emph{Assumptions on the underlying Gaussian graph}
	\begin{enumerate}
		\item \emph{Dimensionality}: There exists $\gamma>0$ such that
		$K(n)=O(n^\gamma)$ as $n\rightarrow\infty$.
		
		\item \emph{Nonsingularity}: For all $ a\in\wbK$ and $n\in\mathbb{N}$, $\var(\eta_a)=1$ and there exists $\upsilon^2>0$ so that
		\begin{align*}
		\var(\eta_a|\bEta_{\wbK\backslash\{a\}})\geq \upsilon^2.
		\end{align*}
		
		\item \emph{Sparsity}
		\begin{enumerate}
			\item There exists
			$0\leq\kappa<1/2$, so that $
			\max_{a\in\wbK}\left|{\nb}_a\right|=O(n^\kappa)$, as $n\rightarrow\infty$.
			\item $\|\bSigma^{-1}\|_\infty = O(1)$ as $n\rightarrow\infty$.  
		\end{enumerate}
		
		\item \emph{Magnitude of partial correlations}: There exist a constant $c>0$ and $1\geq\xi>\kappa$, so that, for all $(a,b)\in E$, $|\pi_{ab}|\geq cn^{-(1-\xi)/2}$.

%
		\item \emph{Asymptotic upper bound on the mean:} $\mu_B(n)=o(\log n)$ for $n\rightarrow\infty$.
	\end{enumerate}
	
	\item[B2] \emph{Block size of networks}: $m_{\min}^{-1}(n)=O(n^{-\nu})$ with some $\nu>1-\xi+2\kappa$ for $n\rightarrow\infty$.
\end{itemize}
Here, the assumption on $m_{\min}$ (assumption B2) is weaker than that assumed for the Lasso-based estimator (assumption A2). Similar remarks as given for A2 also apply to B2 (see Remark right below Theorem~\ref{corollary: net_mb0_asy}).

Assumptions A1 and B1 are similar but not equivalent: A1.1 and B1.1, A1.2 and B1.2, A1.4 and B1.4 respectively, are exactly the same. B1.2(a)  is stronger than A1.3.(a), indicating the underlying graph should be even sparser than the graph in Section~\ref{subsec: lasso}; assumption B1 does not have an analog to A1.3.(b) and A1.5.
\begin{theorem} \label{corollary: net_ds0_asy}
	Let assumptions B1 and B2 hold, and assume $\bbeta$ is known. Let $\epsilon>0$ be such that $\xi>\epsilon>1+2\kappa-\nu$. If $T_n \sim d_T\log n$ with some $d_T>0$, and  $\lambda_n^{-1} = O(n^{\frac{1-\epsilon}{2}})$, there exists $c>0$, so that  
	\begin{align*}
		P(\wh{E}^{\lambda,ds} = E)
		= 1 - O\left(\exp\left(-c(\log n)^2\right)\right)
		\quad
		\text{as }n\rightarrow\infty.
	\end{align*}
\end{theorem}
\begin{remark}	
	The choice of proper $\lambda_{a,n}(\cdot)$ depends on the three parameters $\xi, \kappa$ and $\nu$. However, even the best scenario does not allow for the order $\lambda_p \sim \sqrt{\frac{\log p}{n}},$ which often can be found in the literature. This stems from the fact that we have to deal with an additional estimation error (coming in through the estimation of $\boldsymbol \eta$).
\end{remark}

\subsection{Extension} \label{subsec: extension}
Here we consider the case of an unknown coefficient vector $\bbeta,$ or unknown mean $\bmu = \bX\bbeta$. 
Recall that $\bEta^{(t)}\sim N(\bmu,\bSigma),$ $t=1,\cdots,n$ are i.i.d. Given $\{\bEta^{(t)}: t=1,\cdots,n\}$, a natural way to estimate $\bmu$ is via the MLE $\wb{\bEta}=\frac{1}{n}\sum_{t=1}^{n}\bEta^{(t)}$. Recall, however, that we only have estimates $\wh{\bEta}^{(t)}, t=1,\cdots,n$, available. Using the estimates $\wh{\bEta}^{(t)}$, we estimate the underlying mean $\bmu$ by  $\wb{\wh{\bEta}}=\frac{1}{n}\sum_{t=1}^{n}\wh{\bEta}^{(t)}$. Moreover, we can estimate $\bbeta$ via $\wh{\bbeta}=\bX^{+}\wb{\wh{\bEta}}$, where $\bX^+$ is the Moore-Penrose pseudoinverse of $\bX$ (when $\rank(\bX)=L$, $\bX^+=(\bX^T\bX)^{-1}\bX^T$). In order to derive consistency properties for $\wh{\bbeta}$, assumptions on the design matrix are needed.
Theorem~\ref{thm: net_beta} below states asymptotic properties of the estimators.
\begin{theorem} \label{thm: net_beta}
	Let assumptions A1.1 (or B1.1) and A1.6 (or B1.5) hold. If $m^{-1}_{\min}=O(n^{-\nu})$ for some $\nu>0$, then, for any $b<\min\{1,\nu\}$, and fixed $\delta>0$, there exists some $c>0$ so that
	\begin{align*}
		P\big(n^{\tfrac{b-\gamma}{2}}\|\wb{\wh{\bEta}}-\bmu\|_2>\delta\big)
		=
		O\left(\exp\left(-c(\log n)^2\right)\right)
		\quad
		\text{as }n\rightarrow\infty. 
	\end{align*}
	If, moreover, the design matrix is of full rank and the singular value of $\bX$ is asymptotically upper bounded, that is, $\rank(\bX)=L$ and $\sigma_{\max}(\bX)=O(1)$, then there exists $c>0,$ so that 
	\begin{align*}
		P\big(n^{\tfrac{b-\gamma}{2}}\|\wh{\bbeta}-\bbeta\|_2>\delta\big)=O\left(\exp\left(-c(\log n)^2\right)\right)
		\quad
		\text{as }n\rightarrow\infty.
	\end{align*}
\end{theorem}
Next we consider the estimation of the edge set $E$ based on $\bD=\bSigma^{-1}$. We write $\bEta-\bmu\sim N(\bm{0},\bSigma)$ and consider $(\wh{\bEta}^{(t)}-\wb{\wh{\bEta}})_{t=1,\cdots,n}$ as the observations. We estimate the edge set in the same way as described in Section~\ref{subsec: lasso}, but replace $\wh{\bEta}_a - \mu_a\bIn$ by $\wh{\bEta}_a-\wb{\wh{\eta}}_a\bIn$ and replace $\wh{\bH}$ by $\wh{\bH}-\bIn\wb{\wh{\bEta}}^{T}$ in \eqref{def: thetaLasso}, where $\wb{\wh{\eta}}_a=\frac{1}{n}\sum_{t=1}^n\wh{\eta}_a^{(t)}$ and $\wb{\wh{\bEta}}$ is as above. 
The following consistency result parallels Theorem~\ref{corollary: net_mb0_asy} and Theorem~\ref{corollary: net_ds0_asy}, but stronger assumption are needed to control the additional estimate error. 
\begin{corollary} \label{corollary: net_mb_asy}
	Let assumptions A1 - A2 hold with $\xi>3/4$, and let $\epsilon$ be such that 
	\begin{align*}
		\max\{\kappa+1/2, \tfrac{3-\xi}{3}, \tfrac{4-\xi-\nu}{3},\tfrac{2+2\kappa-\nu}{2}\} < \epsilon < \xi.
	\end{align*}
	Suppose that $T_n \sim d_T\log n$, for $d_T>0,$ and that the penalty parameter satisfies $\lambda_n\sim d_\lambda n^{-(1-\epsilon)/2}$ for some $d_\lambda>0$. Then, there exists $c>0,$ so that 
	\begin{align*}
		P(\wh{E}^{\lambda,lasso}=E)=1- O\left(\exp\left(-c(\log n)^2\right)\right)
		\quad
		\text{as }n\rightarrow\infty.
	\end{align*}
\end{corollary}
\begin{corollary} \label{corollary: net_ds_asy}
	Let assumptions B1 - B2 hold with $\xi>2\kappa$. Let $\epsilon$ be such that  $\xi>\epsilon>\max\{2\kappa, 2\kappa+1-\nu\}$. If $T_n \sim d_T\log n,$ for some $d_T>0$, and  $\lambda_n^{-1} = O(n^{\frac{1-\epsilon}{2}})$, there exists $c>0$ so that  
	\begin{align*}
		P(\wh{E}^{\lambda,ds}=E)=1- O\left(\exp\left(-c(\log n)^2\right)\right)
		\quad
		\text{as }n\rightarrow\infty.
	\end{align*}
\end{corollary}
\begin{example} \label{eg: all1}
	Let $\kappa=0$ and $\xi=1$, that is, the number of blocks is finite, and the partial correlations are lower bounded for the graphcial model. If, in addition, for some $\nu > 0,$ $m_{\min}^{-1}(n)=O(n^{\nu})$ as $n\rightarrow\infty$, then Corollaries~\ref{corollary: net_mb_asy} and \ref{corollary: net_ds_asy}, respectively, apply in the following scenarios: 
	\begin{itemize}
		\item \textbf{The Lasso}: If assumption A1 - A2 hold: Choose the tuning parameter $\lambda_n\sim dn^{-(1-\epsilon)/2}$ with any $\epsilon$ satisfying $1>\epsilon>\max\{0,1-\nu/2\}$ in case $\bmu$ is known, and $\epsilon$ satisfying $1>\epsilon>\max\{2/3,1-\nu/2\}$ for $\bmu$ unknown.
		
		\item \textbf{The Dantzig-type selector}: If assumptions B1 - B2 hold, whether $\bmu$ is known or unknown, 
		choose $\max_{a\in\wbK}\lambda_{a,n}^{-1}(\btheta^a) = O(n^{\frac{1-\epsilon}{2}})$ with any positive $\epsilon$ satisfying $1>\epsilon>1-\nu$. In particular, for \\[-15pt]
		\begin{itemize}[label=\textasteriskcentered]
			\item \textbf{the Dantzig selector}: $\lambda_{a,n}(\|\btheta\|_1) = dn^{-\frac{1-\epsilon}{2}}$ for any $d>0$. Problem \eqref{def: thetaG} becomes  \eqref{def: thetaDs} with tuning parameter $\lambda = dn^{-\frac{1-\epsilon}{2}}$. 
			
			\item \textbf{the MU-selector}: $\lambda_{a,n}(\|\btheta\|_1) = dn^{-\frac{1-\epsilon}{2}}\|\btheta\|_1+dn^{-\frac{1-\epsilon}{2}}$ for any $d>0$. Problem \eqref{def: thetaG} becomes \eqref{def: thetaMus} with tuning parameter $\mu=dn^{-\frac{1-\epsilon}{2}}$ and $\lambda = dn^{-\frac{1-\epsilon}{2}}$. 
		\end{itemize}
	\end{itemize}
\end{example}

\subsection{Selection of penalty parameters in finite samples.} \label{subsec: finite}
The results above only show that consistent edge selection is possible with the Lasso and the Dantzig-type selector in a high-dimensional setting. However, we still have not given a concrete way to choose the penalty parameter for a given data set. In this section, we discuss the choice of tuning parameter for finite $n$ for the following estimation methods: 
\begin{itemize}
	\item The Lasso
	\item The Dantzig-type selectors:
	\begin{itemize}
		\item the Dantzig selector: $\lambda(\|\btheta\|_1) = \lambda$
		\item the MU-selector: $\lambda(\|\btheta\|_1) = \lambda\|\btheta\|_1 + \lambda$
	\end{itemize}
\end{itemize}

\cite{meinshausen2006high} discussed a data-driven choice of the penalty parameter of the Lasso for Gaussian random vectors. Our data are not Gaussian, however. Moreover, according to our numerical studies, 
the choice suggested by Meinshausen and B\"{u}hlmann tends to result in a very sparse graph, which goes along with a very small type I error. Another natural idea is choosing the penalty parameter via cross-validation. However, \cite{meinshausen2006high} already state that the choice $\lambda_{\rm oracle}$ gives an inconsistent estimator, and $\lambda_{\rm cv}$ is an estimate of $\lambda_{\rm oracle}$. So the cross-validation approach is also not recommended. Instead we here consider the following two-stage procedure: for each $a\in\wbK$, let
\begin{equation} \label{def: theta2stage}
	\wh{\theta}_b^{a, \lambda, \tau} = 
	\begin{cases}
		\wt{\theta}_b^{a, \lambda}I\big(|\wt{\theta}_b^{a, \lambda}|/\max_{k\in\wbK\backslash\{a\}}|\wt{\theta}_b^{a, \lambda}| > \tau \big) & \mbox{if  $\wt{\btheta}^{a, \lambda} \neq \bm{0}$ }\\
		
		0 & \mbox{if  $\wt{\btheta}^{a, \lambda} = \bm{0}$ }, 
	\end{cases}
\end{equation}
where $\wt{\btheta}^{a, \lambda}$ is obtained by solving either \eqref{def: thetaLasso}, \eqref{def: thetaDs} or \eqref{def: thetaMus} with $\mu = \lambda$. Such procedures have also been used in \cite{rosenbaum2010sparse} and \cite{zhou2011high}. However, the use of $\max_{k\in\wbK\backslash\{a\}}|\wh{\theta}_b^{a, \lambda}|$ in the truncation is novel. By using $\max_{k\in\wbK\backslash\{a\}}|\wh{\theta}_b^{a, \lambda}|$, we have $\tau \in [0, 1]$, making the tuning parameter $\tau$ more standardized. 
Note that when  $\wh{\theta}_b^{a, \lambda}$ is a Dantzig-type selector, then, under the assumptions in Section~\ref{sec: est}, and for large $n$, \eqref{def: theta2stage} is equivalent to \eqref{def: thetaPd}. 

For the choice of $\lambda$ and $\tau$, we follow a similar idea as in \cite{zhou2011high}, but with some modification: for each $a\in\wbK$, we select $\lambda_a$ via cross-validation to minimize the squared error prediction loss for $a$-th regression. After all $\lambda_a$,  $a\in\wbK$, are chosen, we select $\tau$ via BIC based on a Gaussian assumption: 
\begin{align*}
	{BIC}(\bD) = -2l_n(\bD) + \log(n)\dim(\bD), 
\end{align*}
where $l_n(\bD)$ is the $n$-sample Gaussian log-likelihood and $\dim(\bD) = $ number of free parameters. Note that we do not have a nice form of the likelihood, so we use the Gaussian likelihood instead. 

\section{Simulation study} \label{sec: num}
In this section, we mainly study the finite sample behavior of the three estimation methods mentioned in Section~\ref{subsec: finite}, that is, 
\begin{itemize}
	\item the Lasso; 
	\item the Dantzig selector; 
	\item the MU-selector with $\mu = \lambda$.
\end{itemize}

\subsection{Finite-sample performance as a function of the penalty parameter.} \label{subsec: finite_performance}
Here we consider the methods proposed in Section~\ref{subsec: lasso} and \ref{subsec: ds} with an AR(1) type covariance structure $\bSigma_{K\times K} = \{\rho^{|i-j|}\}_{i, j \in \wbK}$ with $\rho = 0.2, 0.5$ and $0.8$. In this setup, $d_{ij} = 0$ if and only if $|i - j| > 1$. The minimum blocksize in our simulation is set to be $m_{\min} = 100$. We consider the following choices of the sample size and number of blocks: 
\begin{itemize}
	\item $n = 20$ with $K = 15$; 
	\item $n = 100$ with $K = 15, 30, 50, 80, 100$ and $150$. 
\end{itemize}
We only present the results for $n = 20, K = 15$ and $n=100, K=150.$ The rest of the results can be found in the supplementary material. 
Figures~\ref{fig: roc1} - \ref{fig: roc7} show ROC-curves; average error rates (total error, type \uppercase\expandafter{\romannumeral1} error and type \uppercase\expandafter{\romannumeral2} error) as functions of the tuning parameter $\lambda$ are shown in figures~\ref{fig: rate1}  and \ref{fig: rate7}.  
The shown curves are color-coded: Lasso: red, Dantzig selector: blue and MU-selector: green. $\lambda_{\rm opt}$ is the tuning parameter corresponding to the total (overall) minimum error rate.
\begin{figure} [H]
	\centering
	\includegraphics[width=1\linewidth]{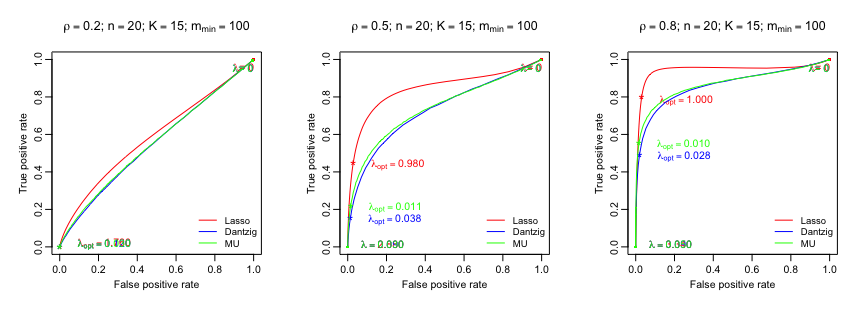} 
	\vspace*{-1cm}
	\caption{\small ROC curves comparing the three proposed methods for $K = 15$ and $n = 20$}
	\label{fig: roc1} 
\end{figure}
\begin{figure} [H]
	\centering
	\includegraphics[width=1\linewidth]{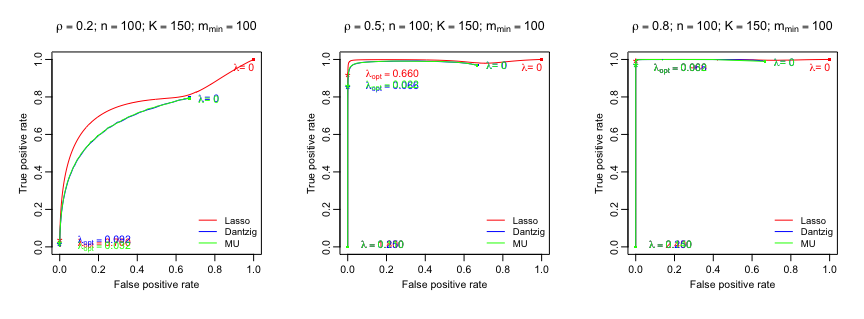} 
	\vspace*{-1cm}
	\caption{\small ROC curves comparing the three proposed methods for $K = 150$ and $n = 100$}
	\label{fig: roc7} 
\end{figure}
We can see that the value of $\rho$ is important. The performance of all the three methods improves as $\rho$ grows. This can be understood by the fact that it determines the size of the partial correlations (cf. assumption A1.4).  
\begin{figure} [H]
	\centering
	\includegraphics[width=1\linewidth]{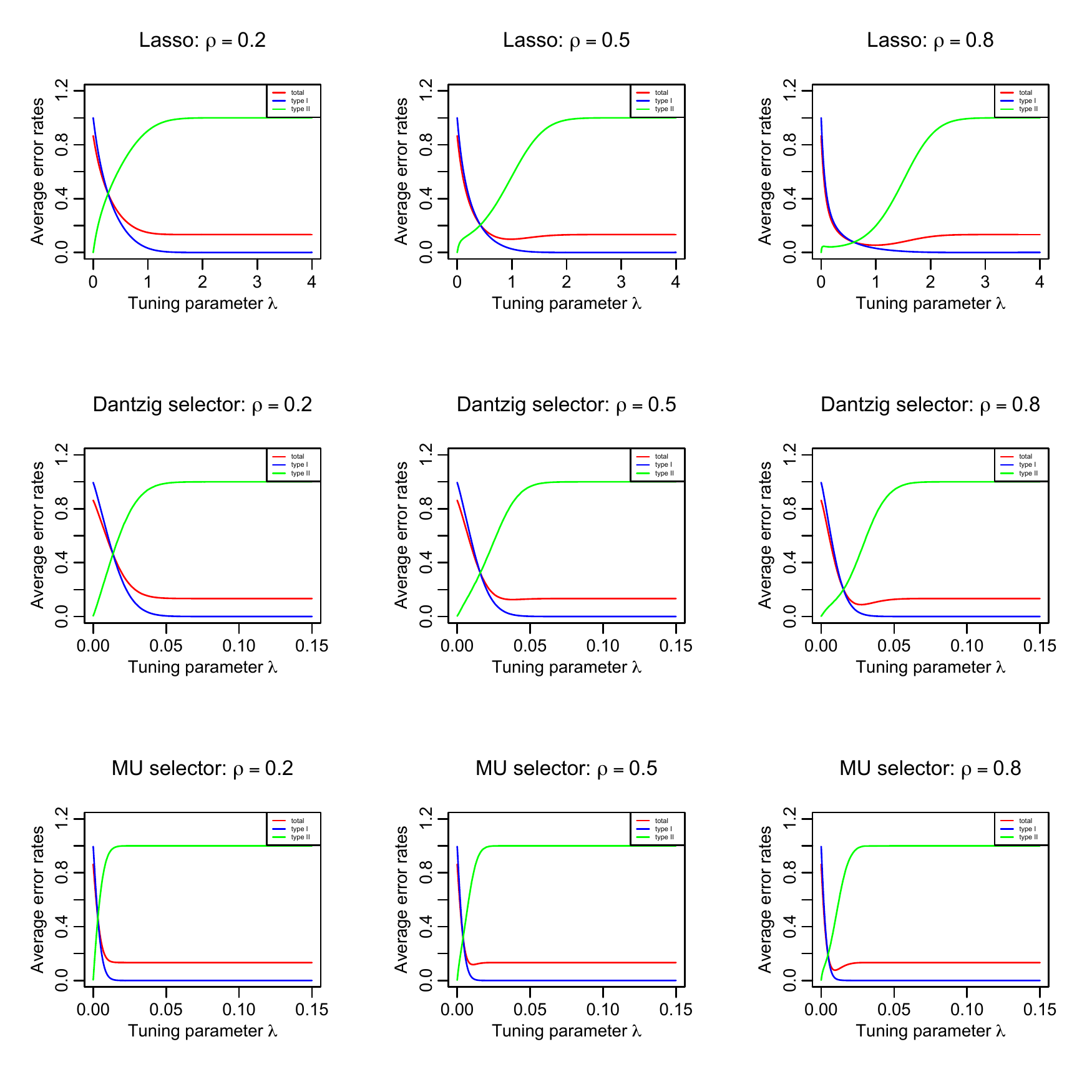} 
	\caption{\small Average error rates as functions of $\lambda$ for $K = 15$ and $n = 20$}
	\label{fig: rate1} 
\end{figure}
Moreover, when $n \geq K$, and for $\lambda = 0$, all these methods result in estimates with all components being non-zero, which result in type \uppercase\expandafter{\romannumeral1} error rate equal to $1$ and type \uppercase\expandafter{\romannumeral2} error equals $0$, that is, $(1, 1)$ in the ROC curves. However, when $n < K$, and $\lambda = 0$, the feasible set $\Theta_a$ is dimension at least $(K - n)$.  The Dantzig-type selectors minimize the $L_1$-norm of these $\btheta$'s, which produces some zero terms of the solution; thus, the corresponding type \uppercase\expandafter{\romannumeral1}  error rate will be less than $1$ and the type \uppercase\expandafter{\romannumeral2} error rate might be greater than $0$, that is why the ROC curves of the Dantzig selector and the MU-selector cannot reach $(1, 1)$ for the case $n = 100$ with $K = 150$. However, the solution of the Lasso is not unique, the coordinate decent algorithm could return a solution with all its elements non-zero, resulting $(1, 1)$ in the ROC curves. 
\begin{figure} [H]
	\centering
	\includegraphics[width=1\linewidth]{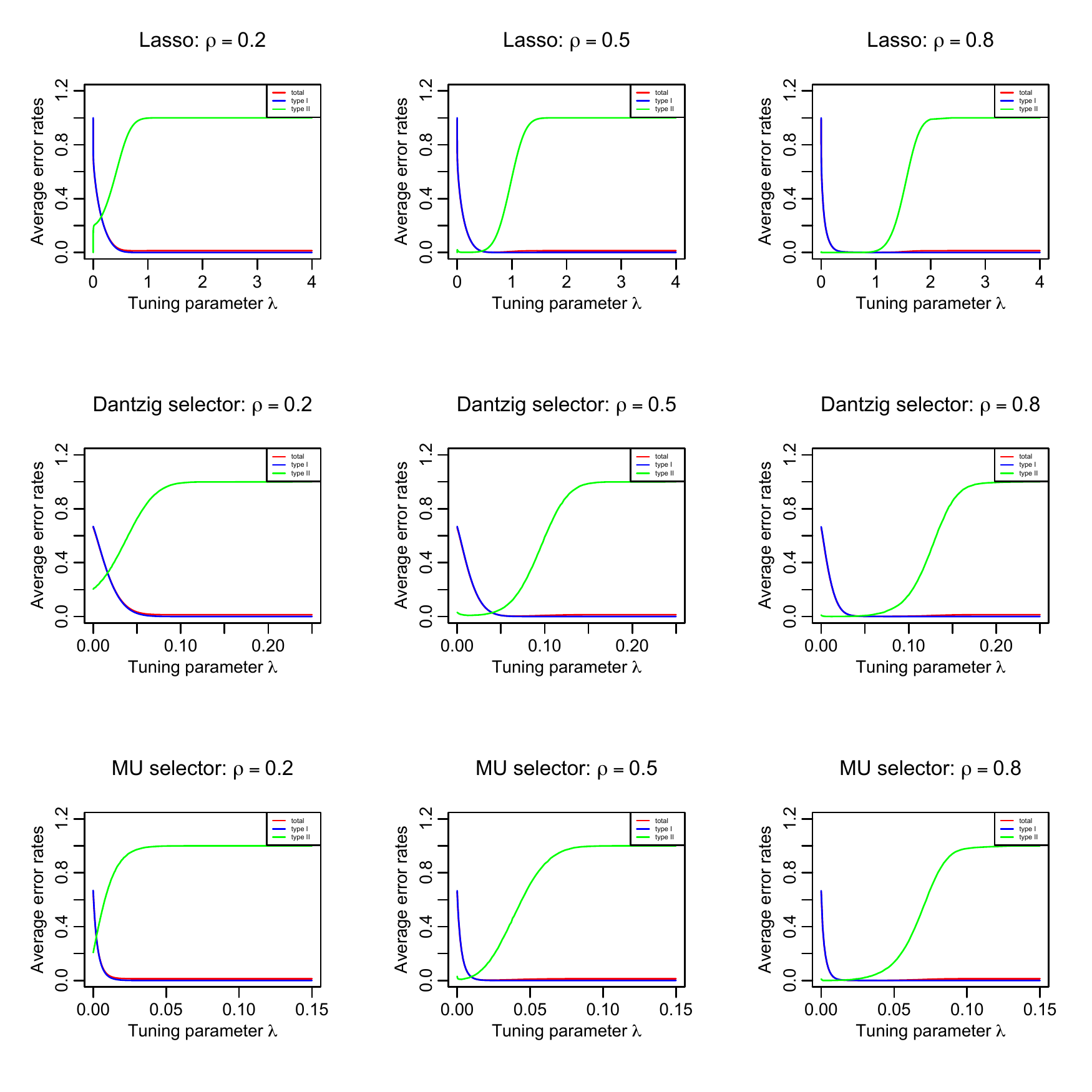} 
	\caption{Average error rates as functions of  $\lambda$ for $K = 150$ and $n = 100$}
	\label{fig: rate7} 
\end{figure}

\subsection{Finite-sample performance with data-driven penalty selection} \label{subsec: finite_penalty}
In this section, we  study the three methods for finite-sample setup discussed in Section~\ref{subsec: finite}. In our simulation study, we consider three different covariance models with $K = 30, 100, 200$, $m_{\min} = 45$ and $n = 100, 500, 1000$. Below we only present the case $K=100$. See supplemental material for the other cases.
\begin{itemize}
	\item AR(1): $\bSigma_{K\times K} = \{\rho^{|i-j|}\}_{1\leq i, j \leq K}$ with $\rho = 0.7$. 
	
	\item AR(4): $d_{ij} = I(|i - j| = 0) + 0.4\dot I(|i - j| = 1) + 0.2\cdot I(|i - j| = 2) + 0.2\cdot I(|i - j| = 3) + 0.1\cdot I(|i - j| = 4)$.
	
	\item A random precision matrix model (see \cite{rothman2008sparse}): $\bD_K = \bB + \delta\bI$ with each off-diagonal entry in $\bB$ is generated independently and equals $0.5$ with probability $\alpha$ or $0$ with probability $1 - \alpha$. $\bB$ has zeros on the diagonal, and $\delta$ is chosen so that the condition number of $\bD$ is $K$.
\end{itemize}
As mentioned in Section~\ref{subsec: finite}, we choose $\lambda$ via cross-validation, and $\tau$ based on $BIC(\tau)$. As for the choice of $\tau$, we often encountered the problem of a very flat BIC-function close to the level of the minimum (some $BIC(\tau)$ plots are shown in figure~\ref{fig: bic}). To combat this problem, we use the following strategy in our simulations: if more than half of the $\tau\in[0, 1]$ result in the same BIC, then we choice the third quartile of these $\tau$'s, otherwise, we choose the one resulting the minimum BIC.
\begin{figure} [H]
	\begin{subfigure}[b]{0.5\linewidth}
		\centering
		\includegraphics[width=1\linewidth]{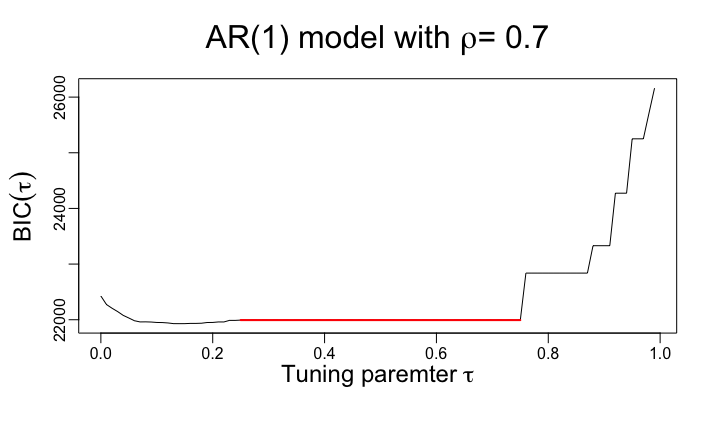} 
	\end{subfigure}
	\begin{subfigure}[b]{0.5\linewidth}
		\centering
		\includegraphics[width=1\linewidth]{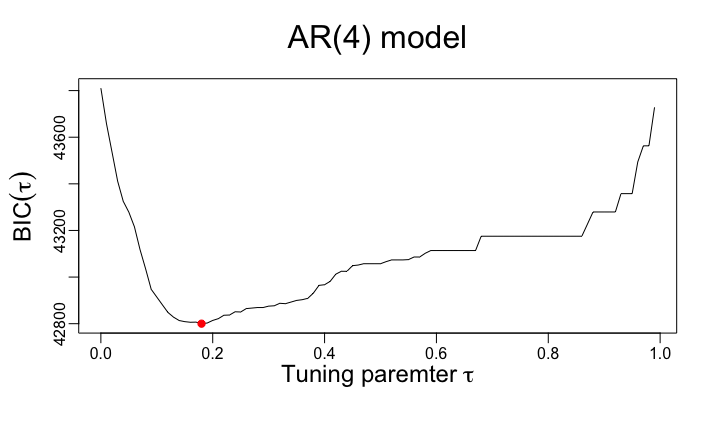} 
	\end{subfigure} 
	\begin{subfigure}[b]{0.5\linewidth}
		\centering
		\includegraphics[width=1\linewidth]{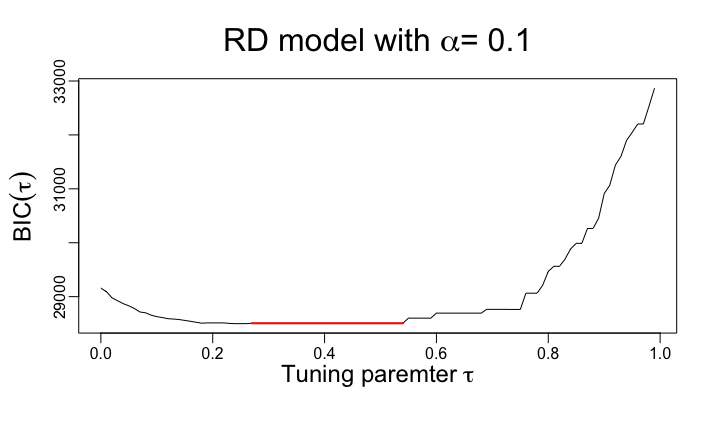} 
	\end{subfigure}
	\begin{subfigure}[b]{0.5\linewidth}
		\centering
		\includegraphics[width=1\linewidth]{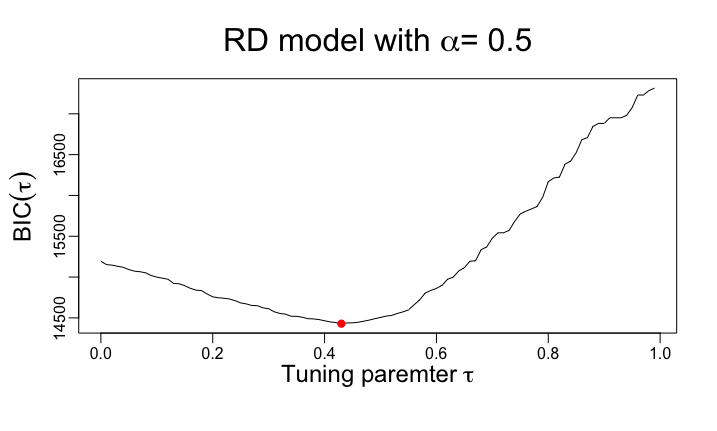} 
	\end{subfigure} 
	\caption{Tuning parameter $\tau$ vs $BIC(\tau)$ plots for the four models with $K = 30$: the red points corresponds to the optimal $\tau$}
	\label{fig: bic} 
\end{figure}

Simulation results for AR(1) and AR(4) models are shown in tables~\ref{tab: ar1_k100} and \ref{tab: ar4_k100}, respectively. For the random precision matrix model we consider $\alpha = 0.1$ and $\alpha = 0.5$ (as in \cite{zhou2011high}). The simulation results are shown in tables~\ref{tab: rd0.1_k100} and \ref{tab: rd0.5_k100}, respectively. The tables show averages and SEs of classification errors in \% over 100 replicates for the three proposed methods with both $``\vee"$ (left) and $``\wedge"$ (right).

\begin{table} [H]
	\caption{AR(1) model with $K = 100$}
	\begin{subtable}{1\textwidth}
		\sisetup{table-format=-1.2}   
		\centering
		\caption{$n = 100$}
		\begin{tabular}{c|c c c }
			\hline
			Ave (SE) 
			& Total (\%) &  Type \uppercase\expandafter{\romannumeral1} (\%)  & Type \uppercase\expandafter{\romannumeral2} (\%) \\
			\hline
			{Lasso}&0.90(0.45); 1.16(0.44)&0.91(0.47); 1.18(0.45)&0.39(0.68); 0.06(0.28)\\
			{Dantzig} &0.84(0.42); 1.26(0.44)&0.85(0.43); 1.28(0.45)&0.38(0.68); 0.12(0.36)
			\\
			{MU} &0.93(0.49); 1.22(0.43)&0.94(0.50); 1.24(0.44)&0.23(0.53); 0.15(0.44)\\
			\hline 
		\end{tabular}
	\end{subtable}
	
	\bigskip
	
	\begin{subtable}{1\textwidth}
		\sisetup{table-format=-1.2}   
		\centering
		\caption{$n = 500$}
		\begin{tabular}{c|c c c }
			\hline
			Ave (SE) 
			& Total (\%) &  Type \uppercase\expandafter{\romannumeral1} (\%)  & Type \uppercase\expandafter{\romannumeral2} (\%) \\
			\hline
			{Lasso}&0.566(0.220); 0.716(0.231)&0.577(0.224); 0.730(0.236)&0(0)\\
			{Dantzig} &0.601(0.208); 0.700(0.255)&0.613(0.212); 0.714(0.260)&0(0)\\
			{MU} &0.574(0.203); 0.676(0.238)&0.586(0.207); 0.690(0.243)&0(0)\\
			\hline 
		\end{tabular}
	\end{subtable}
	
	\bigskip
	
	\begin{subtable}{1\textwidth}
		\sisetup{table-format=-1.2}   
		\centering
		\caption{$n = 1000$}
		\begin{tabular}{c|c c c }
			\hline
			Ave (SE) 
			& Total (\textperthousand) &  Type \uppercase\expandafter{\romannumeral1} (\%)  & Type \uppercase\expandafter{\romannumeral2} (\%) \\
			\hline
			{Lasso}&0.473(0.529); 0.720(0.510)&0.483(0.539); 0.735(0.521)& 0(0) \\
			{Dantzig} &0.501(0.533); 0.749(0.523)&0.511(0.544); 0.764(0.534)& 0(0) \\
			{MU} &0.522(0.532); 0.741(0.501)&0.532(0.543); 0.756(0.511)& 0(0)\\
			\hline 
		\end{tabular}
	\end{subtable}
	
	\label{tab: ar1_k100}
\end{table}

\vspace{6mm}

\begin{table} [H]
	\caption{AR(4) model with $K = 100$}
	\begin{subtable}{1\textwidth}
		\sisetup{table-format=-1.2}   
		\centering
		\caption{$n = 100$}
		\begin{tabular}{c|c c c }
			\hline
			Ave (SE) 
			& Total (\%) &  Type \uppercase\expandafter{\romannumeral1} (\%)  & Type \uppercase\expandafter{\romannumeral2} (\%) \\
			\hline
			{Lasso}&8.18(0.45); 8.26(0.37)&2.37(0.59); 2.27(0.44)&76.0(2.13); 78.3(1.93)\\
			{Dantzig} &8.21(0.39); 8.21(0.33)&2.38(0.51); 2.21(0.43)&76.3(1.96); 78.3(2.23)
			\\
			{MU} &8.33(0.44); 8.28(0.37)&2.57(0.52); 2.33(0.44)&75.5(1.71); 77.7(1.82)\\
			\hline 
		\end{tabular}
	\end{subtable}
	
	\bigskip
	
	\begin{subtable}{1\textwidth}
		\sisetup{table-format=-1.2}   
		\centering
		\caption{$n = 500$}
		\begin{tabular}{c|c c c }
			\hline
			Ave (SE) 
			& Total (\%) &  Type \uppercase\expandafter{\romannumeral1} (\%)  & Type \uppercase\expandafter{\romannumeral2} (\%) \\
			\hline
			{Lasso}&4.69(0.28); 4.76(0.27)&1.22(0.36); 1.30(0.37)&45.2(3.91); 45.2(3.36)\\
			{Dantzig} &4.77(0.28); 4.81(0.27)&1.21(0.45); 1.30(0.38)&46.4(3.87); 45.9(3.61)
			\\
			{MU} &4.74(0.24); 4.79(0.25)&1.18(0.37); 1.30(0.38)&46.3(3.77); 45.6(3.76)
			\\
			\hline 
		\end{tabular}
	\end{subtable}
	
	\bigskip
	
	\begin{subtable}{1\textwidth}
		\sisetup{table-format=-1.2}   
		\centering
		\caption{$n = 1000$}
		\begin{tabular}{c|c c c }
			\hline
			Ave (SE) 
			& Total (\%) &  Type \uppercase\expandafter{\romannumeral1} (\%)  & Type \uppercase\expandafter{\romannumeral2} (\%) \\
			\hline
			{Lasso}&2.82(0.31); 2.77(0.29)&1.37(0.37); 1.34(0.35)&19.8(1.97); 19.5(2.01)\\
			{Dantzig} &2.86(0.25); 2.80(0.29)&1.35(0.30); 1.34(0.35)&20.5(1.96); 19.8(2.07)
			\\
			{MU} &2.87(0.28); 2.79(0.28)&1.37(0.35); 1.31(0.32)&20.4(2.02); 20.1(2.02)
			\\
			\hline 
		\end{tabular}
	\end{subtable}
	
	\label{tab: ar4_k100}
\end{table}

\begin{table} [H]
	\caption{The random precision matrix model with $\alpha = 0.1$ and $K = 100$}
	\begin{subtable}{1\textwidth}
		\sisetup{table-format=-1.2}   
		\centering
		\caption{$n = 100$}
		\begin{tabular}{c|c c c }
			\hline
			Ave (SE) 
			& Total (\%) &  Type \uppercase\expandafter{\romannumeral1} (\%)  & Type \uppercase\expandafter{\romannumeral2} (\%) \\
			\hline
			{Lasso} &10.3(0.68); 9.84(0.65)&4.40(0.59); 3.91(0.61)& 63.9(4.19); 64.3(5.11)\\
			{Dantzig} &10.3(0.70); 9.97(0.68)&4.23(0.74); 4.08(0.56) &66.0(4.32); 64.0(4.40)\\
			{MU}&10.1(0.65); 9.96(0.69)&4.09(0.58); 4.05(0.57)&65.6(4.44); 64.0(4.24)\\
			\hline 
		\end{tabular}
	\end{subtable}
	
	\bigskip
	
	\begin{subtable}{1\textwidth}
		\sisetup{table-format=-1.2}   
		\centering
		\caption{$n = 500$}
		\begin{tabular}{c|c c c }
			\hline
			Ave (SE) 
			& Total (\%) &  Type \uppercase\expandafter{\romannumeral1} (\%)  & Type \uppercase\expandafter{\romannumeral2} (\%) \\
			\hline
			{Lasso} &3.47(0.51); 3.65(0.50)&2.71(0.50); 2.92(0.52)&10.4(2.80); 10.2(2.90)\\
			{Dantzig} &4.02(0.51); 4.12(0.55)&3.11(0.50); 3.43(0.58)&12.2(3.06); 10.2(2.72)\\
			{MU}&4.04(0.45); 4.25(0.61)&3.12(0.50); 3.53(0.64)&12.2(3.27); 10.5(2.39)\\
			\hline 
		\end{tabular}
	\end{subtable}
	
	\bigskip
	
	\begin{subtable}{1\textwidth}
		\sisetup{table-format=-1.2}   
		\centering
		\caption{$n = 1000$}
		\begin{tabular}{c|c c c }
			\hline
			Ave (SE) 
			& Total (\%) &  Type \uppercase\expandafter{\romannumeral1} (\%)  & Type \uppercase\expandafter{\romannumeral2} (\%) \\
			\hline
			{Lasso} &1.34(0.38); 1.47(0.32)&1.35(0.43); 1.50(0.36)&1.27(0.67); 1.32(0.71)\\
			{Dantzig} &1.77(0.38); 1.74(0.38)&1.79(0.43); 1.77(0.42)&1.74(0.87); 1.42(0.70)\\
			{MU}&1.91(0.36); 2.29(0.58)&1.92(0.40); 2.32(0.65)&1.78(0.81); 1.79(0.80)\\
			\hline 
		\end{tabular}
	\end{subtable}
	
	\label{tab: rd0.1_k100}
\end{table}

\vspace{6mm}

\begin{table} [H]
	\caption{The random precision matrix model with $\alpha = 0.5$ and $K = 100$}
	\begin{subtable}{1\textwidth}
		\sisetup{table-format=-1.2}   
		\centering
		\caption{$n = 100$}
		\begin{tabular}{c|c c c }
			\hline
			Ave (SE) 
			& Total (\%) &  Type \uppercase\expandafter{\romannumeral1} (\%)  & Type \uppercase\expandafter{\romannumeral2} (\%) \\
			\hline
			{Lasso} &49.0(0.88); 49.2(0.87)&7.04(1.16); 6.08(1.04)& 90.9(1.40); 92.1(1.30)\\
			{Dantzig} &49.2(0.88); 49.2(0.87) & 6.48(0.92); 6.22(0.90) &91.9(1.10); 92.1(1.09)\\
			{MU}&49.2(0.85); 49.2(0.89) & 6.39(0.88); 6.20(0.86) & 91.9(1.11); 92.1(1.13)\\
			\hline 
		\end{tabular}
	\end{subtable}
	
	\bigskip
	
	\begin{subtable}{1\textwidth}
		\sisetup{table-format=-1.2}   
		\centering
		\caption{$n = 500$}
		\begin{tabular}{c|c c c }
			\hline
			Ave (SE) 
			& Total (\%) &  Type \uppercase\expandafter{\romannumeral1} (\%)  & Type \uppercase\expandafter{\romannumeral2} (\%) \\
			\hline
			{Lasso} &42.5(1.30); 42.6(1.34)&12.6(1.44); 12.8(1.66)&72.2(2.64); 72.3(2.67)\\
			{Dantzig} &44.9(1.32); 44.2(1.25)&13.7(1.73); 14.4(1.72)&76.1(2.81); 73.8(2.51)\\
			{MU}&45.7(1.23); 44.7(1.29)&13.6(1.53); 14.1(1.52)&77.6(2.13); 75.2(2.21)\\
			\hline 
		\end{tabular}
	\end{subtable}
	
	\bigskip
	
	\begin{subtable}{1\textwidth}
		\sisetup{table-format=-1.2}   
		\centering
		\caption{$n = 1000$}
		\begin{tabular}{c|c c c }
			\hline
			Ave (SE) 
			& Total (\%) &  Type \uppercase\expandafter{\romannumeral1} (\%)  & Type \uppercase\expandafter{\romannumeral2} (\%) \\
			\hline
			{Lasso} &33.0(1.43); 32.9(1.43)&13.5(1.58); 13.8(1.58)&52.6(3.17); 52.0(3.00)\\
			{Dantzig} &36.4(1.29); 35.4(1.29)&16.2(1.47); 15.8(1.82)&56.6(2.93); 55.0(3.24)\\
			{MU}&41.2(1.12); 39.7(1.26)&17.5(2.05); 17.8(1.69)&64.8(2.28); 61.7(2.25)\\
			\hline 
		\end{tabular}
	\end{subtable}
	\label{tab: rd0.5_k100}
\end{table}

\section*{Acknowledgment}
This research has partially been supported by the National Science Foundation under Grant No. DMS 1713108.

\section{Appendix: proofs} \label{sec: pfs}
Recall the notation introduced in Section~\ref{subsec: uncertainty}, 
\begin{align*}
	&\wh{\bH} =\bH+\bR \\
	&\bEta^{(t)}=\bX\bbeta+\bm{\epsilon}^{(t)}\;\text{ for }t=1,\cdots,n \\
	&\bm{\epsilon}^{(t)}\sim N(\bm{0},\bSigma) \;\text{ i.i.d. for }t=1,\cdots,n. 
\end{align*}
In the proofs we denote by ``$c$" a positive constant that can be different in each formula. 

\subsection{Proof of Lemma~\ref{lemma: net_etaDiff}}
We first consider the case $\bbeta=\bm{0}$ and show the following result: 
\begin{lemma} \label{lemma: net_b0}
	Let $\sigma=\max_{k\in\wbK}\sqrt{\sigma_{kk}}$. Under the latent block model with $\bbeta=\bm{0}$, for $L>0$, if $m_{\min}\ge 16M^2\log (nK)e^{2L}$,
	\begin{multline*}
		P\Big(\max_{1 \le k \le K, 1 \le t \le n} |\wh \eta_k^{(t)} - \eta_k^{(t)}| < 8Me^L\sqrt{\log(nK)/m_k}\Big) \\
		\geq
		1 -\sqrt{\frac{2}{\pi}}\frac{nK\sigma}{\min\{L, T\}}\exp\Big(-\frac{(\min\{L, T\})^2}{2\sigma^2}\Big) - \Big(\frac{1}{nK}\Big)^{2M^2 - 1}.
	\end{multline*}
\end{lemma}
\begin{proof}
	From Hoeffding's inequality we have, for any $M > 0$, that
	\begin{align*}
		P\big(\sqrt{m_k} | \wt{p}_{k,k}^{(t)} - p_{k,k}^{(t)}| > M\,\big|\,p_{k,k}^{(t)}\big)
		& =
		P\Big(\big|\sum\nolimits_{z[i] =k, z[j] = k} (Y^{(t)}_{ij} - p_{k,k}^{(t)})/m_k\big| > {M}/{\sqrt{m_k}}\, \Big|\, p_{k,k}^{(t)}\Big)\\
		& \leq 
		2\exp\big( - 2m_k \big({M}/{\sqrt{m_k}}\big)^2 \big) = 
		2\exp\big(- 2M^2\big).
	\end{align*}
	Thus, using the fact that, given $\bp = (p_{k,k},\ldots,p_{K,K})^T$, all the $Y_{ij}$ are independent, we obtain
	\begin{align*} 
		P\Big( \max_{1 \leq k \leq K, 1 \leq t \leq n}\sqrt{m_k} & |\wt{p}_{k,k}^{(t)} - p_{k,k}^{(t)}| > M\sqrt{\log(nK)}\,\big|\,\big\{\bp^{(t)}: 1\leq t\leq n\big\}\Big) \\
		& \leq 
		nK\max _{1 \leq k \leq K, 1\le t \leq n}P\left( \sqrt{m_k}|\wt{p}_{k,k}^{(t)} - p_{k,k}^{(t)}|  > M \sqrt{\log (nK)}\, \Big|\, p_{k,k}^{(t)}\right) \\
		& \leq 
		nK\exp\left(-2M^2\log(nK)\right)  = 
		(nK)^{1-2M^2}, 
	\end{align*}
	By integrating over $\{\bp^{(t)}: 1\leq t\leq n\}$, we obtain
	\begin{align}\label{Pineq: net_pDiff}
		P\Big( \max_{1 \leq k \leq K, 1 \leq t \leq n}\sqrt{m_k} & |\wt{p}_{k,k}^{(t)} - p_{k,k}^{(t)}| > M\sqrt{\log(nK)}\Big) \leq (nK)^{1-2M^2}.
	\end{align}
	Note that if $|\eta_k^{(t)}| < T$, we have  
	$|\wh{\eta}_k^{(t)} - \eta_k^{(t)} |\leq |\wt{\eta}_k^{(t)} - \eta_k^{(t)} |$,  and we can write
	\begin{align*}
		\wt{\eta}_k^{(t)} - \eta_k^{(t)} 
		= 
		\log\tfrac{\wt{p}^{(t)}_{k,k}}{ 1 - \wt{p}_{k,k}^{(t)}} - \log\tfrac{p_{k,k}^{(t)}}{1 - p_{k,k}^{(t)}} 
		= 
		\tfrac{1}{\xi_k^{(t)} (1-\xi_k^{(t)})}(\wt{p}_{k,k}^{(t)} - p_{k,k}^{(t)}), 
	\end{align*}
	where $\xi_{k}^{(t)}$ lies between $\wt{p}_{k,k}^{(t)}$ and $p_{k,k}^{(t)}$. Since $|\xi_{k}^{(t)} - p_{k,k}^{(t)}| \leq |\wt{p}_{k,k}^{(t)} - p_{k,k}^{(t)}|$, inequality \eqref{Pineq: net_pDiff} also applies to $|\xi_{k}^{(t)} - p_{k,k}^{(t)}|$, so that
	\begin{align*}
		P\Big(\max _{1 \leq k \leq K, 1\leq t \leq n}\sqrt{m_k} |\xi_{k}^{(t)} - p_{k,k}^{(t)}| > M\sqrt{\log (nK) }\Big) 
		\leq (nK)^{1-2M^2}.
	\end{align*}
	It follows that 
	\begin{multline} \label{Pineq: net_xi}
		P\left( \xi^{(t)}_{k} <  \epsilon \text{ or }\; \xi^{(t)}_{k} > 1- \epsilon,\text{ for all }1 \leq k \leq K, 1 \leq t \leq n\right) \\
		\leq
		P\Big( p_{k,k}^{(t)} < \epsilon + \tfrac{M\sqrt{ \log (nK)}}{\sqrt{m_k}}   \text{ or }\;p_{k,k}^{(t)}  \geq 1 - \epsilon - \tfrac{M\sqrt{ \log (nK)}}{\sqrt{m_k}},\text{ for all } 1 \leq k \leq K,\;1 \leq t \leq n\Big) \\
		+ P\Big(\max _{1 \le k \le K, 1\le t \le n}\sqrt{m_k} |\xi_{k}^{(t)} - p_{k,k}^{(t)}| > M\sqrt{\log(nK)}\Big).
	\end{multline}
	As for $\eta_{k}$, since $\eta_k^{(t)}-\mu_k\sim N(0,\sigma_{kk})$ and  
	\begin{align*}
	\max _{1 \leq k \leq K, 1\leq t \leq n}|\eta_k^{(t)}| \leq \max _{1 \leq k \leq K, 1\leq t \leq n}|\eta_k^{(t)}-\mu_k|+\max_{1\leq k\leq K}|\mu_k|, 
	\end{align*}
    we have, for $C > 1$, that
	\begin{align*}
		P\Big(\max _{1 \leq k \leq K, 1\leq t \leq n}|\eta_k^{(t)}| > \sqrt{2\log C}+\mu_B\Big)
		& \leq 
		P\Big(\max _{1 \leq k \leq K, 1\leq t \leq n}\big|(\eta_{k}^{(t)} - \mu_k)/\sqrt{\sigma_{kk}}\big| > \sqrt{2\log C}/\sigma\Big) \\
		& \leq 2nK\wt{\Phi}\big({\sqrt{2\log C}}/\sigma\big), 
	\end{align*}
	where $\wt{\Phi} = 1-\Phi$ ($\Phi$ is the c.d.f. of $N(0,1)$). Note that 
	\begin{multline*}
		P\Big(\max _{1 \leq k \leq K, 1\leq t \leq n}|\eta_k^{(t)}| > \sqrt{2\log B}\Big) \\
		=
		P\Big(\max _{1 \leq k \leq K, 1\leq t \leq n} p_{k,k}^{(t)} > \frac{1}{1 + e^{-\sqrt{2\log B}}}\Big) + P\Big(\min _{1 \leq k \leq K, 1\leq t \leq n} p_{k,k}^{(t)} < \frac{1}{1 + e^{\sqrt{2\log B}}}\Big). \quad 
	\end{multline*}
	Now we are using the following: 
	\begin{fact}
		For each $c_0>0,$ we can find $x_0=e^{1/c_0^2} > 1$ such that, for $x \ge x_0 > 1,$ we have $e^{-\sqrt{\log x}} \ge x^{-c_0}$. 
	\end{fact}
	\noindent 
	Using this fact, it follows that, for any $L > 0,$ we have that  $\frac{1}{1 + B^{-2/L}} \geq  \frac{1}{1 + e^{-\sqrt{2\log B\,}}}$ for $B\geq e^{L^2/2}$. W.l.o.g. assume that $B > 1.$ Then,
	\begin{align*}
		P\Big(\max _{1 \leq k \leq K, 1\leq t \le n} p_{k,k}^{(t)} > \frac{1}{1 + e^{-\sqrt{2\log B\,}}}\Big) 
		& \geq
		P\Big(\max _{1 \leq k \leq K, 1\leq t \leq n} p_{k,k}^{(t)} > \frac{1}{1 + B^{-2/L}}\Big) \\
		& \geq
		P\Big(\max _{1 \leq k \leq K, 1\leq t \leq n} p_{k,k}^{(t)} > 1 - 2^{-1}B^{-\frac{2}{L}}\Big).
	\end{align*}
	Further, using that $\frac{e^x}{1 + e^x} \ge \frac{1}{2}$ for $x > 0$, we obtain (by using the above fact again) that for any $L > 0$ and $B\geq e^{L^2/2}$,
	\begin{align*}
		P\Big(\min_{1 \leq k \leq K, 1\leq t \leq n} p_{k,k}^{(t)} < \frac{1}{1 + e^{\sqrt{2\log B}}}\Big) 
		& = 
		P\Big(\min_{1 \leq k \leq K, 1\leq t \leq n} p_{k,k}^{(t)} < e^{-\sqrt{2\log B}} \frac{e^{\sqrt{2\log B\,}}}{1 + e^{\sqrt{2\log B\,}}}\Big)\\
		& \geq
		P\Big(\min_{1 \leq k \leq K, 1\leq t \leq n} p_{k,k}^{(t)} < \tfrac{1}{2}B^{-\frac{2}{L}}\Big).
	\end{align*}
	This means that in \eqref{Pineq: net_xi} we can choose $\epsilon = \epsilon(L) = \frac{1}{4} B^{-\frac{2}{L}}$. It follows that  with this choice of $\epsilon$ (for arbitrarily large $L$) and assuming that
	\begin{align}\label{ineq: net_mk}
		\max_{1 \leq k \leq K} \tfrac{M \sqrt{\log (nK)}}{\sqrt{m_k}} \leq \epsilon,
	\end{align}
	we have
	\begin{align*}
		P\big( \xi^{(t)}_{k} <  \epsilon \text{ or } \xi^{(t)}_{k} > 1- \epsilon,\;1 \leq k \leq K, 1 \leq t \leq n\big)
		\leq
		2nK\wt{\Phi}\big({\sqrt{2\log B}}/{\sigma}\big)+ (nK)^{1 - 2M^2}.
	\end{align*}
	Finally, this leads to: let $\epsilon = \epsilon(L) = \frac{1}{4} B^{-\frac{2}{L}}$. If \eqref{ineq: net_mk} holds, then for $B\geq e^{L^2/2}$  
	\begin{align*}
		P\Big(\tfrac{1}{\xi^{(t)}_{k} (1-\xi^{(t)}_{k})} \geq {2}/{\epsilon},\;1 \leq t \leq n, 1 \leq k \leq K  \Big) 
		\leq
		2nK\wt{\Phi}\big({\sqrt{2\log B}}/{\sigma}\big) + (nK)^{1 - 2M^2}.
	\end{align*}
	Then
	\begin{multline*}
		\quad P\Big(\max_{1 \leq k \leq K, 1 \leq t \leq n}\sqrt{m_k} |\wh \eta_k^{(t)} - \eta_k^{(t)}| < 2M\epsilon^{-1}\sqrt{\log(nK)}\Big) \\
		\geq
		1 - 2nK\wt{\Phi}\big(\sigma^{-1}\min\big\{\sqrt{2\log B}, T\big\}\big) - ({nK})^{1 - 2M^2}, \quad 
	\end{multline*}
	i.e. for $L>0$, if $m_{\min}\ge 16M^2\log( nK)B^{4/L}$ and $B\geq e^{L^2/2}$,
	\begin{multline*}
		\quad P\Big(\max_{1 \leq k \leq K, 1 \leq t \leq n}\sqrt{m_k} |\wh \eta_k^{(t)} - \eta_k^{(t)}| <8MB^{\frac{2}{L}}\sqrt{\log(nK)}\Big) \\ 
		\geq
		1 - 2nK\wt{\Phi}\big(\sigma^{-1}\min\big\{\sqrt{2\log B}, T\big\}\big) - ({nK})^{1 - 2M^2}. \quad 
	\end{multline*}
	Choosing $B=e^{L^2/2}$, then for $L>0$,  if $m_{\min}\geq 16M^2\log (nK)e^{2L}$,
	\begin{align*}
		& P\Big(\max_{1 \leq k \leq K, 1 \leq t \leq n} |\wh \eta_k^{(t)} - \eta_k^{(t)}| <8Me^L\sqrt{{\log(nK)}/{m_{\min}}}\,\Big) \\
		& \hspace{2.2cm} \geq 
		1 -\sqrt{\frac{2}{\pi}}\frac{nK\sigma}{\min\{L, T\}}\exp\Big(-\frac{(\min\{L, T\})^2}{2\sigma^2}\Big) - \Big(\frac{1}{nK}\Big)^{2M^2 - 1}.
	\end{align*}
\end{proof}
\noindent
The proof of Lemma~\ref{lemma: net_etaDiff} with $\bbeta \neq \bm{0}$ is similar, but it uses 
\begin{align*}
	P\big(\max _{1 \leq k \leq K, 1\leq t \leq n}|\eta_k^{(t)}| > \sqrt{2\log C}+\mu_B\big) 
	\leq
	nKC^{-1/\sigma^2}/{\sqrt{\pi\log C^{1/\sigma^2}}}, 
\end{align*}
which comes from $\eta_k^{(t)}-\mu_k\sim N(0,\sigma_{kk})$ and 
\begin{align*}
	\max _{1 \leq k \leq K, 1\leq t \leq n}|\eta_k^{(t)}| \leq \max _{1 \leq k \leq K, 1\leq t \leq n}|\eta_k^{(t)}-\mu_k|+\max_{1\leq k\leq K}|\mu_k|.
\end{align*}

\subsection{Proof of Theorem~\ref{corollary: net_mb0_asy}}
We first introduce assumption 
\begin{itemize}
	\item[C0] $|\bR|_\infty < \delta$ for fixed $\delta > 0$. 
\end{itemize}
This assumption makes the proofs more transparent. Intermediate results are using this assumptions. When applying these intermediate results to prove the main results (that are not using assumption C0 explicitly), we will show that C0 holds with sufficiently large probability. The following fact immediately follows from the definition of $\bR$ (see Section~\ref{subsec: uncertainty}):
\begin{fact} \label{claim: mb1}
	Under assumption C0, $\|\wh{\bEta}_k-\bEta_k\|_2\leq \sqrt{n}\delta$.
\end{fact}
\noindent 
We prove a series of results: Theorem~\ref{thm: lasso_type1} and \ref{thm: lasso_type2}, Corollary~\ref{corollary: mb0_asy}, which will then imply Theorem~\ref{corollary: net_mb0_asy}.  The assertion of Theorem~\ref{corollary: net_mb0_asy} then follows from Corollary~\ref{corollary: mb0_asy} together with Lemma~\ref{lemma: net_etaDiff}. The proof is an adaptation of \cite{meinshausen2006high} (proof of Theorem 1), to our more complex situation. Both of the proofs are mainly established with the property of chi-square distribution and the Lasso.  For any $\mathcal{A} \subset \wbK$, let the Lasso estimate $\wh{\btheta}^{a,\mathcal{A},\lambda, lasso}$ of $\btheta^{a,\mathcal{A}}$ be defined as
\begin{align} \label{def: thetaALasso}
	\wh{\btheta}^{a,\mathcal{A},\lambda,lasso}=\argmin_{\btheta: \theta_k=0,\forall k\notin\mathcal{A}}
	\Big(n^{-1}\|(\wh{\bEta}_a - \mu_a\bIn) - (\wh{\bH} - \bIn\bm{\mu})\btheta\|_2^2+\lambda\|\btheta\|_1\Big).
\end{align}
\begin{claim} \label{claim: mb2}
	For problem \eqref{def: thetaALasso}, under assumption C0, for any $q>1$, 
	\begin{align*}
		P\Big(\|\wh{\btheta}^{a,\mathcal{A},\lambda,lasso}\|_1\leq (q+\delta)^2\lambda^{-1}\Big)
		\geq
		1-\exp\Big(-\tfrac{q^2-\sqrt{2q^2-1}}{2}n\Big).
	\end{align*}
\end{claim}
\begin{proof}
	The claim follows directly from the tail bounds of the $\chi^2$-distribution~(see \cite{Laurent2000adaptive}) and the inequality 
	\begin{align*}
		n\lambda\|\wh{\btheta}^{a,\mathcal{A},\lambda,lasso}\|_1
		\leq
		\|\wh{\bEta}_a - \mu_a\bIn\|_2^2
		\leq
		\left(\|\wh{\bEta}_a-\bEta_a\|_2+\|\bEta_a - \mu_a\bIn\|_2\right)^2.
	\end{align*}
\end{proof}
\begin{lemma} \label{lemma: lasso_dual}
	Given $\btheta\in\mathbb{R}^{K}$, let $G(\btheta)$ be a $K$-dimensional vector with elements
	\begin{align*}
		G_b(\btheta) = -2n^{-1}\langle(\wh{\bEta}_a-\mu_a\bIn) - (\wh{\bH} - \bIn\bmu^T)\btheta, \wh{\bEta}_b - \mu_b\bIn\rangle.
	\end{align*}
	A vector $\wh{\btheta}$ with $\wh{\theta}_k=0, \forall k\in\wbK\backslash\mathcal{A}$ is a solution to \eqref{def: thetaALasso} iff for all $b\in\mathcal{A}$, $G_b(\wh{\btheta})=-\sign(\wh{\theta}_b)\lambda$ in case $\wh{\theta}_b\neq0$, and $|G_b(\wh{\btheta})|\leq\lambda$ in case $\wh{\theta}_b=0$. Moreover, if the solution is not unique, and $|G_b(\wh{\btheta})|<\lambda$ for some solution $\wh{\btheta}$, then $\wh{\theta}_b=0$ for all solutions of \eqref{def: thetaALasso}.
\end{lemma}
\noindent
This Lemma is almost the same as Lemma A.1 in \cite{meinshausen2006high} but without normality assumption of $\wh{\bH}$. Since the Gaussian assumption is not needed, the proof is a straightforward adaptation of the proof of Lemma A.1 in \cite{meinshausen2006high}. 
\begin{lemma} \label{lemma: lasso_sign}
	For every $a\in\wbK$, let $\wh{\btheta}^{a,\nb_a,\lambda,lasso}$ be defined as in \eqref{def: thetaALasso}. Let the penalty parameter satisfy $\lambda_n\sim dn^{-(1-\epsilon)/2}$ with some $d>0$ and $\kappa<\epsilon<\xi$. Suppose that assumptions A1 and C0 hold with  $\delta = o(n^{-(4-\xi-3\epsilon)/2})$.  Then there exists $c>0$ so that, for all $a\in\wbK$, 
	\begin{align*}
		P\big(\sign(\wh{\theta}_b^{a,\nb_a,\lambda,lasso})=\sign(\theta_b^a),\; \forall b\in\nb_a\big)
		=
		1 - O\big(\exp(-cn^\epsilon)\big)
		\text{ as }n\rightarrow\infty.
	\end{align*}
\end{lemma}
\begin{proof}
	Using similar notation as in \cite{meinshausen2006high}, we set
	\begin{align} \label{def: thetaNeLasso}
		\wh{\btheta}^{a,\nb_a,\lambda,lasso}=\argmin\limits_{\btheta: \theta_k=0,\forall k\notin\nb_a}\Big(n^{-1}\|(\wh{\bEta}_a-\mu_a\bIn) - (\wh{\bH} - \bIn{\bmu}^T)\btheta\|_2^2+\lambda\|\btheta\|_1\Big), 
	\end{align}
	and for all $a,b\in\wbK$ with $b\in\nb_a$, we let 
	\begin{align} \label{def: thetaOmegaLasso}
		\wt{\btheta}^{a,b,\lambda}(\omega)=\argmin\limits_{\btheta\in\Theta_{a,b}(\omega)}\Big(n^{-1}\|(\wh{\bEta}_a-\mu_a\bIn) - (\wh{\bH} - \bIn{\bmu}^T)\btheta\|_2^2+\lambda\|\btheta\|_1\Big),
	\end{align}
	where
	\begin{align*}
		\Theta_{a,b}(\omega)=\big\{\btheta\in\mathbb{R}^{K(n)}:\theta_b=\omega;\theta_k=0,\forall k\notin\nb_a\big\}.
	\end{align*}
	Setting $\omega=\wh{\theta}_b^{a,\nb_a,\lambda,lasso}$, then  $\wt{\btheta}^{a,b,\lambda}(\omega)= \wh{\btheta}^{a,\nb_a,\lambda,lasso}$,  and by Claim~\ref{claim: mb2} with $q=2$, 
	\begin{align*}
		P\Big(|\wh{\theta}_b^{a,\nb_a,\lambda,lasso}|\leq(2+\delta)^2\lambda^{-1}\Big)
		\geq
		1-\exp\left(-2^{-1}n\right).
	\end{align*}
	Thus, if $\sign(\wh{\theta}_b^{a,\nb_a,\lambda,lasso})\neq\sign(\theta_b^a)$,  with probability at least $1-\exp(2^{-1}n)$, there would exist some $\omega$ with $|\omega| \leq (2+\delta)^2\lambda^{-1}$ so that  $\wt{\btheta}^{a,b,\lambda}(\omega)$ is a solution to \eqref{def: thetaNeLasso} but $\sign(\omega)\sign(\theta_b^a)\leq0$. 
	Without loss of generality, we assume $\theta_b^a>0$ since $\sign(\theta_b^a)\neq0$ for all $b\in\nb_a$. Note that by Lemma~\ref{lemma: lasso_dual}, $\wt{\btheta}^{a,b,\lambda}(\omega)$ can be a solution to \eqref{def: thetaNeLasso} only if $G_b(\wt{\btheta}^{a,b,\lambda}(\omega))\geq-\lambda$ when $\omega\leq0$. This means
	\begin{multline} \label{Pineq: sign}
		\quad P\Big(\sign(\wh{\theta}_b^{a,\nb_a,\lambda,lasso})\neq\sign(\theta_b^a)\Big) \\
		\leq P\Big(\sup_{-\lambda^{-1}(2+\delta)^2\leq\omega\leq0}G_b(\wt{\btheta}^{a,b,\lambda}(\omega))\geq-\lambda\Big)+\exp\left(-2^{-1}n\right). 
	\end{multline}
	Let $\br_a^{\lambda}(\omega)=(\wh{\bEta}_a - \mu_a\bIn) - (\wh{\bH} - \bIn{\bmu}^T)\wt{\btheta}^{a,b,\lambda}(\omega)$
	and write $\wh{\eta}_b$, $\eta_b$ as
	\begin{align*} 
		\wh{\eta}_b - \mu_b = \sum_{k\in\nb_a\backslash\{b\}}\theta_k^{b,\nb_a\backslash\{b\}}(\wh{\eta}_k - \mu_k) + \wh{w}_b 
	\end{align*}
	and 
	\begin{align*}
		\eta_b - \mu_b = \sum_{k\in\nb_a\backslash\{b\}}\theta_k^{b,\nb_a\backslash\{b\}}(\eta_k - \mu_k) + w_b;  
	\end{align*}
	and $\wh{\eta}_a$, $\eta_a$ as
	
	\begin{align*}
		\wh{\eta}_a - \mu_a = \sum_{k\in\nb_a}\theta_k^a(\wh{\eta}_k - \mu_k) + \wh{v}_a
		\quad\text{and}\quad
		\eta_a - \mu_a = \sum_{k\in\nb_a}\theta_k^a(\eta_k - \mu_k) + v_a, 
	\end{align*}
	where $v_a$ and $w_b$ are independent normally distributed random variables with variances $\sigma_{-a}^2$ and $\sigma_{w,b}^2$, respectivley, and $0<\upsilon^2\leq\sigma_{-a}^2, \sigma_{w,b}^2\leq1$ by A1.2. Now we can write 
	\begin{align} \label{eq: etaVW}
		\eta_a - \mu_a = \sum_{k\in\nb_a\backslash\{b\}}(\theta_k^a+\theta_b^a\theta_k^{b,\nb_a\backslash\{b\}})(\eta_k - \mu_k) + \theta_b^aw_b+v_a. 
	\end{align}
	As in \cite{meinshausen2006high}, split the $n$-dimensional vector $\wh{\bw}_b$ of observations of $\wh{w}_b$,  and also the vector $\bw_b$ of observations of $w_b$ into the sum of two vectors, respectively, 
	\begin{align*} 
		\wh{\bw}_b=\wh{\bw}_b^{\bot}+\wh{\bw}_b^{||}
		\quad\text{and}\quad
		\bw_b=\bw_b^{\bot}+\bw_b^{||},
	\end{align*}
	where $\bw_b^{||}$ and $\wh{\bw}_b^{||}$ are contained in the at most $(|\nb_a|-1)$-dimensional space $\mathbb{W}^{||}$ spanned by the vectors $\{\bEta_k: k\in\nb_a\backslash\{b\}\}$, while  $\bw_b^{\bot}$ and $\wh{\bw}_b^{\bot}$ are contained in the orthogonal complement $\mathbb{W}^{\bot}$ of $\mathbb{W}^{||}$ in $\mathbb{R}^n$. Following the proof of \cite{meinshausen2006high} (Appendix, Lemma A.2), we have
	\begin{align} \label{ineq: G}
		G_b(\wt{\btheta}^{a,b,\lambda}(\omega))
		\leq 
		-2n^{-1}\langle\br_a^{\lambda}(\omega),\wh{\bw}_b\rangle+\lambda\vartheta, 
	\end{align}
	where $2n^{-1}\langle\br_a^{\lambda}(\omega),\wh{\bw}_b\rangle$ can be written as $2n^{-1}\langle\br_a^{\lambda}(\omega),\wh{\bw}_b^{\bot}\rangle+2n^{-1}\langle\br_a^{\lambda}(\omega),\wh{\bw}_b^{||}\rangle$.
	By definition of $\br_a^{\lambda}(\omega)$, the orthogonality property of $\wh{\bw}_b^{\bot}$, and \eqref{eq: etaVW},
	\begin{align} \label{ineq: main_Psign}
		2n^{-1}\langle\br_a^{\lambda}(\omega),\wh{\bw}_b^{\bot}\rangle
		&=
		2n^{-1}\langle(\wh{\bEta}_a - \mu_a\bIn) - (\wh{\bH} - \bIn{\bmu}^T)\wt{\btheta}^{a,b,\lambda}(\omega),\wh{\bw}_b^{\bot}\rangle\nonumber\\
		&=
		2n^{-1}\langle(\bEta_a - \mu_a\bIn)-(\bH - \bIn{\bmu}^T)\wt{\btheta}^{a,b,\lambda}(\omega),\wh{\bw}_b^{\bot}\rangle\nonumber\\
		&\qquad\quad + 2n^{-1}\langle(\wh{\bEta}_a-\bEta_a)-(\wh{\bH}-\bH)\wt{\btheta}^{a,b,\lambda}(\omega),\wh{\bw}_b^{\bot}\rangle\nonumber\\
		&=
		2n^{-1}(\theta_b^a-\omega)\langle\bw_b^{\bot},\wh{\bw}_b^{\bot}\rangle+2n^{-1}\langle\bv_a,\wh{\bw}_b^{\bot}\rangle\nonumber\\
		&\qquad\quad + 2n^{-1}\langle(\wh{\bEta}_a-\bEta_a)-(\wh{\bH}-\bH)\wt{\btheta}^{a,b,\lambda}(\omega),\wh{\bw}_b^{\bot}\rangle\nonumber\\
		&\geq
		2n^{-1}(\theta_b^a-\omega)\langle\bw_b^{\bot},\wh{\bw}_b^{\bot}\rangle-|2n^{-1}\langle\bv_a,\wh{\bw}_b^{\bot}\rangle|\nonumber\\
		&\qquad\quad
		- 2n^{-1}\|(\wh{\bEta}_a-\bEta_a)-(\wh{\bH}-\bH)\wt{\btheta}^{a,b,\lambda}(\omega)\|_2\|\wh{\bw}_b^{\bot}\|_2.
	\end{align}
	\begin{claim} \label{claim: mb3}
		$\|\wh{\bw}_b-\bw_b\|_2\leq (\vartheta+1)\sqrt{n}\delta$ under assumption C0.
	\end{claim}
	\begin{proof}
		By assumption $\|\btheta^{b,\nb_a\backslash\{b\}}\|_1\leq\vartheta$, an application of the triangle inequality and Claim~\ref{claim: mb1} gives the assertion. 
	\end{proof}
	In order to estimate the second term $|2n^{-1}\langle\bv_a,\wh{\bw}_b^{\bot}\rangle|$, we first consider $|2n^{-1}\langle\bv_a,\bw_b^{\bot}\rangle|,$ which has already been estimated in \cite{meinshausen2006high}: for every $g>0$, there exists some $c=c(g,d)>0$ so that, 
	\begin{align} \label{Pineq: VW1}
		P\left(|2n^{-1}\langle\bv_a,\bw_b^{\bot}\rangle|\geq g\lambda\right)
		\leq 
		P\left(|2n^{-1}\langle\bv_a,\bw_b\rangle|\geq g\lambda\right) 
		= O(\exp(-cn^{\epsilon})) \quad \text{as }n\rightarrow \infty.
	\end{align}
	Then for the difference $||\langle\bv_a,\wh{\bw}_b^{\bot}\rangle|-|\langle\bv_a,\bw_b^{\bot}\rangle||$, we have
	\begin{align} \label{Pineq: VW2}
		P\left(2n^{-1}| |\langle\bv_a,\wh{\bw}_b^{\bot}\rangle|-|\langle\bv_a,\bw_b^{\bot}\rangle||\leq 4(\vartheta+1)\delta\right)
		\geq 
		1 - \exp(-2^{-1}n), 
	\end{align}
	which follows from the inequality
	\begin{align*}
		\big||\langle\bv_a,\wh{\bw}_b^{\bot}\rangle|-|\langle\bv_a,\bw_b^{\bot}\rangle|\big|
		\leq
		\|\bv_a\|_2\|\wh{\bw}_b^{\bot}-\bw_b^{\bot}\|_2
		\leq
		(\vartheta+1)\sqrt{n}\delta\|\bv_a\|_2
	\end{align*}
	together with $\|\bv_a\|_2\sim\sigma_{-a}\sqrt{\chi_n^2}$. 
	Thus, by \eqref{Pineq: VW1} and \eqref{Pineq: VW2},
	\begin{align} \label{Pineq: VW}
		P\left(|2n^{-1}\langle\bv_a,\wh{\bw}_b^{\bot}\rangle|\geq g\lambda+4(\vartheta+1)\delta\right)
		= O(\exp(-cn^{\epsilon}))
		\quad\text{as }n\rightarrow\infty.
	\end{align}
	Similarly, we have 
	\begin{align*}
		P\left(2n^{-1}|\langle\bw_b^{\bot},\wh{\bw}_b^{\bot}\rangle-\langle\bw_b^{\bot},\bw_b^{\bot}\rangle|\leq 4(\vartheta+1)\delta\right) 
		\geq 
		1-\exp\left(-2^{-1}n\right).
	\end{align*}
	Note that $\sigma_{w,b}^{-2}\langle\bw_b^{\bot},\bw_b^{\bot}\rangle$ follows a $\chi^2_{n-|\nb_a|+1}$ distribution for $n \ge |\nb_a|$. Using again the tail bound of the $\chi^2$-distribution from \cite{Laurent2000adaptive}, we obtain with assumption A.1.3.(a) and $\sigma_{w,b}^2\geq \upsilon^2$, that there exists $n_0$ so that for $n>n_0$,
	\begin{align*}
		P(2n^{-1}\langle\bw_b^{\bot},\bw_b^{\bot}\rangle>\upsilon^2)
		\geq 
		1-\exp\left(-32^{-1}n\right).
	\end{align*}
	It follows that, 
	\begin{align} \label{Pineq: W}
		P\left(2n^{-1}\langle\bw_b^{\bot},\wh{\bw}_b^{\bot}\rangle>\upsilon^2-4(\vartheta+1)\delta\right) = O(\exp(-cn^{\epsilon}))
		\quad\text{as }n\rightarrow\infty.
	\end{align}
	For the third term of \eqref{ineq: main_Psign}, note that by definition of $\wt{\btheta}^{a,b,\lambda}(\omega)$, 
	\begin{align*}
		\|(\wh{\bEta}_a-\bEta_a)-(\wh{\bH}-\bH)\wt{\btheta}^{a,b,\lambda}(\omega)\|_2
		&\leq
		\|\wh{\bEta}_a-\bEta_a\|_2+\sum_{k\in\nb_a}|\wt{\theta}^{a,b,\lambda}_k(\omega)|\|\wh{\bEta}_k-\bEta_k\|_2 \nonumber \\
		&\leq
		\sqrt{n}\delta\big(1+\|\wt{\theta}^{a,b,\lambda}(\omega)\|_1\big)
	\end{align*}
	and we also have 
	\begin{align*}
		\|\wt{\btheta}^{a,b,\lambda}(\omega)\|_1-|\omega|
		& \leq
		n^{-1}\|(\wh{\bEta}_a - \mu_a\bIn)-\omega(\wh{\bEta}_b - \mu_b\bIn)\|_2^2\lambda^{-1} \\
		& \leq
		\left(n^{-1/2}\|\bEta_a - \mu_a\bIn\|_2+|\omega|n^{-1/2}\|\bEta_b - \mu_b\bIn\|_2+\delta(1+|\omega|)\right)^2\lambda^{-1}. 
	\end{align*}
	Together with $\|\wh{\bw}_b\|_2\leq\|\bw_b\|_2+\|\wh{\bw}_b-\bw_b\|
	\leq
	\|\bw_b\|_2+(\vartheta+1)\sqrt{n}\delta$, and the property of the $\chi^2$-distribution, we have with probability at least
	$1-3\exp\left(-2^{-1}n\right)$, 
	\begin{multline} \label{ineq: Weta}
		2n^{-1}\|(\wh{\bEta}_a-\bEta_a)-(\wh{\bH}-\bH)\wt{\btheta}^{a,b,\lambda}(\omega)\|_2\|\wh{\bw}_b\|_2\\
		\leq 2\left(1+|\omega|+(1+|\omega|)^2(2+\delta)^2\lambda^{-1}\right)\left(2+(\vartheta+1)\delta\right)\delta.
	\end{multline}
	Using \eqref{ineq: main_Psign}, \eqref{Pineq: VW}, \eqref{Pineq: W} and \eqref{ineq: Weta}, we obatin that with probability $1-O(\exp(-cn^\epsilon))$, as $n\rightarrow\infty$,  
	\begin{multline*} 
		2n^{-1}\langle\br_a^{\lambda}(\omega),\wh{\bw}_b^{\bot}\rangle
		\geq 
		(\theta_b^a-\omega)\left(\upsilon^2-4(\vartheta+1)\delta\right)-g\lambda-4(\vartheta+1)\delta \\
		-2\left(1+|\omega|+(1+|\omega|)^2(2+\delta)^2\lambda^{-1}\right)\left(2+(\vartheta+1)\delta\right)\delta.
	\end{multline*}
	Moreover, as will be shown in Lemma~\ref{lemma: lasso_ineq}, there exists $n_g=n(g)$ so that, for all $n\geq n_g$, 
	\begin{multline*} 
		P\Big(\inf_{\omega\leq0}\{2n^{-1}\langle\br_a^{\lambda}(\omega),\wh{\bw}_b^{||}\rangle/(1+|\omega|)\}
		\geq
		-2\left(g\lambda+(\vartheta+1)\delta\right)(2+\delta)\Big) \\
		\geq 
		1-2\exp\left(-2^{-1}n\right) - \exp\left(-4^{-1}g^2n\lambda^2\right).
	\end{multline*}
	Thus, with probability $1-O(\exp(-cn^\epsilon))$, as $n\rightarrow\infty$, 
	\begin{align} \label{ineq: RW}
		2n^{-1}\langle\br_a^{\lambda}(\omega),\wh{\bw}_b\rangle
		\geq
		& (\theta_b^a-\omega)\left(\sigma_b^2-4(\vartheta+1)\delta\right)-g\lambda-4(\vartheta+1)\delta \nonumber\\
		-&
		2\left(1+|\omega|+(1+|\omega|)^2(2+\delta)^2\lambda^{-1}\right)\left(2+(\vartheta+1)\delta\right)\delta \nonumber\\
		- & 2\left(1+|\omega|\right)\left(g\lambda+(\vartheta+1)\delta\right)(2+\delta).
	\end{align}
	Note that $\lambda\sim dn^{-(1-\epsilon)/2}$ with $\epsilon < \xi$, and by A1.2 and A1.4, we have
	\begin{align*}
		|\theta_b^a|
		=
		|\pi_{ab}|\sqrt{\frac{\var(\eta_b|\bEta_{\wbK\backslash\{b\}})}{\var(\eta_a|\bEta_{\wbK\backslash\{a\}})}} \geq v\pi_{ab} \geq cvn^{-(1-\xi)/2}. 
	\end{align*}
	Together with \eqref{ineq: RW}, for $\delta=o(n^{-(4-\xi-3\epsilon)/2})$, we have for any $l>0$ that  
	\begin{align*} 
		P\Big(\inf_{-(2+\delta)^2\lambda^{-1}\leq\omega\leq0}\{2n^{-1}\langle\br_a^{\lambda}(\omega),\wh{\bw}_b\rangle\}>l\lambda\Big)
		= 1 - O(\exp(-cn^\epsilon))
		\quad\text{ as }n\rightarrow\infty.
	\end{align*}
	Choosing $l=\vartheta+1$ and using \eqref{ineq: G}, we have 
	\begin{align*}
		P\Big(\sup_{-(2+\delta)^2\lambda^{-1}\leq\omega\leq0}G_b(\wt{\btheta}^{a,b,\lambda}(\omega))<-\lambda\Big)
		= 1-O(\exp(-cn^\epsilon))
		\quad \text{as }n\rightarrow\infty.
	\end{align*}
	Then, by Bonferroni's inequality, assumption A1.3.(a) and \eqref{Pineq: sign}, 
	\begin{align*}
		P\big(\sign(\wh{\theta}_b^{a,\nb_a,\lambda,lasso})=\sign(\theta_b^a),\,\forall b\in\nb_a\big) 
		= 
		1 - O\big(\exp(-cn^\epsilon) 
		\quad \text{as }n\rightarrow\infty. 
	\end{align*}
\end{proof}

\begin{lemma} \label{lemma: lasso_ineq}
	Under assumption C0, for any $g>0$, there exists $n_g=n(g)$ so that, for all $n\geq n_g$, 
	\begin{multline*}
		P\Big(\inf_{\omega\leq0}\{2n^{-1}\langle\br_a^{\lambda}(\omega),\wh{\bw}_b^{||}\rangle/(1+|\omega|)\}
		\geq
		-2\left(g\lambda+(\vartheta+1)\delta\right)(2+\delta)\Big) \\
		\geq 
		1-2\exp\left(-2^{-1}n\right) - \exp\left(4^{-1}g^2n\lambda^2\right).
	\end{multline*}
\end{lemma}
\begin{proof}
	Again following similar arguments as in \cite{meinshausen2006high} (Appendix, Lemma A.3), we have
	\begin{align*}
		|2n^{-1}\langle\br_a^{\lambda}(\omega),\wh{\bw}_b^{||}\rangle|/(1+|\omega|)
		\leq
		2n^{-1/2}\|\wh{\bw}_b^{||}\|_2\frac{n^{-1/2}\|\br_a^{\lambda}(\omega)\|_2}{1+|\omega|}
	\end{align*}
	and 
	\begin{align*}
		P\Big(\sup_{\omega\in\mathbb{R}}\frac{n^{-1/2}\|\br_a^{\lambda}(\omega)\|_2}{1+|\omega|} > 2+\delta\Big)
		&\leq 
		P\left(n^{-1/2}\max\{\|\wh{\bEta}_a - \mu_a\bIn\|_2,\|\wh{\bEta}_b - \mu_b\bIn\|_2\}>2+\delta\right)\nonumber\\
		&\leq 
		2\exp\left(-2^{-1}n\right).
	\end{align*}
	The last inequality uses that $\|\bEta_k - \mu_k\bIn\|_2^2\sim \chi_n^2$ and $\|\wh{\boldsymbol \eta}_k  - {\boldsymbol \eta}_k\|_2 \leq \sqrt{n}\delta$.
	Note that $\sigma_{w,b}^{-2}\langle\bw_b^{||},\bw_b^{||}\rangle$ follows a $\chi^2_{|\nb_a|-1}$ distribution for large $n$ and $|\nb_a|=o(n\lambda^2)$, and thus for any $g>0$, there exists $n_g=n(g)$ so that for all $n\geq n_g$, 
	\begin{align*}
		P\big(n^{-1/2}\|\bw_b^{||}\|_2>g\lambda\big)
		\leq
		\exp\left(-4^{-1}g^2n\lambda^2\right).
	\end{align*}
	Together with Claim~\ref{claim: mb3}, for any $g>0$, there exists $n_g=n(g)$ so that, for all $n\geq n_g$, 
	\begin{align*}
		P\big(n^{-1/2}\|\wh{\bw}_b^{||}\|_2 > g\lambda+(\vartheta+1)\delta\big)  
		\leq
		\exp\left(-4^{-1}g^2n\lambda^2\right), 
	\end{align*}
	and thus, 
	\begin{multline*}
		P\Big(\sup_{\omega\in\mathbb{R}}\big\{|2n^{-1}\langle\br_a^{\lambda}(\omega),\wh{\bw}_b^{||}\rangle|/\left(1+|\omega|\right)\big\}
		\leq 
		2\left(g\lambda+(\vartheta+1)\delta\right)(2+\delta)\Big) \\
		\geq 
		1 - \exp\left(-4^{-1}g^2n\lambda^2\right) - 2\exp\left(-2^{-1}n\right). 
	\end{multline*}
\end{proof}

\begin{theorem} \label{thm: lasso_type1}
	Assume that A1 holds and that $\bmu$ is known. Let the penalty parameter satisfy $\lambda_n\sim dn^{-(1-\epsilon)/2}$ with $d>0$ and $\kappa<\epsilon<\xi$. If, in addition, C0 holds with $\delta=o(n^{\min\{-(4-\xi-3\epsilon)/2, \epsilon-\kappa-1\}})$, then for all $a\in\wbK$, 
	\begin{align*}
		P\big(\wh{\nb}_a^{\lambda,lasso}\subseteq\nb_a\big)
		= 1 - O\big(\exp(-cn^\epsilon)\big)
		\quad
		\text{as }n\rightarrow\infty.
	\end{align*} 
\end{theorem}
\begin{proof}
	Following the proof of Theorem 1, \cite{meinshausen2006high}, we have 
	\begin{align*}
		P\big(\wh{\nb}_a^{\lambda,lasso}\subseteq\nb_a\big)
		= 1 - P\big(\exists b\in\wbK\backslash\cl_a: \wh{\theta}_b^{a,\lambda,lasso} \neq 0\big), 
	\end{align*}
	and
	\begin{align*} 
		P\left(\exists b\in\wbK\backslash\cl_a:\wh{\theta}_b^{a,\lambda,lasso}\neq0\right)
		\leq P\Big(\max_{b\in\wbK\backslash\cl_a}|G_b(\wh{\btheta}^{a,\nb_a,\lambda,lasso})| \geq \lambda\Big),
	\end{align*}
	where
	\begin{align*}
		G_b\big(\wh{\btheta}^{a,\nb_a,\lambda,lasso}\big)
		= -2n^{-1}\big\langle(\wh{\bEta}_a - \mu_a\bIn)-(\wh{\bH} - \bIn\bmu^T)\wh{\btheta}^{a,\nb_a,\lambda,lasso}, \wh{\bEta}_b - \mu_b\bIn\big\rangle.
	\end{align*}
	For any $b\in\wbK\backslash\cl_a$, write
	\begin{align} \label{eq: etaNeV}
		\wh{\eta}_b - \mu_b = \sum_{k\in\nb_a}\theta_k^{b,\nb_a}(\wh{\eta}_k  - \mu_k) + \wh{v}_b
		\text{ and }
		\eta_b - \mu_b = \sum_{k\in\nb_a}\theta_k^{b,\nb_a}(\eta_k - \mu_k) + \wt{v}_b,
	\end{align}
	where $\wt{v}_b \sim N(0,\sigma_{v,b}^2)$ with $\upsilon^2\leq\sigma_{v,b}^2\leq1$ and is independent of $\{\eta_k: k\in\cl_a\}$.
	\begin{claim} \label{claim: mb7}
		Under assumption C0, for any $q>1$, with probability at least $1  -(|\nb_a|+2)\exp\Big(-\frac{q^2-\sqrt{2q^2-1}}{2}n\Big)$, 
		\begin{multline*}
			\big|\langle\wh{\bEta}_a - \mu_a\bIn - (\wh{\bH} - \bIn\bmu^T)\wh{\btheta}^{a,\nb_a,\lambda,lasso}, \wh{\bv}_b\rangle - \langle\bEta_a - \mu_a\bIn - (\bH - \bIn\bmu^T)\wh{\btheta}^{a,\nb_a,\lambda,lasso},\wt{\bv}_b\rangle\big|\\
			\leq
			n\left(\lambda^{-1}(q+\delta)+1\right)\left(1+\upsilon^{-1}|\nb_a|\right)(2q+\delta)\delta;
		\end{multline*}
		and with probability at least $1 -(|\nb_a|+1)\exp\Big(-\frac{q^2-\sqrt{2q^2-1}}{2}n\Big)$,
		\begin{multline*}
			\big|\big(\|(\wh{\bEta}_a - \mu_a\bIn) - (\wh{\bH} - \bIn\bmu^T)\wh{\btheta}^{a,\nb_a,\lambda,lasso}\|_2^2-\|\wh{\bEta}_a - \mu_a\bIn\|_2^2\big) - \\ \big(\|(\bEta_a - \mu_a\bIn)-(\bH - \bIn\bmu^T)\wh{\btheta}^{a,\nb_a,\lambda,lasso}\|_2^2-\|\bEta_a - \mu_a\bIn\|_2^2\big)\big|\\
			\leq
			n\big((\lambda^{-1}(q+\delta)+1)^2+1\big)(2q+\delta)\delta.
		\end{multline*}
	\end{claim}
	\begin{proof}
		Using triangle inequality and Cauchy's inequality, 
		\begin{align} \label{ineq: etaDiffV}
		& 
		\big|(\langle\wh{\bEta}_a - \mu_a\bIn) - (\wh{\bH} - \bIn\bmu^T)\wh{\btheta}^{a,\nb_a,\lambda,lasso}, \wh{\bv}_b\rangle \nonumber\\
		& \hspace{2in} 
		- \langle(\bEta_a - \mu_a\bIn)- (\bH - \bIn\bmu^T)\wh{\btheta}^{a,\nb_a,\lambda,lasso},\wt{\bv}_b\rangle\big| \nonumber\\
		\leq &
		\Big(\sum_{k\in\nb_a}|\wh{\theta}^{a,\nb_a,\lambda,lasso}_k|\|\wh{\bEta}_k-\bEta_k\|_2+\|\wh{\bEta}_a-\bEta_a\|_2\Big)\|\wh{\bv}_b\|_2  \nonumber\\
		& \hspace{2in} +\|\bEta_a - \mu_a\bIn) - (\bH - \bIn\bmu^T)\wh{\btheta}^{a,\nb_a,\lambda,lasso}\|_2\|\wh{\bv}_b-\wt{\bv}_b\|_2\nonumber\\
		\leq &
		\Big(\big(\sum_{k\in\nb_a}|\wh{\theta}^{a,\nb_a,\lambda,lasso}_k|+1\big)\sqrt{n}\delta+\sum_{k\in\nb_a}|\wh{\theta}^{a,\nb_a,\lambda,lasso}_k|\|\bEta_k - \mu_k\bIn\|_2 \nonumber\\
		& \hspace{1in} +\|\bEta_a - \mu_a\bIn\|_2\Big)\|\wh{\bv}_b-\wt{\bv}_b\|_2 
		+ \Big(\sum_{k\in\nb_a}|\wh{\theta}^{a,\nb_a,\lambda,lasso}_k|+1\Big)\sqrt{n}\delta\|\wt{\bv}_b\|_2.
		\end{align}
		By definition of $\wh{v}_b$ and $\wt{v}_b$ in \eqref{eq: etaNeV}, and by using the triangle inequality, 
		\begin{align*}
			\|\wh{\bv}_b-\wt{\bv}_b\|_2
			\leq \Big(\sum_{k\in\nb_a}|\theta_k^{b,\nb_a}|+1\Big)\sqrt{n}\delta.
		\end{align*}
		Moreover, by the definition of partial correlation and assumption A1.4,
		\begin{align*}
			1 \geq |\pi_{bk}^{b,\nb_a}| = |\theta_k^{b,\nb_a}|\sqrt{\frac{\var(\eta_b|\bEta_{\nb_a})}{\var(\eta_k|\bEta_{\{b\}\cup\nb_a\backslash\{k\}})}} \geq v|\theta_k^{b,\nb_a}|, 
		\end{align*}
		and thus 
		\begin{align} \label{ineq: DiffV}
			\|\wh{\bv}_b-\wt{\bv}_b\|_2
			\leq
			(\upsilon^{-1}|\nb_a|+1)\sqrt{n}\delta.
		\end{align}
		Using \eqref{ineq: etaDiffV}, \eqref{ineq: DiffV}, Claim~\ref{claim: mb2} and the property of chi-square distribution, with probability $1  -(|\nb_a|+2)\exp\Big(-\frac{q^2-\sqrt{2q^2-1}}{2}n\Big)$, 
		\begin{multline*}
			\big|\langle(\wh{\bEta}_a - \mu_a\bIn)-(\wh{\bH} - \bIn\bmu^T)\wh{\btheta}^{a,\nb_a,\lambda,lasso}, \wh{\bv}_b\rangle - \langle(\bEta_a - \mu_a\bIn) - (\bH - \bIn\bmu^T)\wh{\btheta}^{a,\nb_a,\lambda,lasso},\wt{\bv}_b\rangle\big| \\
			\leq
			n\left((q+\delta)\lambda^{-1}+1\right)\left(1+\upsilon^{-1}|\nb_a|\right)(2q+\delta)\delta.
		\end{multline*}
		The second part follows similarly. 
	\end{proof}
	\noindent
	Lemma~\ref{lemma: lasso_dual}, \ref{lemma: lasso_sign} and assumption A1.5 imply that, for any  $\delta=o(n^{-(4-\xi-3\epsilon)/2})$,
	there exists $c>0$ so that for all $a\in\wbK$ and $b\in \wbK\backslash\cl_a$, 
	\begin{multline} \label{eq: Peq_G}
		P\Big(|G_b(\wh{\btheta}^{a,\nb_a,\lambda,lasso})|
		\leq\varrho\lambda + |2n^{-1}\langle(\wh{\bEta}_a - \mu_a\bIn)-(\wh{\bH} - \bIn\bmu^T)\wh{\btheta}^{a,\nb_a,\lambda,lasso},\wh{\bv}_b\rangle|\Big) \\
		=
		1-O(\exp(-cn^\epsilon)) 
		\quad\text{as }n\rightarrow\infty.
	\end{multline}
	Now we need to estimate $|2n^{-1}\langle(\wh{\bEta}_a - \mu_a\bIn)-(\wh{\bH} - \bIn\bmu^T)\wh{\btheta}^{a,\nb_a,\lambda,lasso},\wh{\bv}_b\rangle|$. 
	Note that by Claim~\ref{claim: mb7}, for any $\delta=O(1)$, there exists some constant $B>0$ and $c>0$ so that for $n\rightarrow\infty$, with probability $1-O(\exp(-cn^\epsilon))$,
	\begin{multline} \label{ineq: PetaVDiff}
		\big|\langle(\wh{\bEta}_a - \mu_a\bIn)-(\wh{\bH} - \bIn\bmu^T)\wh{\btheta}^{a,\nb_a,\lambda,lasso},\wh{\bv}_b\rangle - \langle(\bEta_a - \mu_a\bIn) - (\bH - \bIn\bmu^T)\wh{\btheta}^{a,\nb_a,\lambda,lasso},\wt{\bv}_b\rangle\big| \\ 
		\leq 
		Bn^{3/2-\epsilon/2+\kappa}\delta.
	\end{multline}
	We already know that $\wt{v}_b{\perp\!\!\!\perp}\bEta_{\cl_a}$\footnote{This follows from the Markov properties of the conditional independence graph and the contraction property of conditional independence.} and $\wh{\bEta}_{\cl_a}{\perp\!\!\!\perp}\bEta_{\wbK\backslash\cl_a}|\bEta_{\cl_a}$, we have $\wt{v}_b{\perp\!\!\!\perp}\wh{\bEta}_{\cl_a}|\bEta_{\cl_a}$ since $\wt{v}_b=\eta_b-\sum_{k\in\nb_a}\theta_k^{b,\nb_a}\eta_k$. Thus, $\wt{v}_b{\perp\!\!\!\perp}\{\bEta_{\cl_a}, \wh{\bEta}_{\cl_a}\}$. Here ${\perp\!\!\!\perp}$ denotes independence. Conditional on $\{\bH_{\cl_a}, \wh{\bH}_{\cl_a}\}=\{\bEta_k, \wh{\bEta}_k: k\in\cl_a\}$, the random variable
	\begin{align*} 
		\langle(\bEta_a - \mu_a\bIn) - (\bH - \bIn\bmu^T)\wh{\btheta}^{a,\nb_a,\lambda,lasso}, \wt{\bv}_b\rangle
	\end{align*}
	is normally distributed with mean zero and variance 
	$$
	\sigma_{v,b}^2\|(\bEta_a - \mu_a\bIn) - (\bH - \bIn\bmu^T)\wh{\btheta}^{a,\nb_a,\lambda,lasso}\|_2^2. 
	$$
	By definition of $\wh{\btheta}^{a,\nb_a,\lambda,lasso}$,
	\begin{align*}
		\|(\wh{\bEta}_a - \mu_a\bIn) - (\wh{\bH} - \bIn\bmu^T)\wh{\btheta}^{a,\nb_a,\lambda,lasso}\|_2
		\leq
		\|\wh{\bEta}_a - \mu_a\bIn\|_2; 
	\end{align*}
	and by Claim~\ref{claim: mb7}, for $\delta=O(1)$, there exists constant $c, B>0$ so that, with probability $1-O(\exp(-cn^\epsilon))$, as $n \rightarrow \infty$, as $n \rightarrow \infty$, 
	\begin{multline*}
		\big|\big(\|(\wh{\bEta}_a - \mu_a\bIn) - (\wh{\bH} - \bIn\bmu^T)\wh{\btheta}^{a,\nb_a,\lambda,lasso}\|_2 - \|\wh{\bEta}_a - \mu_a\bIn\|_2^2\big) \\ 
		-(\|(\bEta_a - \mu_a\bIn) - (\bH - \bIn\bmu^T)\wh{\btheta}^{a,\nb_a,\lambda,lasso}\|_2^2-\|\bEta_a - \mu_a\bIn\|_2^2)\big| 
		\leq 
		Bn^{2-\epsilon}\delta.
	\end{multline*}
	Futhermore, for $\delta=O(n^{\min\{0,2\tau+\epsilon-2\}})$, there exists $c_t>0$, such that for  $t_a = c_tn^{\tau}$,
	\begin{multline*}
		P\left(\|(\bEta_a - \mu_a\bIn) - (\bH - \bIn\bmu^T)\wh{\btheta}^{a,\nb_a,\lambda,lasso}\|_2^2 \leq \|\bEta_a - \mu_a\bIn\|_2^2+ t_a^2\right) \\
		= 1-O(\exp(-cn^\epsilon)) \quad \text{ as }n\rightarrow\infty, 
	\end{multline*}
	thus, $|2n^{-1}\langle(\bEta_a - \mu_a\bIn) - (\bH - \bIn\bmu^T)\wh{\btheta}^{a,\nb_a,\lambda,lasso},\wt{\bv}_b\rangle|$
	is stochastically smaller than $|2n^{-1}(\langle\bEta_a - \mu_a\bIn, \wt{\bv}_b\rangle+t_az_b)|$ with probability $1-O(\exp(-cn^\epsilon))$, as $n \rightarrow \infty$, where $z_b\sim N(0,\sigma_{v,b}^2)$ and is independent of other random variables. Since $\wt{v}_b$ and $\eta_a$ are independent, $\E(\eta_a\wt{v}_b)=0$. Using the Gaussianity and Bernstein's inequality, 
	\begin{align*}
		P\left(|2n^{-1}(\langle\bEta_a - \mu_a\bIn, \wt{\bv}_b\rangle+t_az_b)|\geq(1-\varrho)\lambda/2\right)
		= O\left(\exp\{-cn^{\min\{\epsilon,1+\epsilon-2\tau\}}\}\right) 
		\text{ as } n\rightarrow \infty. 
	\end{align*}
	and thus for $\delta=O(n^{\min\{0,2\tau+\epsilon-2\}})$, as $n\rightarrow\infty$,
	\begin{multline} \label{ineq: PetaV}
		P\big(|2n^{-1}\langle(\bEta_a - \mu_a\bIn) - (\bH - \bIn\bmu^T)\wh{\btheta}^{a,\nb_a,\lambda,lasso},\wt{\bv}_b\rangle|\geq(1-\varrho)\lambda/2\big) \\
		=
		O\left(\exp\{-cn^{\min\{\epsilon,1+\epsilon-2\tau\}}\}\right). 
	\end{multline}
	By \eqref{eq: Peq_G}, \eqref{ineq: PetaVDiff} and \eqref{ineq: PetaV}, for $\delta=o(n^{-(4-\xi-3\epsilon)/2})$ and $\delta=O(n^{2\tau+\epsilon-2})$, there exists $c, B>0$, with probability $1-O\left(\exp\{-cn^{\min\{\epsilon,1+\epsilon-2\tau\}}\}\right)$, as $n\rightarrow\infty$, 
	\begin{align*}
		|G_b(\wh{\btheta}^{a,\nb_a,\lambda,lasso})|<(1+\varrho)\lambda/2+Bn^{1/2-\epsilon/2+\kappa}\delta, 
	\end{align*}
	and we obtain that for $\delta=o(n^{\min\{-(4-\xi-3\epsilon)/2, \epsilon-\kappa-1\}})$, 
	\begin{align*}
		P\Big(\max_{b\in\wbK\backslash\{a\}}|G_b(\wh{\btheta}^{a,\nb_a,\lambda,lasso})|\geq\lambda\Big) = O\big(\exp(-cn^\epsilon)\big)\quad\text{as }n\rightarrow\infty.\\[-30pt]
	\end{align*}
\end{proof}

\begin{theorem} \label{thm: lasso_type2}
	Let assumption A1 hold and assume $\bmu$ to be known. Let the penalty parameter satisfy $\lambda_n\sim dn^{-(1-\epsilon)/2}$ with some $d>0$ and $\kappa<\epsilon<\xi$. If, in addition, C0 holds with $\delta = o(n^{\min\{-(4-\xi-3\epsilon)/2, \epsilon-\kappa-1\}})$, for all $a\in\wbK$, 
	\begin{align*}
		P(\nb_a\subseteq\wh{\nb}_a^{\lambda,lasso})=1-O(\exp(-cn^\epsilon))
		\quad
		\text{as }n\rightarrow\infty.
	\end{align*}
\end{theorem}
\begin{proof}
	Using Theorem~\ref{thm: lasso_type1} and Lemma~\ref{lemma: lasso_sign}, the proof is similar to \cite{meinshausen2006high}, proof of Theorem 2. 
\end{proof}

\begin{corollary} \label{corollary: mb0_asy}
	Let assumption A1 hold and assume $\bmu$ to be known. Let the penalty parameter satisfy $\lambda_n\sim dn^{-(1-\epsilon)/2}$ with some $d>0$ and $\kappa<\epsilon<\xi$. If, in addition, C0 holds with  $\delta = o(n^{\min\{-(4-\xi-3\epsilon)/2, \epsilon-\kappa-1\}})$, then there exists $ c>0$ so that 
	\begin{align*}
		P(\wh{E}^{\lambda,lasso}=E)=1- O(\exp(-cn^\epsilon))
		\quad
		\text{as }n\rightarrow\infty.
	\end{align*}
\end{corollary}
\begin{proof}
	Note that $\wh{E}^{\lambda, lasso}\neq E$ if and only if there exists $a\in\wbK$ so that $\wh{\nb}_a^{\lambda, lasso}\neq \nb_a$. The result now follows from Theorem~\ref{thm: lasso_type1} and Theorem~\ref{thm: lasso_type2} by using Bonferroni's inequality and assumption A1.1. 
\end{proof}

\subsection{Proof of Theorem~\ref{corollary: net_ds0_asy}}
We prove a series of results, which will then imply Theorem~\ref{corollary: net_ds0_asy}. The asseration of Theorem~\ref{corollary: net_ds0_asy} follows from Corollary~\ref{thm: ds0_asy} together with Lemma~\ref{lemma: net_etaDiff}. First we introduce some notation. Let 
\begin{align*}
	& s=\max_{a\in\wbK}|\nb_a| \\
	& \bPsi=\frac{1}{n}(\bH - \bIn\bmu^T)^T(\bH - \bIn\bmu^T) \\
	& \bPsi^a=\frac{1}{n}(\bH_{-a} - \bIn\bmu_{-a}^T)^T(\bH_{-a} - \bIn\bmu_{-a}^T) \quad\text{for each }a\in\wbK.
\end{align*} 
We also define, for each $a\in\wbK$, $1\leq s \leq K$ and $1\leq q\leq \infty$
\begin{align} \label{def: kappa}
	\kappa^a_q(s)=\min_{J:|J|\leq s}\Big(\min_{\bDelta\in\mathbb{R}^p: \|\bDelta_{J^c}\|_1\leq\|\bDelta_{J}\|_1\atop \|\bDelta\|_q=1 }\big|\bPsi^a\bDelta\big|_\infty\Big). 
\end{align}
\begin{claim} \label{claim: ds2}
	For $n$ large enough, $\lvert\bPsi-\bSigma\rvert_\infty<\frac{1}{2\alpha s}$ with probability no less than $1-K^2\exp\left(-\frac{n}{150\alpha^2s^2\sigma^4}\right)$.
\end{claim}
\begin{proof}
	Denote $\bSigma=(\sigma_{ab})_{K\times K}$. 
	For $b\neq a$, we have $\bEta_b - \mu_b\bIn=\sigma_{ba}\sigma_{aa}^{-1}(\bEta_a - \mu_a\bIn)+\bv_{ba}$, where each element of $\bv_{ba}$ is a zero mean normal with variance $\sigma_{b|a}\leq\sigma_{bb}$; and 
	\begin{align*}
	& \; P\Big(\Big|\frac{1}{n}(\bEta_a - \mu_a\bIn)^T(\bEta_b - \mu_b\bIn)-\sigma_{ab}\Big|>\frac{1}{2\alpha s}\Big) \\
	\leq &\;
	P\left(\sigma_{ba}\Big|\frac{(\bEta_a - \mu_a\bIn)^T(\bEta_a - \mu_a\bIn)}{n\sigma_{aa}}-1\Big|>\frac{1}{3\alpha s}\right) + P\left(\Big|\frac{(\bEta_a - \mu_a\bIn)^T\bv_{ba}}{n}\Big|>\frac{1}{6\alpha s}\right)
	\end{align*}
	Using a tail bound for the chi-squared distribution from \cite{Johnstone2001chisquare}, Bernstein's inequality and Bonferroni's inequality, we have for $n$ large enough, 
	\begin{align*}
		P\Big(\lvert\bPsi-\bSigma\rvert_\infty>\frac{1}{2\alpha s}\Big)
		\leq
		K^2\exp\left(-\frac{n}{150\alpha^2s^2\sigma^4}\right).
	\end{align*}
\end{proof}

\begin{claim} \label{claim: ds3}
	Under assumption C0, for any $q >0$, we have with probability no less than $1-2K^2\exp\big(-\frac{nq^2}{4\sigma^4+2\sigma^2q}\big) - 2K\exp\left(-2^{-1}n\right)$,
	\begin{align*}
		\max_{a\in\wbK}\big|n^{-1}(\wh{\bH}_{-a} - \bIn\bmu_{-a}^T)^T\bm{\xi}_a\big|_\infty 
		\leq 
		q+4\sigma\delta+\delta^2.
	\end{align*}
\end{claim}
\begin{proof}
	By straightforward calculation, 
	\begin{align*}
		&\max_{a\in\wbK}\big|n^{-1}(\wh{\bH}_{-a} -  \bIn\bmu_{-a}^T)^T\bm{\xi}_a\big|_\infty \\ 
		\leq 
		&
		\max_{a\in\wbK}\big|n^{-1}(\wh{\bH}_{-a} - \bIn\bmu_{-a}^T)^T\bv_a\big|_\infty + \max_{a\in\wbK}\big|n^{-1}(\wh{\bH}_{-a} - \bIn\bmu_{-a}^T)^T(\wh{\bEta}_a-\bEta_a)\big|_\infty \\
		\leq
		&
		\max_{a\in\wbK}\max_{ b\in\wbK\backslash\{a\}}n^{-1}\left|(\bEta_{b} - \mu_b\bIn)^T\bv_a\right| + n^{-1/2}\big(\max_{a\in\wbK}\|\bEta_a - \mu_a\bIn\|_2 + \max_{a\in\wbK}\|\bv_a\|_2\big)\delta+\delta^2.
	\end{align*}
	Note that $n^{-1}|(\bEta - \mu_a\bIn)^T_{b}\bv_a|$ can be estimated by Bernstein's inequality, and both $\|\bEta_a - \mu_a\bIn\|_2^2$ and $\sigma_{aa}^{-1}\|\bv_a\|_2^2$ are chi-square distributed. Thus, together with Bonferroni's inequality, gives 
	\begin{align*}
		P\Big(\max_{a\in\wbK}\max_{ b\in\wbK\backslash\{a\}}n^{-1}\left|(\bEta_{b} - \mu_b\bIn)^T\bv_a\right|\geq q\Big)
		\leq 
		2K^2\exp\left(-\frac{nq^2}{4\sigma^4+2\sigma^2 q}\right)
	\end{align*}
	and
	\begin{align*}
		P\Big(n^{-1/2}\big(\max_{a\in\wbK}\|\bEta_a-\mu_a\bIn\|_2+\max_{a\in\wbK}\|\bv_a\|_2\big) > 4\sigma\Big)
		\leq 
		2K\exp\left(-2^{-1}n\right).
	\end{align*}
\end{proof}
\begin{claim} \label{claim: ds4}
	For any $\alpha > \|\bSigma^{-1}\|_\infty$, there exists $n_\alpha=n(\alpha)$ so that, for all $n\geq n_\alpha$, 
	\begin{align*}
		P\Big(\max_{a\in\wbK}\kappa^a_\infty(s)\geq \|\bSigma^{-1}\|_\infty^{-1}-\alpha^{-1}\Big)
		\leq 
		1 - K^2\exp\left(-\frac{n}{150\alpha^2s^2\sigma^4}\right)
	\end{align*}
\end{claim}
\begin{proof}
	The result follows from Claim~\ref{claim: ds2} and the inequality below: for all $a\in\wbK$, 
	\begin{align*}
		\kappa^a_\infty(s) 
		& \geq 
		\min_{J:|J|\leq s}\Big(\min_{\bDelta\in\mathbb{R}^p: \|\bDelta_{J^c}\|_1\leq\|\bDelta_{J}\|_1\atop |\bDelta|_\infty=1 }\left|\bSigma\bDelta\right|_\infty\Big)
		- \max_{J: |J|\leq s}\Big(\max_{\bDelta\in\mathbb{R}^p: \|\bDelta_{J^c}\|_1 \leq \|\bDelta_{J}\|_1 \atop |\bDelta|_\infty=1 }\left|\left(\bPsi^a-\bSigma\right)\bDelta\right|_\infty\Big) \\
		& \geq
		\|\bSigma\|_\infty^{-1} - \max_{J: |J|\leq s}\max_{\bDelta\in\mathbb{R}^p: \|\bDelta_{J^c}\|_1 \leq \|\bDelta_{J}\|_1 \atop |\bDelta|_\infty=1 }\|\bDelta\|_1\left|\bPsi^a - \bSigma\right|_\infty \\
		& \geq
		\|\bSigma^{-1}\|_\infty^{-1} - 2s\left|\bPsi - \bSigma\right|_\infty.
	\end{align*}
\end{proof}

\begin{theorem} \label{thm: ds0_asy}
	Let assumption B1 hold and assume $\bmu$ to be known. Let $\lambda_n^{-1} = O(n^{\frac{1-\epsilon}{2}})$ with some $\epsilon>0$ be such that $\xi>\epsilon>2\kappa-2p+1$. If, in addition, C0 holds with  $\delta = O(n^{-p})$ for some $p>\kappa+(1-\xi)/2$, 
	\begin{align*}
		P\big(\wh{E}^{\lambda,ds} = E\big)
		= 1 - O\big(\exp(-cn^{\min\{\epsilon,1-2\kappa\}})\big)
		\quad
		\text{as }n\rightarrow\infty.
	\end{align*}
\end{theorem}
\begin{proof}
	Note that by Claim~\ref{claim: ds3}, the property of chi-square distribution,  and 
	\begin{align*}
		&\, \big|n^{-1}(\wh{\bH}_{-a} - \bIn\bmu_{-a}^T)^T\big((\wh{\bEta}_a - \mu_a\bIn)-(\wh{\bH} - \bIn\bmu^T)\btheta^a\big)\big|_\infty \\
		= &\,
		\big|n^{-1}(\wh{\bH}_{-a} - \bIn\bmu_{-a}^T)^T\big((\bH_{-a} - \bIn\bmu_{-a}^T)\btheta^a_{-a} + \bm{\xi}_a - (\wh{\bH}_{-a} - \bIn\bmu_{-a}^T)\btheta^a_{-a}\big)\big|_\infty \\
		\leq &\, 
		\big|n^{-1}(\wh{\bH}_{-a} - \bIn\bmu_{-a}^T)^T\bR_{-a}\btheta^a_{-a}\big|_\infty + \max_{a\in\wbK}\big|n^{-1}(\wh{\bH}_{-a} - \bIn\bmu_{-a}^T)^T\bm{\xi}_a\big|_\infty\\
		\leq &\, 
		n^{-1/2}\delta\big(\sqrt{n}\delta + \max_{b\in\wbK}\|\bEta_b - \mu_b\bIn\|_2\big)s\upsilon^{-1} + \max_{a\in\wbK}\big|n^{-1}(\wh{\bH}_{-a} - \bIn\bmu_{-a})^T\bm{\xi}_a\big|_\infty, 
	\end{align*}
	we have for each $q > 0$, that with probability no less than $1 - 2K\exp\left(-2^{-1}n\right) - 2K^2\exp\left(-\tfrac{nq^2}{4\sigma^4+2\sigma^2q}\right)$,
	\begin{multline*}
		\max_{a\in\wbK}\big|n^{-1}(\wh{\bH}_{-a} - \bIn\bmu_{-a}^T)^T\big((\wh{\bEta}_a - \mu_a\bIn) - (\wh{\bH} - \bIn\bmu^T)\btheta^a\big)\big|_\infty \\
		\leq q+(4\sigma+\delta)\delta+(2\sigma+\delta)\upsilon^{-1}s\delta
	\end{multline*}
	Note that $\lambda_n^{-1} = O(n^{\frac{1-\epsilon}{2}})$, which means there exits $b > 0$ so that $\lambda_n \geq bn^{-\frac{1-\epsilon}{2}}$ for all $n$ large enough. Moreover, $s\delta = O(n^{\kappa-p}) = o(\lambda_{a,n}(\|\btheta^a\|_1))$.   Choosing $q=bn^{-\frac{1-\epsilon}{2}}/2$, we obtain $q+2\upsilon^{-1}s\delta \leq \lambda_{a,n}(\|\btheta^a\|_1)$ for all large $n$. Thus
	\begin{multline*}
		P\big(\big|n^{-1}(\wh{\bH}_{-a} - \bIn\bmu_{-a}^T)^T\big((\wh{\bEta}_a - \mu_a\bIn) - (\wh{\bH} - \bIn\bmu^T)\btheta^a\big)\big|_\infty
		\leq \lambda_{a,n}(\|\btheta^a\|_1), \,\forall a\in\wbK \big) \\
		= 1 - O(\exp(-cn^{\epsilon}))
		\quad\text{as }n\rightarrow\infty,
	\end{multline*}
	which means that the true parameter $\btheta^a$ falls into the feasible set of problem (3.8) with probability $1-O(\exp(-cn^\epsilon))$ as $n\rightarrow\infty$.  With $\bDelta^a=\wt{\btheta}^{a,\lambda,ds}_{-a}-\btheta^a_{-a}$, we obtain 
	\begin{align*}
	\left|\bPsi^a\bDelta^a\right|_\infty 
	\leq &  \big|n^{-1}(\wh{\bH}_{-a} - \bIn\bmu_{-a}^T)^T(\bH_{-a} - \bIn\bmu_{-a}^T)\bDelta^a\big|_\infty + \big|n^{-1}\bR_{-a}^T(\bH_{-a} - \bIn\bmu_{-a}^T)\bDelta^a\big|_\infty \\
	\leq  & 
	\big|n^{-1}(\wh{\bH}_{-a} - \bIn\bmu_{-a}^T)^T((\wh{\bEta}_a - \mu_a\bIn) - (\wh{\bH}_{-a} - \bIn\bmu_{-a}^T)\wt{\btheta}_{-a}^{a,\lambda,ds})\big|_\infty  \\
	& \hspace{2cm} + \big|n^{-1}(\wh{\bH}_{-a} - \bIn\bmu_{-a}^T)^T\bm{\xi}_a\big|_\infty + \big|n^{-1}(\wh{\bH}_{-a} - \bIn\bmu_{-a}^T)^T\bR_{-a}\wt{\btheta}^{a,ds}_{-a}\big|_\infty \\
	& \hspace{6cm} + \big|n^{-1}\bR_{-a}^T(\bH_{-a} - \bIn\bmu_{-a}^T)\bDelta^a\big|_\infty  \\
	\leq  & 
	\lambda_{a,n}\big(\|\wt{\btheta}^{a,\lambda,ds}\|_1\big) + \big|n^{-1}(\wh{\bH}_{-a} - \bIn\bmu_{-a}^T)^T\bm{\xi}_a\big|_\infty +\delta^2\|\wt{\btheta}_{-a}^{a,\lambda,ds}\|_1  \\
	& \hspace{5cm} + n^{-1/2}\big(2\|\wt{\btheta}_{-a}^{a,\lambda,ds}\|_1+\|\btheta_{-a}^a\|_1\big)\delta\max_{b\in\wbK}\|\bEta_b\|_2.
	\end{align*}
	Then, by Claim~\ref{claim: ds3}, the definition of $\wt{\btheta}^{a,\lambda,ds}_{-a}$ and the property of chi-square distribution, the following holds. For any constant $d>0$, there exists $c=c(d)>0$ so that as $n\rightarrow\infty$, 
	\begin{align*}
		P\big(\left|\bPsi^a\bDelta^a\right|_\infty \leq \lambda_{a,n}(\|\wt{\btheta}^{a,\lambda,ds}\|_1) + dn^{-\frac{1-\epsilon}{2}} +4\upsilon^{-1}s\delta, \, \forall a\in\wbK\big)
		=
		1-O\left(\exp(-cn^{\epsilon})\right). 
	\end{align*}
	By definition, we have  $\kappa_\infty(s)|\bDelta^a|_\infty\leq \left|\bPsi^a\bDelta^a\right|_\infty$. Using Claim~\ref{claim: ds4} with assumption B1.3.(b), there exists a constant $t > 0$ and some $c > 0$ so that
	\begin{align*}
		P\big(\left|\bDelta^a\right|_\infty \leq t\lambda_{a,n}(\|\wt{\btheta}^{a,\lambda,ds}\|_1),\, \forall a\in\wbK\big)
		=
		1-O\left(\exp(-cn^{\min\{1-2\kappa,\epsilon\}})\right)
		\quad\text{for }n\rightarrow\infty. 
	\end{align*}
	The assertion of Theorem~\ref{thm: ds0_asy} follows from the fact that $|\theta_b^a|=\Omega(n^{-(1-\xi)/2})$ for all $b\in\nb_a$, $a\in\wbK$, $\max_{a\in\wbK}\lambda_{a,n}(\|\btheta^a\|_1) = o(n^{-\frac{1-\xi}{2}}(\log n)^{-1})$, and $\log n \rightarrow\infty$. 
\end{proof}

\subsection{Proof for Section~\ref{subsec: extension}} 
To verify Theorem~\ref{thm: net_beta}, first observe that for all $k\in\wbK$, under assumption C0, 
\begin{align*}
	|\wb{\wh{\eta}}_k-\mu_k|
	&\leq
	|\wb{\wh{\eta}}_k-\wb{\eta}_k|+|\wb{\eta}_k-\mu_k| \\
	& =
	\frac{1}{n}\|\bm{1}_n^T(\wh{\bEta}_k-\bEta_k)\|_2+|\wb{\eta}_k-\mu_k| \\
	&\leq
	\frac{1}{\sqrt{n}}\|\wh{\bEta}_k-\bEta_k\|_2+|\wb{\eta}_k-\mu_k|\leq 
	\delta + |\wb{\eta}_k-\mu_k|. 
\end{align*}
Recalling that $\sqrt{n}(\wb{\eta}_k-\mu_k)/\sqrt{\sigma_{kk}}\sim N(0,1)$ and  $\sigma=\max_{k\in\wbK}\sqrt{\sigma_{kk}}$, we obtain
\begin{align} \label{Pineq: maxMu}
	P\Big(\max_{k\in\wbK}|\wb{\eta}_k-\mu_k| \geq \delta_\mu\Big)
	\leq
	2\sum_{k\in\wbK}\wt{\Phi}\left(\frac{\sqrt{n}\delta_\mu}{\sqrt{\sigma_{kk}}}\right)
	\leq 
	\sqrt{\frac{2}{\pi}}\frac{K\sigma}{\sqrt{n}\delta_\mu}\exp\left( -\frac{n\delta_\mu^2}{2\sigma^2}\right), 
\end{align}
and thus,
\begin{align*}
	P\Big(\max_{k\in\wbK}|\wb{\wh{\eta}}_k-\mu_k|\geq \delta+\delta_\mu\Big)
	\leq 
	\sqrt{\frac{2}{\pi}}\frac{K\sigma}{\sqrt{n}\delta_\mu}\exp\left( -\frac{n\delta_\mu^2}{2\sigma^2}\right).
\end{align*}
Using the trivial bound $\|\wb{\wh{\bEta}}-\bmu\|_2\leq \sqrt{K}\max_{k\in\wbK}|\wb{\wh{\eta}}_k-\mu_k|$, we have
\begin{align*}
	P\left(n^{\frac{b-\gamma}{2}}\|\wb{\wh{\bEta}}-\bmu\|_2 \geq n^{\frac{b-\gamma}{2}}K^{\frac{1}{2}}(\delta+\delta_\mu)\right)
	\leq 
	\sqrt{\frac{2}{\pi}}\frac{K\sigma}{\sqrt{n}\delta_\mu}\exp\left( -\frac{n\delta_\mu^2}{2\sigma^2}\right)
	\quad
	\text{as }n\rightarrow\infty.
\end{align*}
Thus, for $\delta=o(n^{-\frac{b}{2}})$ with $b<1,$ and $K=O(n^{\gamma})$, we obtain that, for any fixed $\delta_0>0$, 
\begin{align*}
	P\big(n^{\frac{b-\gamma}{2}}\|\wb{\wh{\bEta}}-\bmu\|\geq \delta_0\big) =
	O\left(\exp(-cn^{1-b})\right) 
	\quad
	\text{as }n\rightarrow\infty,
\end{align*}
which, if combined with Lemma~\ref{lemma: net_etaDiff} (note that the above arguments assume C0) is the first statement of Theorem~\ref{thm: net_beta}. Recall that  $\bbeta = \bX^+\bmu $ and $ \widehat\bbeta = \bX ^+\overline{\widehat{\bEta}}$. Using the assumptions, $\rank(\bX)=L$ and $\sigma_{\max}(\bX)=O(1)$, we also obtain that, for fixed $\delta_0>0$,
\begin{align*}
	P\big(n^{\frac{b-\gamma}{2}}\|\wh{\bbeta}-\bbeta\| \geq \delta_0\big) = O\left(\exp(-cn^{1-b})\right) 
	\quad
	\text{as }n\rightarrow\infty.
\end{align*}
Together with Lemma~\ref{lemma: net_etaDiff}, we have the results in Theorem~\ref{thm: net_beta}. 

For Corollary~\ref{corollary: net_mb_asy} and Corollary~\ref{corollary: net_ds_asy}, recall that, since $\bmu$ is unknown, we replace $\wh{\bH}$ by $\wh{\bH}-\bIn\wb{\wh{\bEta}}^{T}$ in \eqref{def: thetaLasso}. For our model (cf. \ref{model: matrix}), this means that we replace $\bR$ by 
\begin{align} \label{def: tildeR}
	\wt{\bR}=(\wh{\bH} - \bIn\wb{\wh{\bEta}}^{T}) - (\bH - \bIn{\bmu}^T) = \left((\wh{\eta}_k^{(t)}-\wb{\wh{\eta}}_k)-(\eta_k^{(t)}-\mu_k)\right)_{1\leq t\leq n \atop 1\leq k\leq K}.
\end{align}
Note that under assumption C0, for all $k\in\wbK$, 
\begin{align*} 
	\big|(\wh{\eta}_k^{(t)}-\wb{\wh{\eta}}_k)-(\eta_k^{(t)}-\mu_k)\big|
	\leq
	\big|\wh{\eta}_k^{(t)}-\eta_k^{(t)}\big| + \big|\wb{\wh{\eta}}_k-\mu_k\big| 
	\leq 
	2\delta+\big|\wb{\eta}_k-\mu_k\big|, 
\end{align*}
which means 
\begin{align} \label{ineq: norm_tildeR}
	|\wt{\bR}|_\infty\leq 2\delta + \max_{k\in\wbK}|\wb{\eta}_k-\mu_k|. 
\end{align}
Together with \eqref{Pineq: maxMu}, we have
\begin{align} \label{Pineq: norm_tildeR}
	P\big(|\wt{\bR}|_\infty\leq 2\delta + \delta_\mu\big) 
	\geq 
	1 - \sqrt{\frac{2}{\pi}}\frac{K\sigma}{\sqrt{n}\delta_\mu}\exp\left( -\frac{n\delta_\mu^2}{2\sigma^2}\right). 
\end{align}
Using \eqref{Pineq: norm_tildeR}, Corollary~\ref{corollary: mb0_asy} (or Theorem~\ref{thm: ds0_asy}) and Lemma~\ref{lemma: net_etaDiff}, we have Corollary~\ref{corollary: net_mb_asy} and Corollary~\ref{corollary: net_ds_asy}. 
\bibliography{reference1}

\newpage
\section*{Supplemental material} \label{sec: results}

Here we present further results of our simulation studies of the three methods introduced in the manuscript.

\subsection{Results: finite-sample performance as a function of the penalty parameter} \label{subsec: fun_penalty}
ROC curves are shown in figure~\ref{fig: roc2} - \ref{fig: roc6} for each of the following six cases: $n = 100$ and $K = 15, 30, 50, 80, 100$. The ROC curves are color-coded: Lasso: red, Dantzig selector: blue and MU-selector: green. $\lambda_{\rm opt}$ is the tuning parameter corresponding to the total (overall) minimum error rate. 
\begin{figure}[H]
	\centering
	\includegraphics[width=1\linewidth]{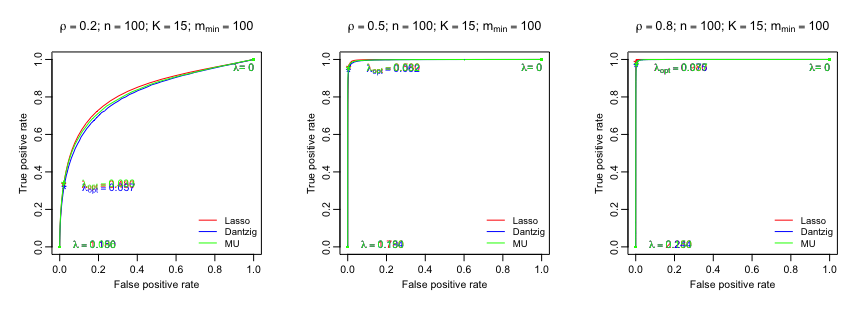} 
	\caption{ROC curves comparing the three proposed methods for $K = 15$ and $n = 100$}
	\label{fig: roc2} 
\end{figure}
\begin{figure}[H]
	\centering
	\includegraphics[width=1\linewidth]{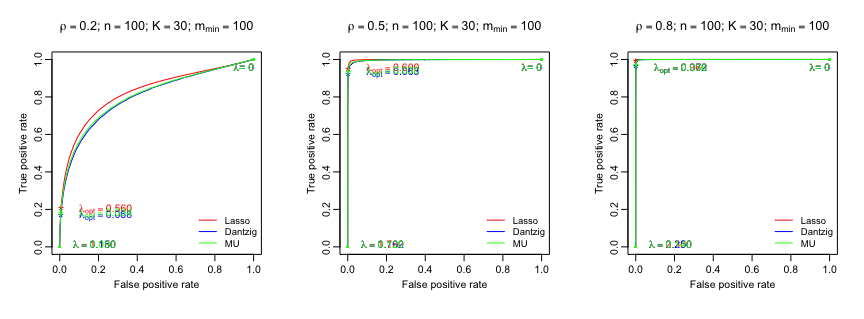} 
	\caption{ROC curves comparing the three proposed methods for $K = 30$ and $n = 100$}
	\label{fig: roc3} 
\end{figure}
\begin{figure}[H]
	\centering
	\includegraphics[width=1\linewidth]{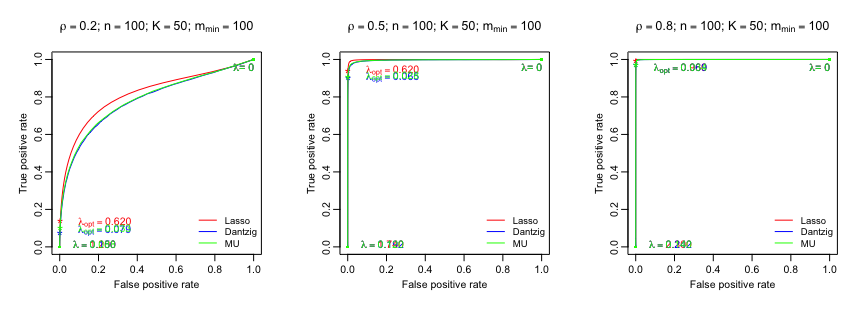} 
	\caption{ROC curves comparing the three proposed methods for $K = 50$ and $n = 100$}
	\label{fig: roc4} 
\end{figure}
\begin{figure}[H]
	\centering
	\includegraphics[width=1\linewidth]{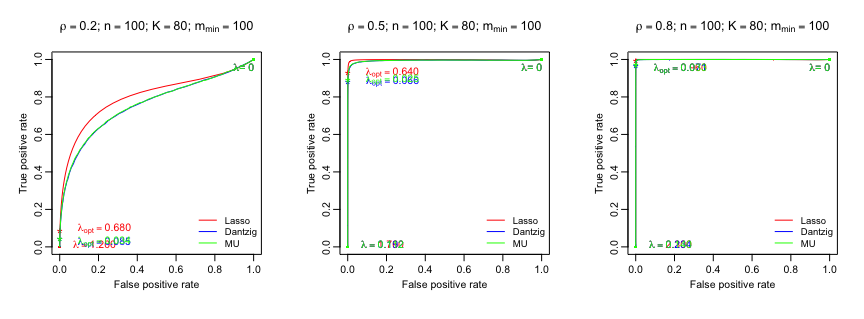} 
	\caption{ROC curves comparing the three proposed methods for $K = 80$ and $n = 100$}
	\label{fig: roc5} 
\end{figure}
\begin{figure}[H]
	\centering
	\includegraphics[width=1\linewidth]{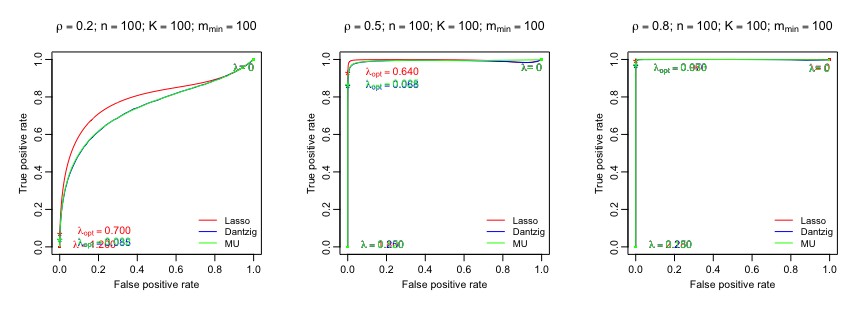} 
	\caption{ROC curves comparing the three proposed methods for $K = 100$ and $n = 100$}
	\label{fig: roc6} 
\end{figure}

\newpage
\noindent 
The average error rates for the three methods are shown in figure~\ref{fig: rate2} - \ref{fig: rate6} for each of cases: $n = 100$ and $K = 15, 30, 50, 80, 100$. 
\begin{figure}[H]
	\centering
	\includegraphics[width=1.02\linewidth]{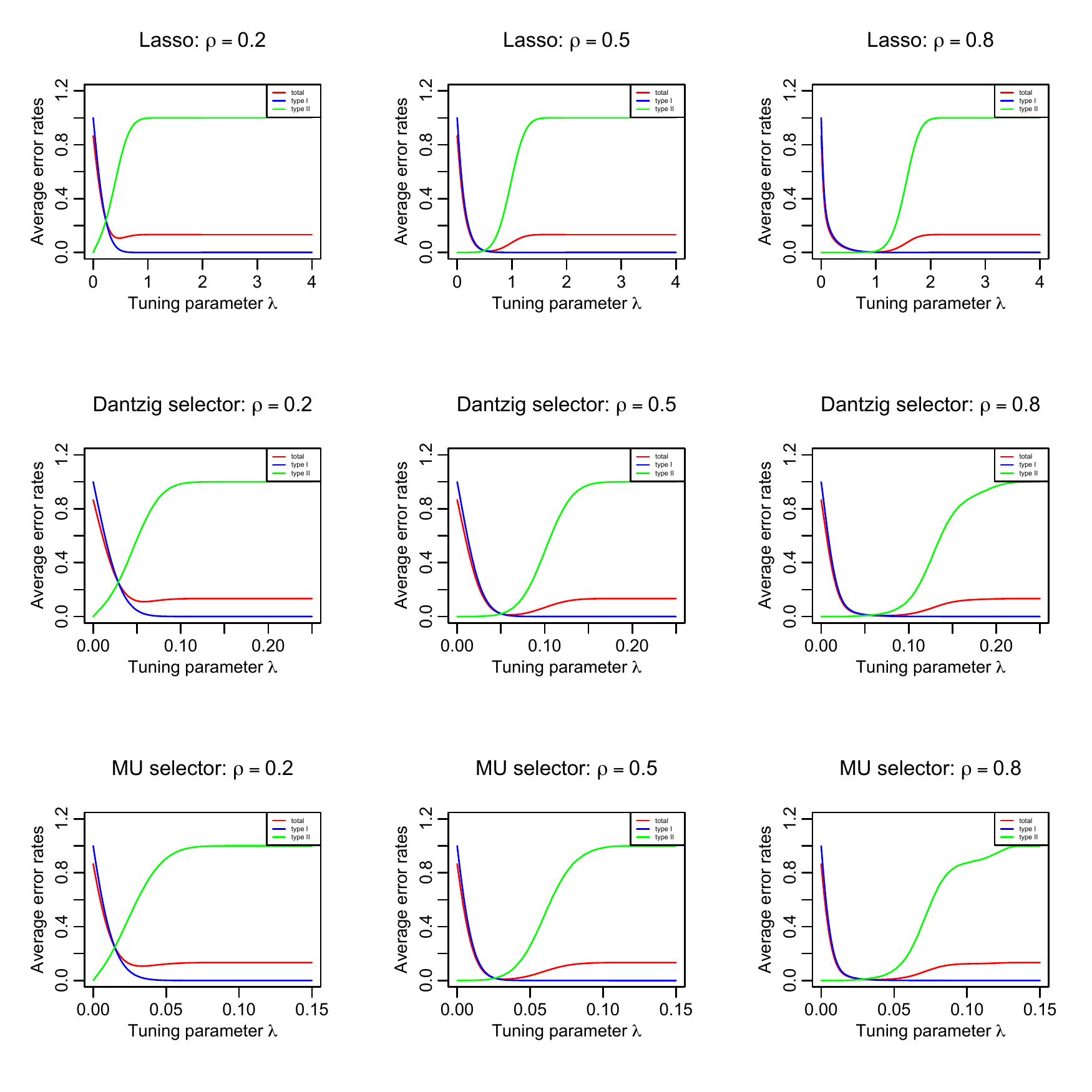} 
	\caption{Average error rates as functions of $\lambda$ for $K = 15$ and $n = 100$.}
	\label{fig: rate2} 
\end{figure}
\begin{figure}[H]
	\centering
	\includegraphics[width=1.02\linewidth]{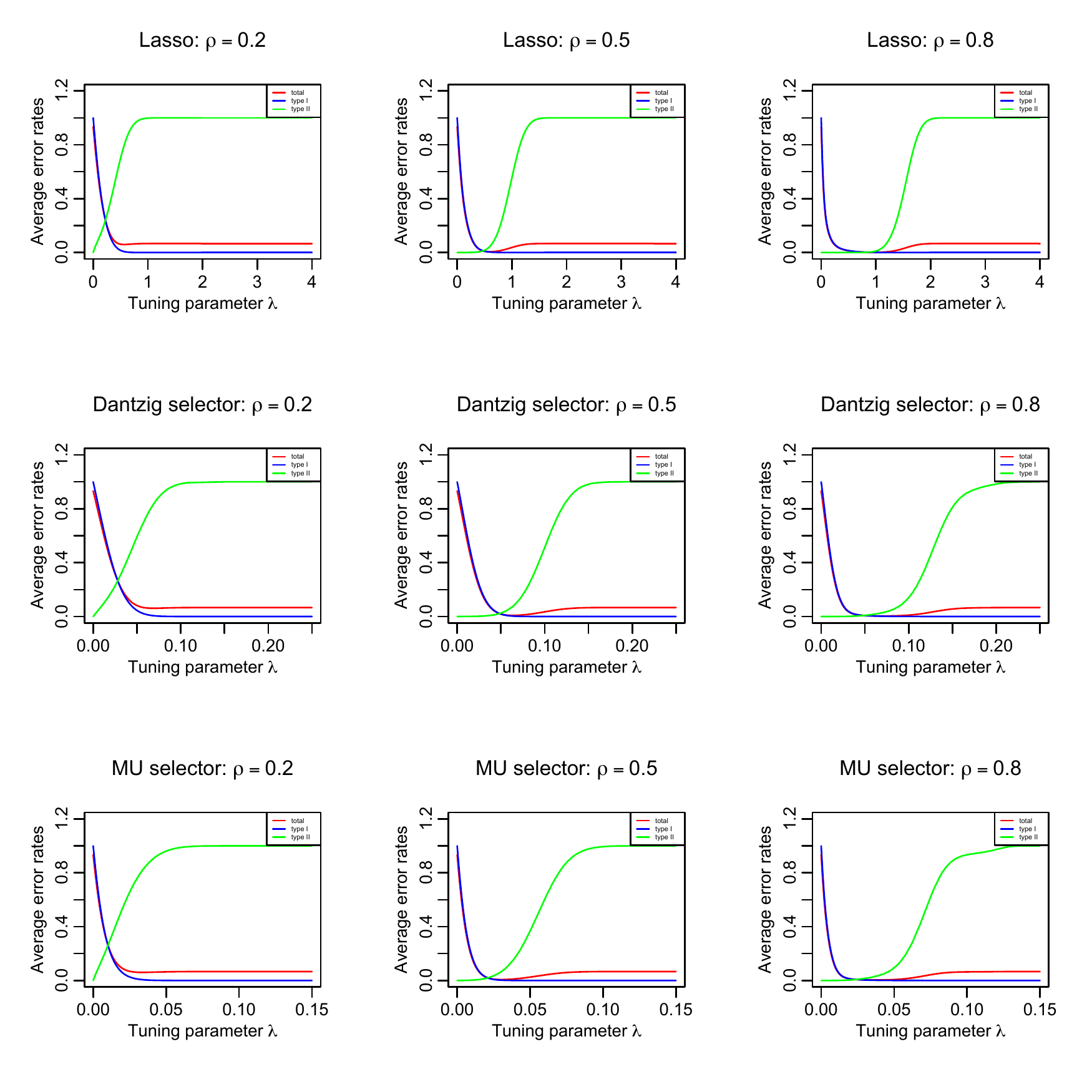} 
	\caption{Average error rates as functions of $\lambda$ for $K = 30$ and $n = 100$.}
	\label{fig: rate3} 
\end{figure}
\begin{figure}[H]
	\centering
	\includegraphics[width=1.02\linewidth]{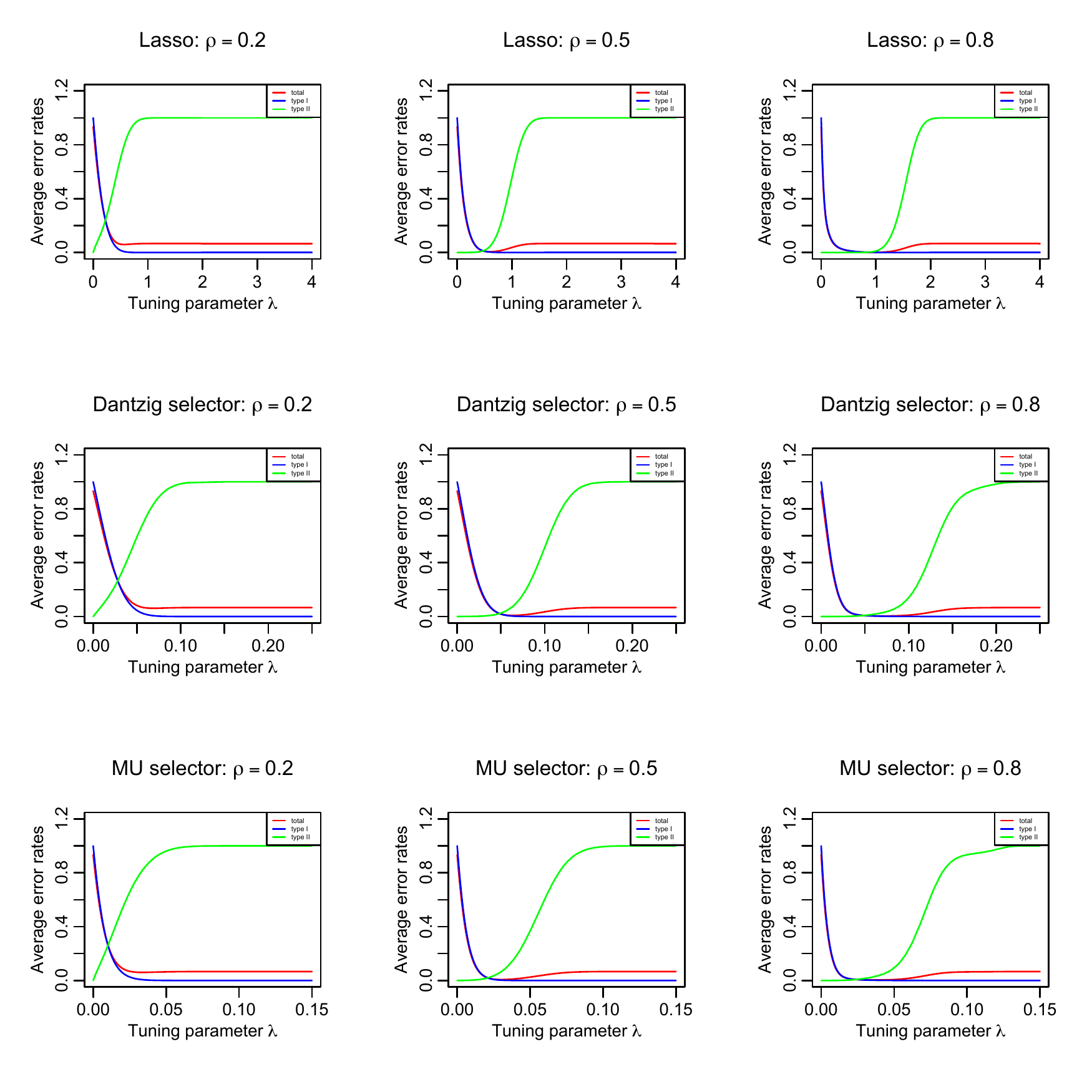} 
	\caption{Average error rates as functions of $\lambda$  for $K = 50$ and $n = 100$.}
	\label{fig: rate4} 
\end{figure}
\begin{figure}[H]
	\centering
	\includegraphics[width=1.02\linewidth]{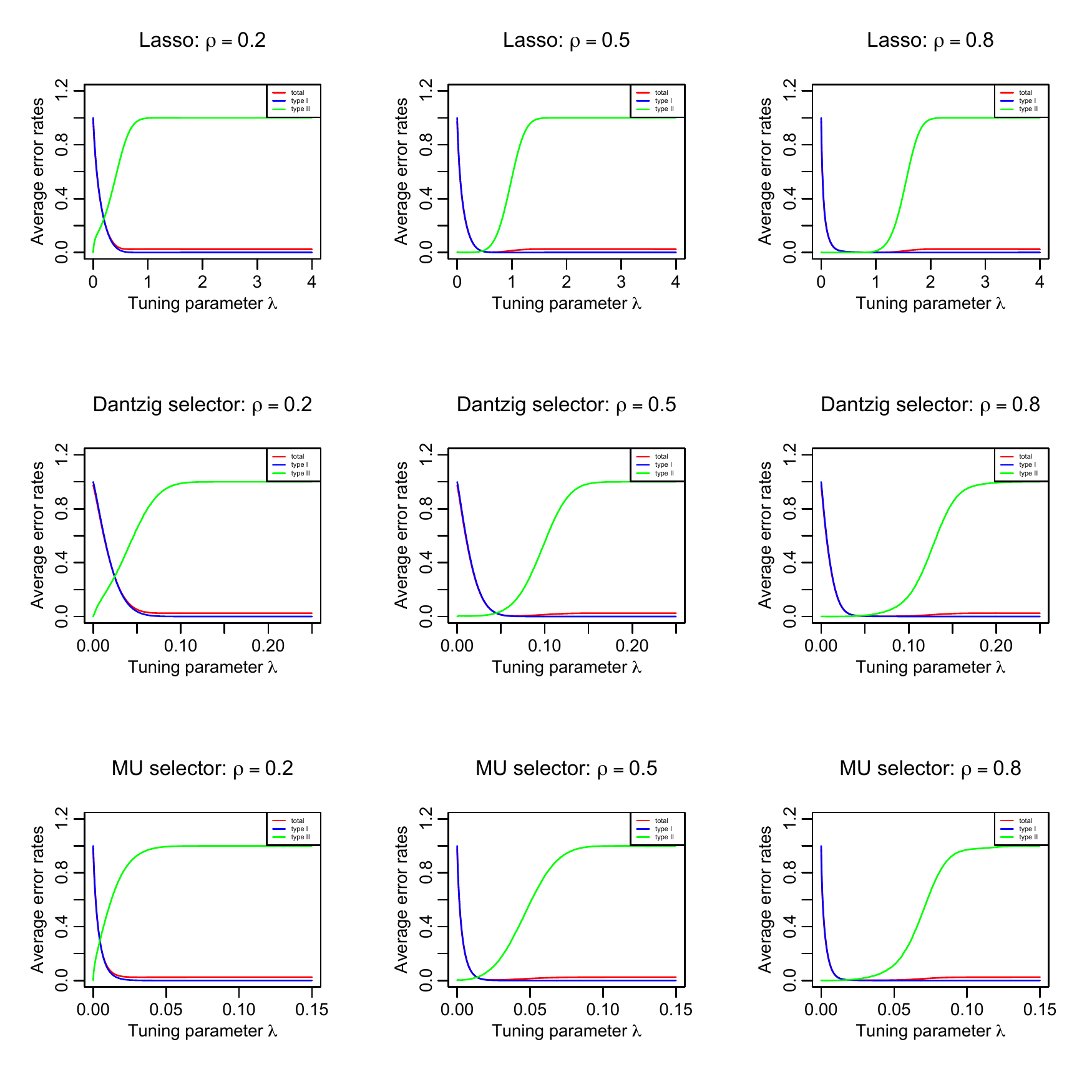} 
	\caption{Average error rates as functions of $\lambda$ for $K = 80$ and $n = 100$.}
	\label{fig: rate5} 
\end{figure}
\begin{figure}[H]
	\centering
	\includegraphics[width=1.02\linewidth]{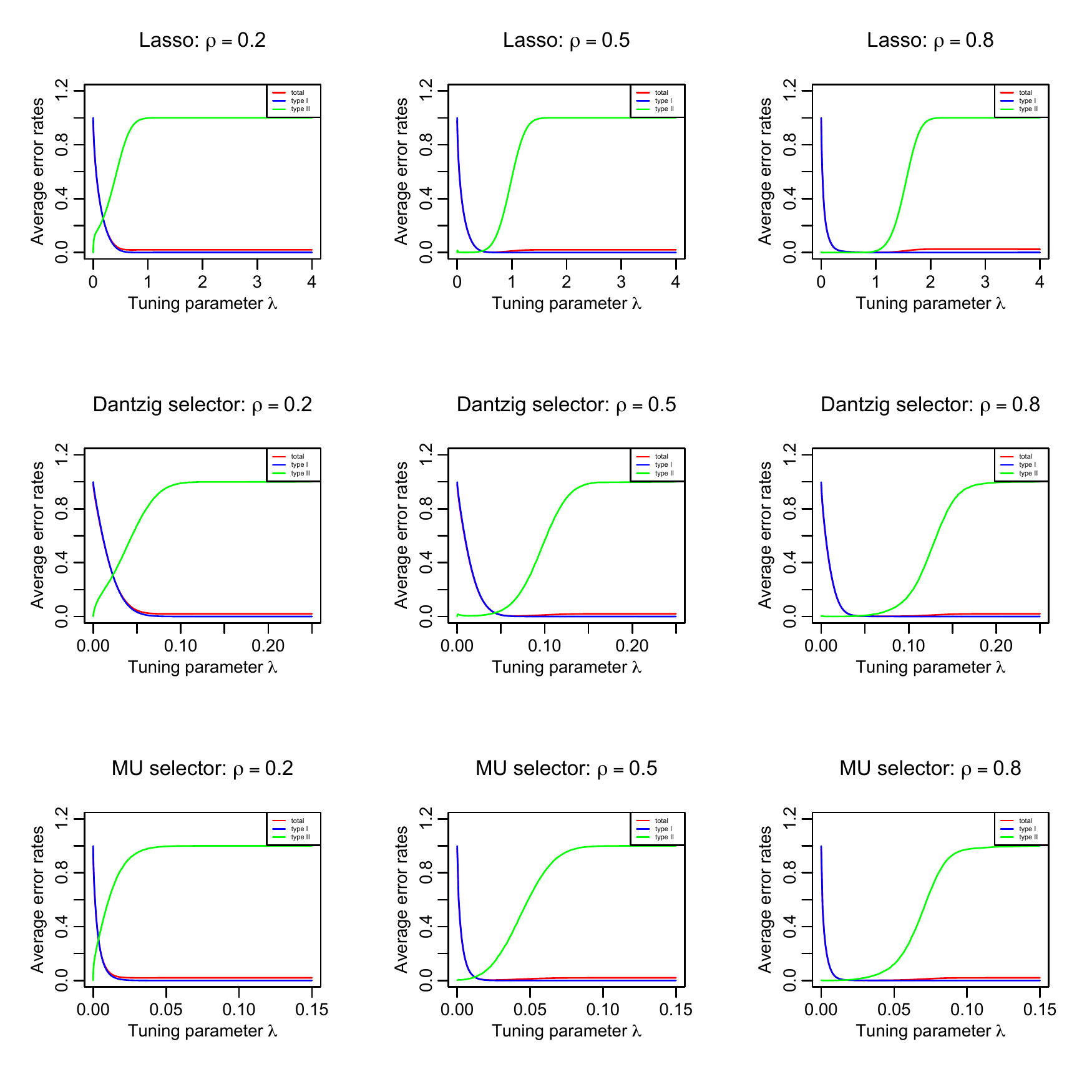} 
	\caption{Average error rates as functions of $\lambda$  for $K = 100$ and $n = 100$.}
	\label{fig: rate6} 
\end{figure}

\vspace{6mm}

\subsection{Results: finite-sample performance with data-driven penalty selection} \label{subsec: data_penalty}
All the following tables show averages and SEs of classification errors in \% over 100 replicates for the three proposed methods with both $``\vee"$ (left) and $``\wedge"$ (right).

\begin{table}[H]
	\caption{AR(1) model with $K = 30$}
	\begin{subtable}{1\textwidth}
		\sisetup{table-format=-1.2}   
		\centering
		\caption{$n = 100$}
		\begin{tabular}{c|c c c }
			\hline
			Ave (SE) 
			& Total (\%) &  Type \uppercase\expandafter{\romannumeral1} (\%)  & Type \uppercase\expandafter{\romannumeral2} (\%) \\
			\hline
			{Lasso}&1.59(1.19); 1.80(1.17)&1.68(1.26); 1.92(1.25)&0.43(1.39); 0.13(0.81)\\
			{Dantzig} &1.60(1.18); 1.73(1.10)&1.69(1.25); 1.84(1.18)&0.43(1.22); 0.13(0.81)
			\\
			{MU} &1.41(1.09); 1.83(1.29)&1.49(1.17); 1.96(1.38)&0.30(1.07) 0.07(0.47)\\
			\hline 
		\end{tabular}
	\end{subtable}
	
	\bigskip
	
	\begin{subtable}{1\textwidth}
		\sisetup{table-format=-1.2}   
		\centering
		\caption{$n = 500$}
		\begin{tabular}{c|c c c }
			\hline
			Ave (SE) 
			& Total (\%) &  Type \uppercase\expandafter{\romannumeral1} (\%)  & Type \uppercase\expandafter{\romannumeral2} (\%) \\
			\hline
			{Lasso}&1.16(1.07); 1.54(1.08) & 1.25(1.15); 1.65(1.16) & 0(0) \\
			{Dantzig} &1.21(1.05); 1.60(1.06) & 1.30(1.12); 1.72(1.13) & 0(0) \\
			{MU} &1.16(1.07); 1.73(1.16) &1.24(1.15); 1.86(1.24) & 0(0)\\
			\hline 
		\end{tabular}
	\end{subtable}

	\bigskip
	
	\begin{subtable}{1\textwidth}
		\sisetup{table-format=-1.2}   
		\centering
		\caption{$n = 1000$}
		\begin{tabular}{c|c c c }
			\hline
			Ave (SE) 
			& Total (\%) &  Type \uppercase\expandafter{\romannumeral1} (\%)  & Type \uppercase\expandafter{\romannumeral2} (\%) \\
			\hline
			{Lasso}&0.464(1.17); 1.04(1.70) & 0.498(1.26); 1.12(1.83) & 0(0) \\
			{Dantzig} &0.593(1.37); 1.09(1.75) & 0.636(1.47); 1.17(1.88) & 0(0) \\
			{MU} &0.543(1.39); 1.02(1.71) & 0.582(1.49); 1.09(1.83) & 0(0)\\
			\hline 
		\end{tabular}
	\end{subtable}
	\label{tab: ar1_k30}
\end{table}

\vspace{6mm}

\begin{table}[H]
	\caption{AR(1) model with $K = 200$}
	\begin{subtable}{1\textwidth}
		\sisetup{table-format=-1.2}   
		\centering
		\caption{$n = 100$}
		\begin{tabular}{c|c c c }
			\hline
			Ave (SE) 
			& Total (\%) &  Type \uppercase\expandafter{\romannumeral1} (\%)  & Type \uppercase\expandafter{\romannumeral2} (\%) \\
			\hline
			{Lasso}&0.386(0.250); 0.543(0.143)&0.390(0.253); 0.549(0.144)&0(0)\\
			{Dantzig} &0.416(0.237); 0.549(0.148)&0.420(0.239); 0.555(0.149)&0(0)
			\\
			{MU} &0.457(0.210); 0.550(0.174)&0.461(0.212); 0.556(0.176)&0(0)\\
			\hline 
		\end{tabular}
	\end{subtable}
	
	\bigskip
	
	\begin{subtable}{1\textwidth}
		\sisetup{table-format=-1.2}   
		\centering
		\caption{$n = 500$}
		\begin{tabular}{c|c c c }
			\hline
			Ave (SE) 
			& Total (\%) &  Type \uppercase\expandafter{\romannumeral1} (\%)  & Type \uppercase\expandafter{\romannumeral2} (\%) \\
			\hline
			{Lasso}&0.377(0.104); 0.436(0.119)&0.380(0.105); 0.441(0.121)&0(0)\\
			{Dantzig} &0.389(0.119); 0.436(0.129)&0.393(0.120); 0.440(0.130)&0(0)
			\\
			{MU} &0.372(0.107); 0.445(0.124)&0.376(0.108); 0.450(0.126)&0(0)\\
			\hline 
		\end{tabular}
	\end{subtable}
	
	\bigskip
	
	\begin{subtable}{1\textwidth}
		\sisetup{table-format=-1.2}   
		\centering
		\caption{$n = 1000$}
		\begin{tabular}{c|c c c }
			\hline
			Ave (SE) 
			& Total (\%) &  Type \uppercase\expandafter{\romannumeral1} (\%)  & Type \uppercase\expandafter{\romannumeral2} (\%) \\
			\hline
			{Lasso}&0.386(0.250); 0.543(0.143)&0.390(0.253); 0.549(0.144)&0(0)\\
			{Dantzig} &0.416(0.237); 0.549(0.148)&0.420(0.239); 0.555(0.149)&0(0)
			\\
			{MU} &0.457(0.210); 0.550(0.174)&0.461(0.212); 0.556(0.176)&0(0)\\
			\hline 
		\end{tabular}
	\end{subtable}
	\label{tab: ar1_k200}
\end{table}

\begin{table}[H]
	\caption{AR(4) model with $K = 30$}
	\begin{subtable}{1\textwidth}
		\sisetup{table-format=-1.2}   
		\centering
		\caption{$n = 100$}
		\begin{tabular}{c|c c c }
			\hline
			Ave (SE) 
			& Total (\%) &  Type \uppercase\expandafter{\romannumeral1} (\%)  & Type \uppercase\expandafter{\romannumeral2} (\%) \\
			\hline
			{Lasso}&20.3(1.21); 20.5(1.15)&2.72(1.47); 2.89(1.32)&71.7(4.44); 72.1(4.08)\\
			{Dantzig} &20.3(1.14); 20.6(1.29)&2.51(1.25); 3.10(1.48)&72.4(4.58); 71.9(4.57)
			\\
			{MU} &20.4(1.27); 20.6(1.24)&2.93(1.39); 3.03(1.38)&71.6(4.62); 72.2(4.64)\\
			\hline 
		\end{tabular}
	\end{subtable}
	
	\bigskip
	
	\begin{subtable}{1\textwidth}
		\sisetup{table-format=-1.2}   
		\centering
		\caption{$n = 500$}
		\begin{tabular}{c|c c c }
			\hline
			Ave (SE) 
			& Total (\%) &  Type \uppercase\expandafter{\romannumeral1} (\%)  & Type \uppercase\expandafter{\romannumeral2} (\%) \\
			\hline
			{Lasso}&10.4(2.03); 10.4(1.78)&1.21(0.89); 1.24(0.98)&37.4(9.06) 37.7(8.22)\\
			{Dantzig} &10.7(2.09); 10.6(1.87)&1.41(1.06); 1.36(1.07)&38.0(9.68); 38.1(8.68)
			\\
			{MU} &10.6(1.94); 10.5(1.85)&1.31(0.92); 1.38(1.04)&37.8(8.58); 37.4(8.57)
			\\
			\hline 
		\end{tabular}
	\end{subtable}
	
	\bigskip
	
	\begin{subtable}{1\textwidth}
		\sisetup{table-format=-1.2}   
		\centering
		\caption{$n = 1000$}
		\begin{tabular}{c|c c c }
			\hline
			Ave (SE) 
			& Total (\%) &  Type \uppercase\expandafter{\romannumeral1} (\%)  & Type \uppercase\expandafter{\romannumeral2} (\%) \\
			\hline
			{Lasso}&5.04(1.02); 5.10(0.98)&2.02(1.13); 1.94(1.06)&14.0(3.30); 14.5(3.33)\\
			{Dantzig} &5.25(1.04); 5.17(1.00)&2.03(1.19); 2.05(1.02)&14.8(3.48); 14.5(3.25)
			\\
			{MU} &5.27(0.95); 5.13(1.04)&2.10(1.07); 2.23(1.17)&14.6(3.01); 13.8(3.02)
			\\
			\hline 
		\end{tabular}
	\end{subtable}
	\label{tab: ar4_k30}
\end{table}
\vspace{6mm}
\begin{table}[H]
	\caption{AR(4) model with $K = 200$}
	\begin{subtable}{1\textwidth}
		\sisetup{table-format=-1.2}   
		\centering
		\caption{$n = 100$}
		\begin{tabular}{c|c c c }
			\hline
			Ave (SE) 
			& Total (\%) &  Type \uppercase\expandafter{\romannumeral1} (\%)  & Type \uppercase\expandafter{\romannumeral2} (\%) \\
			\hline
			{Lasso}&5.13(0.27); 4.81(0.18)&2.12(0.30); 1.64(0.21)&78.0(1.27); 81.4(1.36)\\
			{Dantzig} &5.07(0.30); 4.80(0.17)&2.03(0.34); 1.63(0.20)&78.5(1.41); 81.4(1.36)
			\\
			{MU} &5.16(0.28); 4.84(0.18)&2.15(0.31); 1.69(0.21)&77.9(1.17); 81.0(1.29)\\
			\hline 
		\end{tabular}
	\end{subtable}
	
	\bigskip
	
	\begin{subtable}{1\textwidth}
		\sisetup{table-format=-1.2}   
		\centering
		\caption{$n = 500$}
		\begin{tabular}{c|c c c }
			\hline
			Ave (SE) 
			& Total (\%) &  Type \uppercase\expandafter{\romannumeral1} (\%)  & Type \uppercase\expandafter{\romannumeral2} (\%) \\
			\hline
			{Lasso}&3.01(0.15); 3.09(0.16)&1.03(0.21); 1.06(0.21)&50.9(2.31); 52.0(2.09)\\
			{Dantzig} &3.03(0.16); 3.09(0.16)&1.02(0.20); 1.06(0.20)&51.6(2.10); 52.2(1.97)
			\\
			{MU} &3.02(0.13); 3.09(0.15)&1.02(0.17); 1.08(0.18)&51.3(2.13); 51.7(1.74)
			\\
			\hline 
		\end{tabular}
	\end{subtable}
	
	\bigskip
	
	\begin{subtable}{1\textwidth}
		\sisetup{table-format=-1.2}   
		\centering
		\caption{$n = 1000$}
		\begin{tabular}{c|c c c }
			\hline
			Ave (SE) 
			& Total (\%) &  Type \uppercase\expandafter{\romannumeral1} (\%)  & Type \uppercase\expandafter{\romannumeral2} (\%) \\
			\hline
			{Lasso}&1.99(0.13); 1.98(0.16)&1.08(0.17); 1.07(0.19)&24.0(1.80); 23.9(1.64)\\
			{Dantzig} &2.03(0.16); 2.02(0.16)&1.11(0.20); 1.11(0.20)&24.4(1.75); 24.1(1.64)
			\\
			{MU} &1.98(0.16); 2.00(0.14)&1.05(0.20); 1.09(0.17)&24.5(1.76); 24.1(1.52)
			\\
			\hline 
		\end{tabular}
	\end{subtable}
	\label{tab: ar4_k200}
\end{table}

\begin{table}[H]
	\caption{The random precision matrix model with  $\alpha = 0.1$ and $K = 30$}
	\vspace{-1.2mm}
	\begin{subtable}{1\textwidth}
		\sisetup{table-format=-1.2}   
		\centering
		\caption{$n = 100$}
		\begin{tabular}{c|c c c }
			\hline
			Ave (SE) 
			& Total (\%) &  Type \uppercase\expandafter{\romannumeral1} (\%)  & Type \uppercase\expandafter{\romannumeral2} (\%) \\
			\hline
			{Lasso}&6.19(2.24); 6.00(2.02)&5.01(2.01); 4.78(1.86)&16.9(8.50); 18.0(8.71)\\
			{Dantzig} &6.11(1.99); 6.32(2.06)&4.80(1.83); 5.04(1.84)&18.0(8.37); 18.6(8.73)
			\\
			{MU} &6.17(1.91); 6.29(1.92)&4.87(1.71); 5.08(1.63)&17.7(9.07); 17.5(8.84)\\
			\hline 
		\end{tabular}
	\end{subtable}
	
	\bigskip
	
	\begin{subtable}{1\textwidth}
		\sisetup{table-format=-1.2}   
		\centering
		\caption{$n = 500$}
		\begin{tabular}{c|c c c }
			\hline
			Ave (SE) 
			& Total (\%) &  Type \uppercase\expandafter{\romannumeral1} (\%)  & Type \uppercase\expandafter{\romannumeral2} (\%) \\
			\hline
			{Lasso}&1.21(0.87); 1.05(0.77)&1.33(0.95); 1.19(0.88)&0.05(0.38); 0.03(0.28)\\
			{Dantzig} &1.23(0.91); 1.09(0.89)&1.37(1.00); 1.23(1.01)&0.09(0.45); 0.00(0.00)
			\\
			{MU} &1.30(0.94); 1.01(0.73)&1.44(1.04); 1.13(0.81)&0.08(0.48); 0.01(0.09)
			\\
			\hline 
		\end{tabular}
	\end{subtable}
	
	\bigskip
	
	\begin{subtable}{1\textwidth}
		\sisetup{table-format=-1.2}   
		\centering
		\caption{$n = 1000$}
		\begin{tabular}{c|c c c }
			\hline
			Ave (SE) 
			& Total (\%) &  Type \uppercase\expandafter{\romannumeral1} (\%)  & Type \uppercase\expandafter{\romannumeral2} (\%) \\
			\hline
			{Lasso}&1.19(0.92); 0.90(0.74)&1.33(1.02); 1.02(0.84)&0(0)\\
			{Dantzig} &1.02(0.94); 0.94(0.82)&1.15(1.05); 1.06(0.92)&0(0)\\
			{MU} &1.01(0.92); 0.91(0.72)&1.14(1.04); 1.02(0.81)&0(0)\\
			\hline 
		\end{tabular}
	\end{subtable}
	\label{tab: rd0.1_k30}
\end{table}
\vspace{6mm}
\begin{table}[H]
	\caption{The random precision matrix model with $\alpha = 0.1$ and $K = 200$}
	\begin{subtable}{1\textwidth}
		\sisetup{table-format=-1.2}   
		\centering
		\caption{$n = 100$}
		\begin{tabular}{c|c c c }
			\hline
			Ave (SE) 
			& Total (\%) &  Type \uppercase\expandafter{\romannumeral1} (\%)  & Type \uppercase\expandafter{\romannumeral2} (\%) \\
			\hline
			{Lasso} &11.3(0.36); 10.7(0.29)&3.57 2.82&80.8(1.84); 82.3(1.96)\\
			{Dantzig} &11.2(0.31); 10.8(0.28)&3.27(0.34); 2.97(0.33)&82.3(1.80); 81.9(1.90)\\
			{MU}&11.0(0.29); 10.8(0.28)&3.18(0.32); 2.92(0.29)&82.1(1.74); 81.9(1.74)\\
			\hline 
		\end{tabular}
	\end{subtable}
	
	\bigskip
	
	\begin{subtable}{1\textwidth}
		\sisetup{table-format=-1.2}   
		\centering
		\caption{$n = 500$}
		\begin{tabular}{c|c c c }
			\hline
			Ave (SE) 
			& Total (\%) &  Type \uppercase\expandafter{\romannumeral1} (\%)  & Type \uppercase\expandafter{\romannumeral2} (\%) \\
			\hline
			{Lasso} &6.25(0.34); 6.23(0.33)&3.01(0.34); 3.05(0.31)&35.5(3.59); 35.0(3.48)\\
			{Dantzig} &6.90(0.30); 6.63(0.35)&3.26(0.36); 3.44(0.36)&39.7(3.52); 35.3(3.21)\\
			{MU}&6.89(0.34); 6.67(0.32)&3.27(0.39); 3.46(0.35)&39.5(3.47); 35.3(3.33)\\
			\hline 
		\end{tabular}
	\end{subtable}
	
	\bigskip
	
	\begin{subtable}{1\textwidth}
		\sisetup{table-format=-1.2}   
		\centering
		\caption{$n = 1000$}
		\begin{tabular}{c|c c c }
			\hline
			Ave (SE) 
			& Total (\%) &  Type \uppercase\expandafter{\romannumeral1} (\%)  & Type \uppercase\expandafter{\romannumeral2} (\%) \\
			\hline
			{Lasso} &3.15(0.27); 3.21(0.26)&2.24(0.27); 2.30(0.26)&11.3(1.95); 11.4(2.05)\\
			{Dantzig} &3.82(0.27); 3.60(0.28)&2.71(0.30); 2.68(0.30)&13.8(2.14); 11.8(2.03)\\
			{MU}&3.96(0.30); 4.06(0.35)&2.82(0.35); 3.09(0.38)&13.9(2.03); 12.3(1.71)\\
			\hline 
		\end{tabular}
	\end{subtable}
	\label{tab: rd0.1_k200}
\end{table}

\begin{table}[H]
	\caption{The random precision matrix model with  $\alpha = 0.5$ and $K = 30$}
	\begin{subtable}{1\textwidth}
		\sisetup{table-format=-1.2}   
		\centering
		\caption{$n = 100$}
		\begin{tabular}{c|c c c }
			\hline
			Ave (SE) 
			& Total (\%) &  Type \uppercase\expandafter{\romannumeral1} (\%)  & Type \uppercase\expandafter{\romannumeral2} (\%) \\
			\hline
			{Lasso}&42.0(2.86); 42.0(3.04)&11.8(3.44); 11.9(3.04)&72.5(5.76); 72.4(5.37)\\
			{Dantzig} &43.7(2.65); 43.3(2.75)&12.5(3.50); 13.1(3.04)&75.2(4.90); 73.8(4.77)
			\\
			{MU} &43.6(2.76); 43.4(2.84)&12.1(3.30); 13.0(3.42)&75.3(4.77); 74.0(4.79)
			\\
			\hline 
		\end{tabular}
	\end{subtable}
	
	\bigskip
	
	\begin{subtable}{1\textwidth}
		\sisetup{table-format=-1.2}   
		\centering
		\caption{$n = 500$}
		\begin{tabular}{c|c c c }
			\hline
			Ave (SE) 
			& Total (\%) &  Type \uppercase\expandafter{\romannumeral1} (\%)  & Type \uppercase\expandafter{\romannumeral2} (\%) \\
			\hline
			{Lasso}&16.7(3.46); 16.7(3.38)&11.4(3.46); 11.1(3.59)&22.0(6.49); 22.5(6.01)
			\\
			{Dantzig} &18.7(3.47); 17.8(3.41)&13.2(3.67); 12.6(3.73)&24.4(6.18); 23.2(6.07)
			\\
			{MU} &22.8(3.77); 21.5(3.64)&15.7(3.80); 15.4(3.50)&29.9(7.03); 27.6(6.60)
			\\
			\hline 
		\end{tabular}
	\end{subtable}
	
	\bigskip
	
	\begin{subtable}{1\textwidth}
		\sisetup{table-format=-1.2}   
		\centering
		\caption{$n = 1000$}
		\begin{tabular}{c|c c c }
			\hline
			Ave (SE) 
			& Total (\%) &  Type \uppercase\expandafter{\romannumeral1} (\%)  & Type \uppercase\expandafter{\romannumeral2} (\%) \\
			\hline
			{Lasso}&6.27(1.76); 6.42(1.69)&8.20(2.84); 8.68(2.96)&4.62(2.32); 4.51(2.48)\\
			{Dantzig} &7.89(1.91); 7.41(1.99)&10.5(3.12); 10.1(3.22)& 5.77(3.00); 5.22(2.48)\\
			{MU} &12.3(2.71); 12.0(3.00)&16.7(4.75); 16.5(5.40)&8.61(3.38); 8.02(2.90)
			\\
			\hline 
		\end{tabular}
	\end{subtable}
	\label{tab: rd0.5_k30}
\end{table}
\vspace{6mm}
\begin{table}[H]
	\caption{The random precision matrix model with $\alpha = 0.5$ and $K = 200$}
	\begin{subtable}{1\textwidth}
		\sisetup{table-format=-1.2}   
		\centering
		\caption{$n = 100$}
		\begin{tabular}{c|c c c }
			\hline
			Ave (SE) 
			& Total (\%) &  Type \uppercase\expandafter{\romannumeral1} (\%)  & Type \uppercase\expandafter{\romannumeral2} (\%) \\
			\hline
			{Lasso} &49.6(0.41); 49.7(0.42)&4.53(0.54); 3.43(0.49)&94.8(0.62); 96.0(0.55)\\
			{Dantzig} &49.7(0.43); 49.7(0.41)&4.19(0.51); 3.53(0.47)&95.3(0.58); 95.9(0.54)\\
			{MU}&49.7(0.42); 49.7(0.41)&4.17(0.44); 3.48(0.44)&95.3(0.54); 96.0(0.50)\\
			\hline 
		\end{tabular}
	\end{subtable}
	
	\bigskip
	
	\begin{subtable}{1\textwidth}
		\sisetup{table-format=-1.2}   
		\centering
		\caption{$n = 500$}
		\begin{tabular}{c|c c c }
			\hline
			Ave (SE) 
			& Total (\%) &  Type \uppercase\expandafter{\romannumeral1} (\%)  & Type \uppercase\expandafter{\romannumeral2} (\%) \\
			\hline
			{Lasso} &47.9(0.45); 47.9(0.46)&8.96(0.76); 9.10(0.77)&86.9(0.95); 86.7(0.99)\\
			{Dantzig} &48.8(0.45); 48.4(0.42)&8.63(0.78); 9.41(0.76)&88.9(0.93); 87.4(0.95)\\
			{MU}&48.8(0.44); 48.4(0.43)&8.57(0.71); 9.31(0.79)&89.0(0.82); 87.5(0.86)\\
			\hline 
		\end{tabular}
	\end{subtable}
	
	\bigskip
	
	\begin{subtable}{1\textwidth}
		\sisetup{table-format=-1.2}   
		\centering
		\caption{$n = 1000$}
		\begin{tabular}{c|c c c }
			\hline
			Ave (SE) 
			& Total (\%) &  Type \uppercase\expandafter{\romannumeral1} (\%)  & Type \uppercase\expandafter{\romannumeral2} (\%) \\
			\hline
			{Lasso} &44.6(0.49); 44.6(0.48)&12.4(1.11); 12.6(1.06)&76.8(1.65); 76.7(1.60)\\
			{Dantzig} &46.8(0.48); 46.0(0.48)&13.0(1.15); 13.7(1.18)&80.6(1.48); 78.3(1.57)\\
			{MU}&47.4(0.46); 46.5(0.47)&12.6(1.05); 13.1(0.98)&82.1(1.11); 80.0(1.17)\\
			\hline 
		\end{tabular}
	\end{subtable}
	\label{tab: rd0.5_k200}
\end{table}
\end{spacing}
\end{document}